\documentclass[a4,11pt,fleqn]{article}

\usepackage{amsmath,amssymb,graphicx,amsthm,amsfonts}
\usepackage{amscd}
\usepackage{slashed}
\usepackage{enumitem}
\usepackage[all]{xy}

\usepackage[top=30truemm,bottom=30truemm,left=25truemm,right=25truemm]{geometry}

\theoremstyle{definition}
\newtheorem{Def}{Definition}[section]
\newtheorem{Thm}[Def]{Theorem}
\newtheorem{Prp}[Def]{Proposition}
\newtheorem{Lem}[Def]{Lemma}
\newtheorem{Cor}[Def]{Corollary}

\newtheorem{Remark}[Def]{Remark}

\def\i<#1>{\langle #1 \rangle}

\newcommand{\Par}		{ {\mathcal{P}ar} }

\newcommand{\mb}[1]	{ \boldsymbol{#1} }

\newcommand{\comp}	{	\circ }
\newcommand{\eps}	{ \varepsilon }

\newcommand{\dirac}	{ \slashed{\partial} }
\newcommand{\Dirac}	{ \slashed{D} }
\newcommand{\deebar} { \overline{\partial} }

\title{The Nahm transform of spatially periodic instantons}
\author{Masaki Yoshino}
\date{\today}

\begin{document}
	\maketitle
	\begin{abstract}
  	We construct the Nahm transform from finite energy instantons on the product of a real line $\mathbb{R}$ and a three dimensional torus $T^3$
    to Dirac-type singular monopoles on the dual torus $\hat{T}^3$.
    Moreover, we show the correspondence between the data which handle the asymptotic behavior of 
    instantons at infinity and one of monopoles at singular points.
  \end{abstract}
  \section{Introduction}

	Set $T^3 := \mathbb{R}^3/\Lambda_3$, where $\Lambda_3 \subset \mathbb{R}^3$ is a lattice of $\mathbb{R}^3$.
	Let $(V,h)$ be a Hermitian vector bundle on $\mathbb{R}\times T^3$
  	and $A$ be a connection on $(V,h)$.
  	The triple $(V,h,A)$ is called an instanton on $\mathbb{R}\times T^3$ if its curvature $F(A)$ satisfies the ASD equation $F(A) = -\ast F(A)$.
  Additionally, an instanton $(V,h,A)$ on $\mathbb{R}\times T^3$ is $L^2$-finite
  if it satisfies the finite energy condition $|F(A)| \in L^2(\mathbb{R} \times T^3)$.
  Let $\hat{T}^3$ be the dual torus of $T^3$
  \textit{i.e.} $\hat{T}^3 = \mathrm{Hom}(\mathbb{R}^3,\mathbb{R})/ \Lambda_3^{\ast}$,
  where $\Lambda_3^{\ast} = \{\xi \in \mathrm{Hom}(\mathbb{R}^3,\mathbb{R}) \mid \xi(\Lambda_3)\subset\mathbb{Z}\}$
  is the dual subgroup of $\Lambda_3$.
  Let $Z \subset \hat{T}^3$ be a finite subset.
  Let $(\hat{V},\hat{h},\hat{A})$ be a Hermitian vector bundle with a connection on $\hat{T}^3 \setminus Z$.
  Let $\hat{\Phi}$ be a skew-Hermitian section of $\mathrm{End}(V)$ on $\hat{T}^3 \setminus Z$.
  The tuple $(\hat{V},\hat{h},\hat{A},\hat{\Phi})$ is said to be a monopole on $\hat{T}^3 \setminus Z$
  if it satisfies the Bogomolny equation $F(\hat{A}) = \ast\nabla_{\hat{A}}(\hat{\Phi})$.
  Moreover, $Z$ is the Dirac-type singularities of $(\hat{V},\hat{h},\hat{A},\hat{\Phi})$ if
  it has a certain type of the asymptotic behavior around each point of $Z$ (Definition \ref{Def:def of monopoles}).
  Then we call $(\hat{V},\hat{h},\hat{A},\hat{\Phi})$ a Dirac-type singular monopole on $\hat{T}^3$.
  In this paper,
  we will construct the Nahm transform of an $L^2$-finite instanton $(V,h,A)$ on the product of a real line $\mathbb{R}\times T^3$
  to a Dirac-type singular monopole $(\hat{V},\hat{h},\hat{A},\hat{\Phi})$ on $\hat{T}^3$.
  
  In general,
  for any closed subgroup $\Lambda \subset \mathbb{R}^4$ and the dual subgroup $\Lambda^{\ast} \subset \mathrm{Hom}(\mathbb{R}^4,\mathbb{R})$,
  it is believed that there exists a way to construct 
  $\Lambda^{\ast}$-invariant instantons on $\mathrm{Hom}(\mathbb{R}^4,\mathbb{R})$
  from $\Lambda$-invariant instantons on $\mathbb{R}^4$.
  For example, if $\Lambda = \mathbb{R}^4$ and $\Lambda^{\ast} = \{0\}$,
  it was constructed by Atiyah, Drinfeld, Hitchin and Manin \cite{Ref:ADHM}, and called the ADHM construction.
  The case $(\Lambda,\Lambda^{\ast})\simeq (\mathbb{R},\mathbb{R}^3)$
  was studied by Nahm \cite{Ref:Nah}, Hitchin \cite{Ref:Hit} and Nakajima \cite{Ref:Nak}.
  See Jardim \cite{Ref:Jar} for a list of many Nahm transformations.
  
  Since $\mathbb{R}$-invariant instantons on $\mathbb{R}\times\hat{T}^3$ can be regarded as monopoles on $\hat{T}^3$,
  the construction in this paper corresponds to the case $(\Lambda,\Lambda^{\ast}) \simeq (\mathbb{Z}^3, \mathbb{R}\times \mathbb{Z}^3)$.
  This case was previously studied by Charbonneau \cite{Ref:Cha2}
  who constructed singular monopoles on $\hat{T}^3$,
  from $L^2$-finite instantons of rank $2$ on $\mathbb{R}\times T^3$ whose curvature have exponentially decay.
  We will generalize this result to general $L^2$-finite instantons
  and prove that any singularities of the transformed monopole are Dirac-type.
  
  Next, let us consider relations between the Nahm transforms and the Kobayashi-Hitchin correspondences.
  On a connected compact K\"{a}hler surface $(M,g)$ with the K\"{a}hler form $\omega$,
  there exists a one-to-one correspondence between irreducible instantons and stable holomorphic vector bundles $V$ with
  the condition $\left(c_1(V)\cup\omega\right)/[X] = 0$ up to their gauge transformations,
  which is called the Kobayashi-Hitchin correspondence and proved by Uhlenbeck and Yau \cite{Ref:Uhl-Yau}.
  In our case, there exist similar relations under the assumption that
  $T^3$ is isomorphic to $S^1 \times T^2$ as a Riemannian manifold.
  On one hand, in \cite{Ref:Cha-Hur} Charbonneau and Hurtubise obtained the Kobayashi-Hitchin correspondence between
  Dirac-type singular monopoles on $\hat{T}^3$ 
  and polystable singular mini-holomorphic bundles (Definition \ref{Def:min-hol. bdle}) on $\hat{T}^3$.
  On the other hand, we will give a construction of polystable parabolic bundles with parabolic degree $0$  on $(\mathbb{P}^1\times T^2,\{0,\infty\}\times T^2)$ 
  from $L^2$-finite instantons on $\mathbb{R}\times T^3$\;(Theorem \ref{Thm:inst is stable}).
  However, it is only a half part of the Kobayashi-Hitchin correspondence
  and we have not yet proved the other part.
  
  Next we will construct mini-holomorphic bundles on $\hat{T}^3$ from stable parabolic bundles on $(\mathbb{P}^1\times T^2,\{0,\infty\}\times T^2)$ of rank $r>1$
  in reference to \cite{Ref:Moc1}.
  We call this construction the \textit{algebraic Nahm transform} as in \cite{Ref:Moc1} and it satisfies the following commutative diagram.
  
  
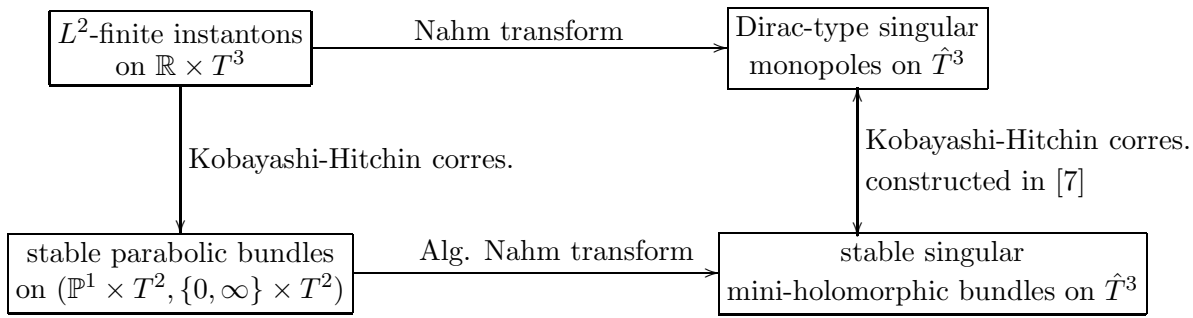
\begin{figure}[h]
	  \begin{xy}
  		(0,30) 		*+[F]\txt{$L^2$-finite instantons \\on $\mathbb{R}\times T^3$}																= "A",
    	(90,30)	*+[F]\txt{Dirac-type singular \\monopoles on $\hat{T}^3$}																			= "B",
    	(0,0)			*+[F]\txt{stable parabolic bundles \\on $(\mathbb{P}^1\times T^2,\{0,\infty\}\times T^2)$}		= "C",
    	(100,0)		*+[F]\txt{stable singular \\ mini-holomorphic bundles on $\hat{T}^3$}												= "D",
    	\ar @{->}		"A";"B"^{\mbox{Nahm transform}}
    	\ar @{->}		"A";"C"^{\mbox{Kobayashi-Hitchin corres.}}
    	\ar @{<->}	"B";(90,5)^(0.5){\mbox{Kobayashi-Hitchin corres.}}^(0.72){\mbox{constructed in \cite{Ref:Cha-Hur}}}
    	\ar @{->}		"C";"D"^{\mbox{Alg. Nahm transform}}
  	\end{xy}
  \end{figure}
  
  \subsection{Main result}
  	The main results of this paper are summarized as follow.
  	\begin{enumerate}[label=(\Roman*)]
  		\item\label{Enum:Asymp behavior}
      		For any $L^2$-finite instanton $(V,h,A)$ on $\mathbb{R}\times T^3$,
        	there exist model solutions of the Nahm equation $(\Gamma_{\pm}, N_{\pm}) = (\Gamma_{i,\pm}, N_{i,\pm})_{i=1,2,3}$
        	such that $(V,h,A)$ is approximated by the $T^3$-invariant instantons associated to $(\Gamma_{\pm}, N_{\pm})$
        	at $t \to \pm\infty$ (Corollary \ref{Cor:Asymp of Inst}).
    	\item\label{Enum:Construct monopole}
      		We construct a monopole $(\hat{V},\hat{h},\hat{A},\hat{\Phi})$
        	on $\hat{T}^3 \setminus \mathrm{Sing}(V,h,A)$ from an $L^2$-finite instanton $(V,h,A)$ on $\mathbb{R}\times T^3$,
        	where $\mathrm{Sing}(V,h,A) \subset \hat{T}^3$ is a finite subset determined by $(\Gamma_{\pm})$ (Proposition \ref{Prp:transformed bundle is monopole}).
        	Moreover, each point of $\mathrm{Sing}(V,h,A)$
        	is a Dirac-type singularity of $(\hat{V},\hat{h},\hat{A},\hat{\Phi})$ (Theorem \ref{Thm:Transformed monopole is Dirac type}).
    	\item\label{Enum:Correspodence}
      		Assume that $T^3$ is isomorphic to $S^1 \times T^2$ as a Riemannian manifold.
      		If $(V,h,A)$ is irreducible and $\mathrm{rank}(V)>1$,
        	then the weight $\vec{k} \in \mathbb{Z}^{\mathrm{rank}(\hat{V})}$ of
        	$(\hat{V},\hat{h},\hat{A},\hat{\Phi})$ at each singular point $\xi \in \mathrm{Sing}(V,h,A)$
        	is determined by the weights of the $\mathfrak{su}(2)$ representation $\rho_{\pm,\xi}$ constructed from $(N_{\pm})$ (Theorem \ref{Thm:Correspodence between weights}).
  	\end{enumerate}
    The first result is an analytical preparation of the Nahm transform in \ref{Enum:Construct monopole}.
    The second is the construction of the Nahm transform.
    The third is an application of the commutativity in the above figure (Theorem \ref{Thm:Algebraic Nahm trans}).
    Let us explain more details in the following.
    \subsubsection{Main result \ref{Enum:Asymp behavior}}
    	Let $h_{\mathbb{C}^r}$ be the canonical Hermitian metric on $\mathbb{C}^r$.
    	For a smooth manifold $S$,
    	we denote by $(\underline{\mathbb{C}^r}_S,\underline{h}_S)$ the product bundle of $(\mathbb{C}^r,h_{\mathbb{C}^r})$ on $S$.
    	If there is no risk of confusion,
    	then we abbreviate $(\underline{\mathbb{C}^r}_S,\underline{h}_S)$ to $(\underline{\mathbb{C}^r},\underline{h})$.
    	
    	Let $A$ be a connection on $(\underline{\mathbb{C}^r},\underline{h})$ on $(0,\infty)\times T^3$,
    	and assume that the connection form $\alpha dt + \sigma_i A_i dx^i$ of $A$ is invariant under $T^3$-action \textit{i.e.} $\alpha$ and $A_i$ are $T^3$-invariant functions on $(0,\infty)\times T^3$.
    	Then the ASD equation for $(\underline{\mathbb{C}^r},\underline{h},A)$ is equivalent to the following Nahm equation:
      \begin{eqnarray}\label{Eqn:Intro Nahm eq}
    		\left\{\begin{array}{l}
      		\displaystyle\frac{\partial A_1}{\partial t} + [\alpha, A_1] = -[A_2, A_3] \\
        	\\
        	\displaystyle\frac{\partial A_2}{\partial t} + [\alpha, A_2] = -[A_3, A_1] \\
        	\\
        	\displaystyle\frac{\partial A_3}{\partial t} + [\alpha, A_3] = -[A_1, A_2]. \\
				\end{array}\right.
    	\end{eqnarray}
      
      Let $\Gamma_i \in \mathfrak{u}(r)\;(i=1,2,3)$ be skew-Hermitian matrices which are commutative each other.
      For the tuple $\Gamma=(\Gamma_i)$,
      we set $\mathrm{Center}(\Gamma) := \{a \in \mathfrak{u}(r)\;|\;[\Gamma_i,a]=0\;(i=1,2,3)\}$.
      Take $N_i \in \mathrm{Center}(\Gamma)\;(i=1,2,3)$ satisfying 
      $N_i = [N_j,N_k]$
      for any even permutation $(ijk)$ of $(123)$.
      Then the tuple
      $\alpha=0, A_i = \Gamma_i + N_{i}/t$ forms a solution of (\ref{Eqn:Intro Nahm eq}) on $(0,\infty)$.
    	\begin{Def}[Definition \ref{Def:model sol of Nahm eq}]
      	A tuple $(\Gamma=(\Gamma_i),N=(N_i))$ is called a model solution of the Nahm equation if it satisfies the above conditions.
    	\end{Def}
      We obtain the following theorem as a consequence of results in \cite{Ref:Biq} and \cite{Ref:Mor-Mro-Rub}.
      \begin{Thm}[Corollary \ref{Cor:Asymp of Inst}]\label{Thm: thm1 in Intro}
      	Let $(V,h,A)$ be an $L^2$-finite instanton on $(0,\infty)\times T^3$,
        \textit{i.e.} its curvature $F(A)$ is $L^2$.
      	
        If we fix a positive number $0<\lambda<1$ and
        take a sufficiently large $R>0$, then there exist a trivialization of $C^{4,\lambda}$-class
      	$\sigma:(V,h)|_{(R,\infty)\times T^3} \simeq (\underline{\mathbb{C}^r} , \underline{h})$
      	and a model solution of the Nahm equation $(\Gamma,N)$ such that the following holds for
        the connection form $\alpha dt + \sum_i A_i dx^{i}$ of $A$ with respect to $\sigma$.
    		\begin{enumerate}
      		\item
        		The trivialization $\sigma$ is a temporal gauge \textit{i.e.} we have $\alpha = 0$.
        	\item
        		There exists a decomposition 
          	$A_i - (\Gamma_i + N_i /t) = \eps_{1,i}(t) + \eps_{2,i}(t) + \eps_{3,i}(t,x)$
          	such that we have $\eps_{1,i}(t) \in \mathrm{Center}(\Gamma)$, $\eps_{2,i}(t) \in (\mathrm{Center}(\Gamma))^{\perp}$
          	and the following estimates for any $1\leq i,j\leq 3$,
            where $(\mathrm{Center}(\Gamma))^{\perp}$ means the orthogonal complement of
            $\mathrm{Center}(\Gamma)$ in $\mathfrak{u}(r)$ with respect to the inner product $\i<A,B> := -\mathrm{tr}(AB)$.
          	\begin{eqnarray*}
      				\begin{array}{ll}
            		|\partial^{j}_{t} \eps_{1,i}| &= O(t^{-(1+j+\delta)})\\
            		|\partial^{j}_{t} \eps_{2,i}| &= O(\exp(- \delta t))\\
              	||\eps_{3,i}||_{C^{3,\lambda}([t,t+1]\times T^3)} &= O(\exp(- \delta t)),
            	\end{array}
          	\end{eqnarray*}
          	where $\delta$ is a positive number.
        \end{enumerate}
      \end{Thm}
      \begin{Remark}
      	Since the tuple $\alpha=0, A_i = \Gamma_i + N_{i}/t$ also form a solution of the Nahm equation on $(-\infty,0)$,
      	a similar result holds for $L^2$-finite instantons on $(-\infty,0)\times T^3$.
      \end{Remark}
    \subsubsection{Main result \ref{Enum:Construct monopole}}
    	Let $(V,h,A)$ be an $L^2$-finite instanton on $\mathbb{R}\times T^3$ of rank $r$.
      Applying Theorem \ref{Thm: thm1 in Intro} to $(V,h,A)|_{(-\infty,0)\times T^3}$ and $(V,h,A)|_{(0,\infty)\times T^3}$,
      we obtain model solutions $(\Gamma_{\pm},N_{\pm})$ which approximate $(V,h,A)$ at $t \to \pm\infty$.
      Since the simultaneous eigenvalues of $\sum_i\Gamma_{\pm,i}dx^i \in \Omega^1(T^3)$ are
      $T^3$-invariant pure imaginary $1$-forms on $T^3$,
      they can be regarded as elements of $\mathrm{Hom}(\mathbb{R}^3,\sqrt{-1}\mathbb{R})$.
      Thus we take $\widetilde{\mathrm{Spec}(\Gamma_{\pm})} \subset \mathrm{Hom}(\mathbb{R}^3,\mathbb{R})$
      as the set of $(2\pi\sqrt{-1})^{-1}$ times simultaneous eigenvalues of $\sum_i\Gamma_{\pm,i}dx^i$.
      We define the spectrum set $\mathrm{Spec}(\Gamma_{\pm}) \subset \hat{T}^3$ to be the image of $\widetilde{\mathrm{Spec}(\Gamma_{\pm})}$
      by the quotient map $\mathrm{Hom}(\mathbb{R}^3,\mathbb{R}) \to \hat{T}^3$.
			We define the singularity set of $(V,h,A)$ as
      $\mathrm{Sing}(V,h,A) := \mathrm{Spec}(\Gamma_{+}) \cup \mathrm{Spec}(\Gamma_{-})$.
      For $\xi \in \hat{T}^3$ and the associated flat Hermitian line bundle
      $L_{\xi} := (\underline{\mathbb{C}},\underline{h},d+2\pi\sqrt{-1}\i<\xi,x>)$ on $\mathbb{R}\times T^3$,
      we set the twisted instanton $(V,h,A_{\xi}):=(V,h,A)\otimes L_{-\xi}$.
      Then we have $\mathrm{Sing}(V,h,A_{\xi}) = \mathrm{Sing}(V,h,A) - \xi = \{\mu-\xi\in\hat{T}^3\mid \mu\in\mathrm{Sing}(V,h,A)\}$.
      
      We construct the Nahm transform of $(V,h,A)$ as follows.
      Let $S^{\pm}$ be the spinor bundle on $\mathbb{R}\times T^3$ with respect to the trivial spin structure and
      $\dirac_{A}^{\pm}:S^{\pm}\otimes V \to S^{\mp}\otimes V$ be the Dirac operator of the connection $A$.
      Let $\mathcal{V}$ be a Hermitian flat vector bundle on $\hat{T}^3$ which is the quotient of the product bundle 
    	$(\underline{L^2(\mathbb{R}\times T^3, V\otimes S^{-})}, ||\cdot||_{L^2})$ on $\mathrm{Hom}(\mathbb{R}^3,\mathbb{R})$
    	by a $\Lambda^{\ast}_3$-action $v\cdot(\xi, f) := (\xi +v, \exp(2\pi\sqrt{-1}\i<x,v>)f)$.
      For a family of $L^2$-finite instantons $\{(V, h, A_{\xi})\}_{\xi \in \hat{T}^3}$,
      we set $(\hat{V},\hat{h})$ as a subbundle of $\mathcal{V}|_{\hat{T}^3 \setminus \mathrm{Sing}(V,h,A)}$
    	defined by $\hat{V}_{\xi} := \mathrm{Ker}\left( \dirac^{-}_{A_{\xi}} \right) \cap L^2$.
      Then $(\hat{V},\hat{h})$ is well-defined and of finite rank
      because $\dirac^{-}_{A_{\xi}}$ is a continuous family of surjective Fredholm operators
      for $\xi \in \hat{T}^3 \setminus \mathrm{Sing}(V,h,A)$ (see Theorem \ref{Thm:Dirac op is Fredholm} and Remark \ref{Rem:surj of neg Dirac op}).
      By Theorem \ref{Thm:Index of dirac op}, we have $\mathrm{rank}(\hat{V}) = (8\pi^2)^{-1}||F(A)||_{L^2}^2$.
    	
      We define a connection $\hat{A}$ on $\hat{V}$ to be the induced connection from the flat connection $d_{\mathcal{V}}$ of $\mathcal{V}$,
      \textit{i.e.} we can write $\hat{A} = Pd_{\mathcal{V}}$ for the orthogonal projection $P:\mathcal{V}|_{\hat{T}^3 \setminus \mathrm{Sing}(V,h,A)} \to \hat{V}$.
      We take a skew-Hermitian endomorphism $\hat{\Phi}$ as $\hat{\Phi}(f) := P(2\pi\sqrt{-1} tf)$.
      Then we have the next theorem.
      \begin{Thm}[Proposition \ref{Prp:transformed bundle is monopole} and Theorem \ref{Thm:Transformed monopole is Dirac type}]\label{Thm:thm2 in intro}
      	$(\hat{V}, \hat{h}, \hat{A}, \hat{\Phi})$ is a monopole on $\hat{T}^3 \setminus \mathrm{Sing}(V,h,A)$,
        and each point of $\mathrm{Sing}(V,h,A)$ is a Dirac-type singularity of $(\hat{V}, \hat{h}, \hat{A}, \hat{\Phi})$.
      \end{Thm}
      Here we recall the definition of Dirac-type singularity of monopole by following \cite{Ref:Cha-Hur}.
      Let $(X,g)$ be an oriented Riemannian $3$-fold and $Z \subset X$ be a discrete subset.
      Let $(V,h,A,\Phi)$ be a monopole on $X \setminus Z$ of rank $r$.
      Each point $p \in Z$ is a Dirac-type singularity of $(V,h,A,\Phi)$ with weight $\vec{k}= (k_1,\ldots,k_r) \in \mathbb{Z}^r$
      if the following conditions are satisfied.
      \begin{itemize}
      	\item
        	There exists a neighborhood $B \subset X$ of $p$ such that
          $(V, h)|_{B \setminus \{p\}}$ is decomposed into a direct sum of Hermitian line bundles
          $\bigoplus_{i=1}^{r} L_i$\;\;($\mathrm{deg}(L_i) =\int_{\partial B}c_1(L_i) = k_i$).
        \item
        	Under the above decomposition,
          we have the next estimates.
          \begin{eqnarray*}
          	\left\{
            	\begin{array}{l}
		          	\Phi  = \displaystyle\frac{\sqrt{-1}}{2R}\sum_{i=1}^{r} k_i\cdot \mathrm{Id}_{L_i} + O(1)\\
		            \nabla_A(R\Phi) = O(1),
              \end{array}
            \right.
          \end{eqnarray*}
          where $R$ is the distance from $p$.
      \end{itemize}
    \subsubsection{Main result \ref{Enum:Correspodence}}
    	Since $N_{\pm}=(N_{\pm,i})$ satisfies $N_{\pm,i} = [N_{\pm,j},N_{\pm,k}]$ for any even permutation $(ijk)$ of $(123)$,
      we can construct $\mathfrak{su}(2)$ representation $\rho_{\pm}$ from $N_{\pm}$.
      Then, $\rho_{\pm}$ can be decomposed into $\rho_{\pm} =\bigoplus_{\xi \in \mathrm{Sing}(V,h,A)} \rho_{\pm,\xi}$
      because $N_{\pm,i} \in \mathrm{Center}(\Gamma_{\pm})$.
      Now we define the weight of $\rho_{\pm,\xi}$ to be $w_{\pm,\xi} :=(\mathrm{rank}(\rho_{\pm,\xi,i})) \in \mathbb{Z}^{m_{\pm,\xi}}$,
      where $\rho_{\pm,\xi} = \bigoplus_{i=1}^{m_{\pm,\xi}} \rho_{\pm,\xi,i}$ is the irreducible decomposition.
      Let $\vec{k}_{\xi}$ be the weight of the monopole $(\hat{V},\hat{h},\hat{A},\hat{\Phi})$ at $\xi \in \mathrm{Sing}(V,h,A)$
      and $\vec{k}_{+,\xi},\vec{k}_{-,\xi}$ the positive and negative part of $\vec{k}_{\xi}$.
      Here main result \ref{Enum:Correspodence} can be described as follows.
      \begin{Thm}[Theorem \ref{Thm:Correspodence between weights}]\label{Thm:thm3 in Intro}
      	Assume $T^3$ is isomorphic to $S^1\times T^2$ as a Riemannian manifold.
        If the $L^2$-finite instanton $(V,h,A)$ is irreducible and  of $\mathrm{rank}(V) >1$,
    	then $\vec{k}_{\pm}$ agrees with $\pm w_{\pm,\xi}$ with a suitable permutation.
  	  \end{Thm}
  \subsection*{Remark}
	When the first version of this paper was almost accomplished,
	the author found the paper of Charbonneau and Hurtubise \cite{Ref:Cha-Hur2},
	where they seem to study that the Nahm transform induces
	the bijection of the equivalence classes of 
	spatially periodic instantons and singular monopoles on the dual torus,
	under the genericity assumption as in \cite{Ref:Cha1,Ref:Cha2,Ref:Cha-Hur}.
	In our Theorem \ref{Thm:thm2 in intro}, we study the Nahm transform 
	$(\hat{V},\hat{h},\hat{A},\hat{\Phi})$ of any $L^2$-finite
	spatially periodic instantons $(V,h,A)$ without the genericity
	assumption
	as a refinement of the construction in \cite{Ref:Cha1,Ref:Cha2},
	and we prove that the singularities of the monopoles
	$(\hat{V},\hat{h},\hat{A},\hat{\Phi})$ 
	are of Dirac type,
	in a more direct way using a result in \cite{Ref:Moc-Yos}.
	We also study the comparison of the weights of 
	the singularities of the instantons and the monopoles (Theorem \ref{Thm:thm3 in Intro})
	in this generalized context.
	\subsection*{Acknowledgment}
		I express my sincere thanks to my supervisor Takuro Mochizuki for his support
		and patience. His deep insight and attentive guidance have greatly
		helped me to accomplish this article.
	\section{Preliminary}\label{Sec:Preliminary}
	For the product bundle $(\underline{\mathbb{C}^r}_S,\underline{h}_S)$ on a smooth manifold $S$,
	we denote by $d_S$ the trivial connection on $(\underline{\mathbb{C}^r}_S,\underline{h}_S)$.
	If there is no risk of confusion,
	then we abbreviate $d_S$ to $d$.
	
	\subsection{Tori and dual tori}\label{SubSec:Tori}
	For a finite dimensional $\mathbb{R}$-vector space $X$
	and a lattice $\Lambda \subset X$,
  	we set $T = X/\Lambda$.
    Let $X^{\ast}$ be the dual space of $X$.
    Let $\Lambda^{\ast}$ denote the dual lattice of  $\Lambda$,
		\textit{i.e.}, $\Lambda^{\ast}:=\{v\in X^{\ast}\,|\,v(\Lambda)\subset Z\}$.
    We define the dual torus $\hat{T}$ of $T$ by $\hat{T} := X^{\ast} / \Lambda^{\ast}$.
    
    For any $\xi \in \hat{T}$, we define a flat Hermitian line bundle $L_{\xi}$ on $T$ as
    $L_{\xi} := ( \underline{\mathbb{C}}_{T},\underline{h}_{T} , d_{T} + 2\pi\sqrt{-1}\i<\xi, dx>)$.
    By this correspondence, we can naturally regard $\hat{T}$ as the moduli space of flat Hermitian line bundles on $T$.
    The double dual of $T$ is naturally isomorphic to $T$,
    and hence $x \in T$ also gives a flat Hermitian line bundle 
    $L_{x} := ( \underline{\mathbb{C}}_{\hat{T}},\underline{h}_{\hat{T}} , d_{\hat{T}} + 2\pi\sqrt{-1}\i<x, d\xi>)$ on $\hat{T}$.
    
    We recall a differential-geometric construction of the Poincar\'{e} bundle on $T\times\hat{T}$ in \cite{Ref:Dna-Kro}.
    On $T \times X^{\ast}$, we have the following Hermitian line bundle with a connection on $T\times \hat{T}$
    \[
    	\tilde{\mathcal{L}} = \left( \underline{\mathbb{C}},\underline{h}, d - 2\pi\sqrt{-1}\i<\xi, dx> \right).
    \]
    The $\Lambda^{\ast}$-action on $T \times X^{\ast}$ is naturally lifted to the action on $\tilde{\mathcal{L}}$ given by
    \[
    	v\cdot(x, \xi, s) = \left(x, \xi+v, \exp(2\pi\sqrt{-1}\i< x,v >)s\right).
    \]
    The induced Hermitian line bundle with a connection is called the Poincar\'{e} bundle,
    and denoted by $\mathcal{L}$.
    \begin{Lem}[Lemma 3.2.14 in \cite{Ref:Dna-Kro}]
    	The Poincar\'{e} bundle $\mathcal{L}$ has the following properties.
    	\begin{itemize}
      	\item
		    	For any $\xi \in \hat{T}$, $\mathcal{L}|_{T \times \{\xi\}}$ is isomorphic to $L_{-\xi}$.
        \item
		      For any $x \in T$, $\mathcal{L}|_{\{x\} \times \hat{T}}$ is isomorphic to $L_{x}$.
      \end{itemize}
    \end{Lem}
    \begin{proof}
    	The first claim is clear by the construction of $\mathcal{L}$.
      For $x \in T$,
      the connection form of $\tilde{\mathcal{L}}$ on the slice $\{x\}\times X^{\ast}$ vanishes,
      and $\mathcal{L}|_{\{x\} \times \hat{T}}$ has the global section induced by the function $s(\xi) := \exp(2\pi\sqrt{-1}\i< x,\xi >)$ on $\{x\}\times X^{\ast}$,
      which satisfies $ds = s(2\pi\sqrt{-1}\i< x,d\xi >)$.
      Hence, $\mathcal{L}|_{\{x\} \times \hat{T}}$ is isomorphic to $L_{x}$.
    \end{proof}
    \begin{Remark}
    	If $X$ is a complex vector space,
    	then $T$ and $\hat{T}$ are equipped with the induced complex structures and $\mathcal{L}$ becomes a holomorphic line bundle on $T\times \hat{T}$ by the holomorphic structure induced by the $(0,1)$-part of the connection on $\mathcal{L}$.
    \end{Remark}
    In this paper,
    we fix a lattice $\Lambda_3 \subset \mathbb{R}^3$.
    Thus $T^3 = \mathbb{R}^3/\Lambda_3$, $\Lambda^{\ast}_3 \subset (\mathbb{R}^3)^{\ast}$ and $\hat{T}^3 = \mathrm{Hom}(\mathbb{R}^3,\mathbb{R})/\Lambda^{\ast}_3$
    are also fixed.
  \subsection{$L^2$-finite instantons}\label{SubSec:Def L^2-fin Inst}
  	Let $(X,g)$ be a connected oriented Riemannian 4-fold and $\ast$ be the Hodge operator on $X$.
    Let $L^2(X,g)$ denote the space of $L^2$-functions on $X$ with respect to the measure induced by $g$.
  	\begin{Def}
	    Let $(V, h)$ be a Hermitian vector bundle and $A$ be a unitary connection on $(V,h)$.
      \begin{itemize}
      	\item
        	The tuple $(V, h, A)$ is an instanton on $X$
          if the ASD equation $F(A) = -\ast F(A)$ holds.
        \item
        	An instanton $(V, h, A)$ is an $L^2$-finite instanton on $X$
          if we have $|F(A)| \in L^2(X,g)$,
          where the norm is induced by $g$ and $h$.
      \end{itemize}
    \end{Def}
    \subsubsection{Some easy properties of $L^2$-finite instantons on $\mathbb{R}\times T^3$}
      \begin{Lem}\label{Lem:rank1 inst is flat}
    		Let $(L,h,A)$ be an $L^2$-finite instanton on $\mathbb{R} \times T^3$ of rank $1$.
        Then, $(L, h, A)$ is a flat  Hermitian line bundle.
        In particular, $(L,h)$ is a topologically trivial  Hermitian line bundle.
      \end{Lem}
      \begin{proof}
      	Let $t$ and $(x^1,x^2,x^3)$ be the standard coordinates of $\mathbb{R}$ and $\mathbb{R}^3$ respectively.
        By abuse of notation,
        we use $(t,x^1,x^2,x^3)$ to denote a local chart of $\mathbb{R}\times T^3$.
        Using this coordinate, we write $F(A) = \sum_{i} F_{ti} dt\wedge dx^{i} + \sum_{i< j} F_{ij} dx^{i}\wedge dx^{j}$.
        By Bianchi's identity $\nabla_{A}(F(A))=0$ and the ASD equation $F(A) = -\ast F(A)$,
        we have $\Delta(F(A))=0$,
        where $\Delta = \nabla^{\ast}_A\nabla_A + \nabla_A\nabla^{\ast}_A$ is the Laplacian.
        It implies that the functions $F_{ti}$ and $F_{ij}$ are harmonic and $L^2$ on $\mathbb{R}\times T^3$,
        and hence $0$.
        Thus we obtain that $(L,h,A)$ is a flat Hermitian line bundle.
      \end{proof}
    	\begin{Cor}\label{Cor:Topological triviality of inst}
      	Let $(V, h, A)$ be an $L^2$-finite instanton on $\mathbb{R} \times T^3$ of rank $r$.
        Then, $(V,h)$ is a trivial Hermitian vector bundle.
      \end{Cor}
      \begin{proof}
      	By Lemma \ref{Lem:rank1 inst is flat},
        $\left(\mathrm{det}(V), \mathrm{det}(h)\right)$ is a trivial Hermitian line bundle.
        Thus, it suffices to prove that any principal $SU(r)$-bundle $P$ on $T^3$ is topologically trivial.
        
        We may assume $T^3 = (\mathbb{R}/\mathbb{Z})^3$.
        Let $q : \mathbb{R}\to S^1 = \mathbb{R}/\mathbb{Z}$ be the quotient map, and put $T^2=(\mathbb{R}/\mathbb{Z})^2$.
        Take the open intervals $I_1=(0, 1)$, $I_2=(1/2, 3/2)$.
        We set $U_i := q(I_i) \times T^2$,
        and we obtain an open covering $\{U_1,U_2\}$ of $T^3$.
        Then, $P|_{U_i}$ is trivial because $SU(r)$ is simply-connected
        and $U_i$ and $T^2$ are homotopy equivalent.
        Moreover, we can patch each trivialization of $P|_{U_i}$ and get a global one.
        Indeed, any smooth map $f:T^2 \to SU(r)$ is homotopic to a smooth map
        $f_1:T^2\to SU(r)$ such that $f_1(\{0\}\times S^1) \subset \{e\}$ and $f_1(S^1\times\{0\}) \subset \{e\}$
        because $SU(r)$ is simply connected,
        and $f_1$ is homotopic to a constant map because $\pi_2(SU(r))=0$.
        Therefore, $P$ is a trivial $SU(r)$-bundle.
      \end{proof}
	\subsection{Monopoles with Dirac-type singularities}
  	In this subsection, we recall the definition of monopoles with Dirac-type singularities
    by following \cite{Ref:Cha-Hur}.
    \begin{Def}\label{Def:def of monopoles}
    	Let $(X,g)$ be an oriented Riemannian 3-fold and $\ast$ be the Hodge operator on $X$.
      \begin{enumerate}
      	\item
        	Let $(V, h, A)$ be a Hermitian vector bundle with a unitary connection,
          and $\Phi$ be a skew-Hermitian section of $\mathrm{End}(V)$.
          The tuple $(V, h, A, \Phi)$ is called a monopole on $X$ if it satisfies the Bogomolny equation $F(A) = \ast \nabla_A(\Phi)$.
        \item
        	Let $Z \subset X$ be a discrete subset.
          Let $(V, h, A, \Phi)$ be a monopole of rank $r \in \mathbb{N}$ on $X \setminus Z$.
          Each point $p\in Z$ is called a Dirac-type singularity of the monopole $(V, h, A, \Phi)$ with weight $\vec{k}=(k_i) \in \mathbb{Z}^r$
          if the following holds.
          \begin{itemize}
          	\item
            	There exists a small neighborhood $B$ of $p$ such that
              $(V,h)|_{B\setminus\{p\}}$ is decomposed into a sum of Hermitian line bundles $\bigoplus_{i=1}^{r} L_i$
              with $\int_{\partial B} c_1(L_i) = k_i$.
            \item
            	In the above decomposition, we have the following estimates,
              \begin{align*}
              	\left\{
                	\begin{array}{l}
	              		\Phi  = \displaystyle\frac{\sqrt{-1}}{2R}\sum_{i=1}^{r} k_i\cdot Id_{L_i} + O(1)\\
	                	\nabla_A(R\Phi) = O(1),
              		\end{array}
                \right.
              \end{align*}
              where $R$ is the distance from $p$.
          \end{itemize}
      \end{enumerate}
    \end{Def}
    In \cite{Ref:Moc-Yos}, the following proposition is proved.
    \begin{Prp}\label{Prp:Moc's criterion for Dirac sing}
    	Let $U \subset \mathbb{R}^3$ be a neighborhood of $0 \in \mathbb{R}^3$.
      Let $(V,h,A,\Phi)$ be a monopole on $(U \setminus \{0\},g_{\mathbb{R}^3})$.
      Then, the point $0$ is a Dirac-type singularity of $(V,h,A,\Phi)$ if and only if
      $|\Phi(x)| = O(|x|^{-1})\;(x \to 0)$.
    \end{Prp}
    \subsubsection{Monopoles and mini-holomorphic structure}
    	We introduce a complex-geometric interpretation of monopole by following \cite{Ref:Cha-Hur} and \cite{Ref:Moc-Yos}.
      Let $\Sigma$ be a Riemann surface with a K\"ahler metric $g_\Sigma$.
      Set $S^1 := \mathbb{R}/\mathbb{Z}$
      and $X:= S^1 \times \Sigma$.
      Let $p_i$ be the projection from $X$ to the $i$-th component.
      Let $q:\mathbb{R} \to S^1$ be the quotient map.
      Let us recall the mini-holomorphic structure on $X$ in \cite{Ref:Moc-Yos}.
      \begin{Def}\label{Def:min-hol. bdle}
      \ 
      \begin{enumerate}
      	\item
        	We define $\Omega^{0,1}(X) := {p_1}^{\ast}\Omega_{\mathbb{C}}^{1}(S^1) \oplus {p_2}^{\ast}\Omega^{0,1}(\Sigma)$
          and $\Omega^{0,2}(X) := \bigwedge^{2}\Omega^{0,1}(X)$.
          We define $\deebar_X : \Omega^{0,i}(X) \to \Omega^{0,i+1}(X)$ to be
          $\deebar_X = d_{S^1} + \deebar_{\Sigma}$.
          We call the tuple $(\Omega^{0,i}(X),\deebar_X)$ \textit{mini-holomorphic structure} on $X$.
        \item
        	Let $V$ be a vector bundle on an open subset $U\subset X$.
          Let $\Omega^{0,i}(U,V)$ denote the space of $V$-valued differential forms on $U$ of degree $(0,i)$.
          A differential operator $\deebar_{V}:\Omega^{0,0}(U,V) \to \Omega^{0,1}(U,V)$ is called
          mini-holomorphic structure of $V$ if it satisfies the following conditions.
          \begin{itemize}
          	\item
            	For any $f \in C^{\infty}(U)$ and $s \in \Omega^{0,i}(U,V)$,
              we have $\deebar_V(fs) = \deebar_X(f) \wedge s + f\deebar_V(s)$.
              Note that the differential operators $\deebar_{V}:\Omega^{0,i}(U,V) \to \Omega^{0,i+1}(U,V)$
              are naturally induced.
            \item
            	The integrability condition $\deebar_V \comp \deebar_V = 0$ is satisfied.
          \end{itemize}
      \end{enumerate}
      \end{Def}
      Let $I=(a,b) \subset \mathbb{R}$ be an open interval with $|b-a| < 1$
      and $W \subset \Sigma$ be a domain.
      Let $(V,\deebar_V)$ be a mini-holomorphic bundle on the open subset of the form $q(I) \times W \subset X$.
      We decompose the differential operator $\deebar_V(s)$ as
      $\deebar_V(s) = d_{V,S^1}(s) + \deebar_{V,\Sigma}(s) \in {p_1}^{\ast}\Omega^{1}(S^1) \oplus {p_2}^{\ast}\Omega^{0,1}(\Sigma)$ for a local section $s$ of $V$.
      For $t \in I$,
      let $V^t$ denote the holomorphic bundle $(V|_{q(t)\times W},\deebar_{V,q(t)})$ on $W$,
      where $\deebar_{V,q(t)}$ is the restriction of the differential operator $\deebar_{V,\Sigma}$ on $\{q(t)\}\times W$.
      For any fixed $x \in W$, we obtain a connection on $V|_{q(I)\times\{x\}}$ as the restriction of $d_{V,S^1}$ on $q(I)\times\{x\}$.
      Hence we have the parallel transport $\Psi_{t,t'}:V^t \to V^{t'}$ for any $t,t' \in I$.
      The isomorphism $\Psi_{t,t'}$ is called the scattering map in \cite{Ref:Cha-Hur}.
      Recall that the scattering map $\Psi_{t,t'}:V^t \to V^{t'}$ is a holomorphic isomorphism,
      which follows from the integrability condition $\deebar_V \comp \deebar_V = [d_{V,S^1},\deebar_{V,\Sigma}] = 0$.
      The next proposition shows that monopoles on $X$ have the underlying mini-holomorphic structures.
      \begin{Prp}\label{Prp:monopole induce mini-hol}
      	Let $(L,h,A,\Phi)$ be a monopole on $X$.
        We decompose the covariant derivative $\nabla_A$
        into 
        $
        	\nabla_A(f) = 
          \nabla_{A,t}(f)dt + \nabla^{1,0}_{A}(f) + \nabla^{0,1}_{A}(f) 
          \in {p_1}^{\ast}\Omega^{1}(S^1) \oplus {p_2}^{\ast}\Omega^{1,0}(\Sigma) \oplus {p_2}^{\ast}\Omega^{0,1}(\Sigma)
        $ for a local section $f$ of $L$.
        Then, the differential operator
        $\deebar_{L} := (\nabla_{A,t} -\sqrt{-1}\Phi)dt + \nabla^{0,1}_{A}$ is
        a mini-holomorphic structure on $L$.
      \end{Prp}
      \begin{proof}
      	It is standard that the integrability condition $\deebar_{L}\comp\deebar_{L} = 0$ 
        follows from the Bogomolny equation $F(A) = \ast\nabla_A(\Phi)$.
      \end{proof}
      
      Let $w$ be the standard coordinate of $\mathbb{C}$.
      Let $U \subset \mathbb{C}$ be a neighborhood of $0$ and put $U^{\ast} := U \setminus \{0\}$.
      Let $(V,h,A,\Phi)$ be a monopole of rank $r$ on $([-\eps,\eps]\times U) \setminus \{(0,0)\} \subset \mathbb{R}\times\mathbb{C}$,
      and let $(V,\deebar_V)$ denote the underlying mini-holomorphic bundle.
      The following proposition in \cite{Ref:Cha-Hur}
      allows us to interpret the weights of Dirac-type singularities
      in terms of the scattering maps of the underlying mini-holomorphic bundles.
      \begin{Prp}\label{Prp:dirac weight is merom degree}
        If $(0,0)$ is a Dirac-type singularity of weight $\vec{k} = (k_i) \in \mathbb{Z}^r$,
        then the scattering map $\Psi_{-\eps,\eps} : V^{-\eps}|_{U^{\ast}} \to V^{\eps}|_{U^{\ast}}$ is extended to a meromorphic isomorphism 
        $\Psi_{-\eps,\eps}:V^{-\eps}(\ast 0) \to V^{\eps}(\ast 0)$ and there exists a holomorphic frame $\mb{v}^{-}$ (resp. $\mb{v}^{+}$) of $V^{-\eps}$ (resp. $V^{\eps}$)
        such that $\Phi^{-\eps,\eps}$ can be represented as $\Psi_{-\eps,\eps}(\mb{v}^{-}) = \mb{v}^{+}\cdot \mathrm{diag}(w^{k_i})$,
        where $\mathrm{diag}(c_i)$ is the diagonal matrix whose $(i,i)$-th entries are $c_i$.
        Moreover, this type of diagonal matrix representation is unique up to permutations.
      \end{Prp}
  \subsection{Filtered sheaves and filtered bundles}\label{SubSec:Parabolic Str}
  	We recall the definitions of parabolic sheaves and bundles by following \cite{Ref:Moc1}.
    \subsubsection{Filtered sheaves}
    	Let $X$ be a complex manifold.
      Let $D$ be a smooth hypersurface of $X$ and $D = \coprod_{i=1}^{d} D_i$ be the decomposition into connected components.
    	Let $\mathcal{E}$ be a coherent $\mathcal{O}_X(\ast D)= \bigcup_{n\in\mathbb{Z}}\mathcal{O}_X(nD)$-module.
      
    	A tuple $P_{\ast}\mathcal{E}= \{P_{\mb{a}}\mathcal{E}\}_{\mb{a}=(a_i) \in \mathbb{R}^{d}}$
      of $\mathcal{O}_{X}$-submodules of $\mathcal{E}$ 
    	is called a filtered sheaf over $\mathcal{E}$ if it satisfies the following conditions:
    	\begin{itemize}
    		\item
      		$P_{\mb{a}}\mathcal{E} \subset \mathcal{E}$ is a coherent $\mathcal{O}_{X}$-module
        	and $P_{\mb{a}}\mathcal{E}|_{X \setminus D} = \mathcal{E}|_{X \setminus D}$ holds.
      	\item
      		For $\mb{a}=(a_i)$ and $\mb{a'}=(a'_i) \in \mathbb{R}^d$,
          we have $P_{\mb{a'}}\mathcal{E} \subset P_{\mb{a}}\mathcal{E}$
        	if $a'_i \leq a_i$ for any $i=1,\ldots,d$.
      	\item
      		On a small neighborhood $U$ of $D_i$, $P_{\mb{a}}\mathcal{E}|_{U}$ depends only on $a_i$,
        	which we denote by ${}^{i}P_{a_i}(\mathcal{E}|_{U})$.
      	\item
      		For any $i=1, \ldots, d$ and $a \in \mathbb{R}$, there exists $\epsilon > 0$
        	such that we have ${}^{i}P_{a}(\mathcal{E}|_{U}) ={}^{i}P_{a + \epsilon}(\mathcal{E}|_{U})$.
      	\item
      		For any $\mb{a} \in \mathbb{R}^d$ and $\mb{n}=(n_i) \in \mathbb{Z}^d$,
    			we have $P_{\mb{a}+\mb{n}}\mathcal{E} = P_{\mb{a}}\mathcal{E}(\sum n_i D_i)$.
    	\end{itemize}
    	
    	A filtered subsheaf $P_{\ast}\mathcal{E}' \subset P_{\ast}\mathcal{E}$ is a filtered sheaf
    	over a subsheaf $\mathcal{E}' \subset \mathcal{E}$ such that $P_{\mb{a}}\mathcal{E}' \subset P_{\mb{a}}\mathcal{E}$
    	for any $\mb{a} \in \mathbb{R}^d$.
    	If $P_{\mb{a}}\mathcal{E}' = \mathcal{E}' \cap P_{\mb{a}}\mathcal{E}$ holds for any $\mb{a} \in \mathbb{R}^d$,
    	it is called strict.
    	
    	For a small neighborhood $U$ of $D_i$,
    	we set ${}^iP_{<a}(\mathcal{E}|_{U}) := \sum_{a'<a}{}^iP_{a'}(\mathcal{E}|_{U})$.
    	We also define a coherent $\mathcal{O}_{D_i}$-module ${}^i\mathrm{Gr}_{a}(\mathcal{E})$ by 
    	${}^i\mathrm{Gr}_{a}(\mathcal{E}) = {}^iP_{<a}(\mathcal{E}|_{U}) / {}^iP_{<a}(\mathcal{E}|_{U})$.
    	We set
    	\[
    		\Par(P_{\ast}\mathcal{E}, i) := \{ a \in \mathbb{R} \mid {}^i\mathrm{Gr}_{a}(\mathcal{E}) \neq 0\}.
    	\]
    	
    	Suppose that $\mathcal{E}$ is torsion free.
    	The parabolic first Chern class $\mathrm{par\mathchar`-c_1}(P_{\ast}\mathcal{E})$ is defined as
    	\[
    		\mathrm{par\mathchar`-c_1}(P_{\ast}\mathcal{E}) 
    	  := c_1(P_{0,\ldots,0}V) - \sum_{i=1}^{d} \sum_{-1< a_i \leq 0} a_i \mathrm{rank}_{D_i}({}^{i}\mathrm{Gr}_{a_i}(\mathcal{E}))[D_i].
  		\]
    	Here $\mathrm{rank}_{D_i}$ denote the rank of coherent sheaves on $D_i$, and $[D_i]$ is the cohomology class of $D_i$ on $X$.
    \subsubsection{Filtered bundles}\label{SubSec:par bdle and their prop}
    	A filtered sheaf $P_{\ast}\mathcal{E}$ on $(X,D)$ is called a filtered bundle 
      if it satisfies the following conditions:
      \begin{itemize}
    		\item
        	For any $\mb{a} \in \mathbb{R}^d$,
          $P_{\mb{a}}\mathcal{E}$ is a locally free $\mathcal{O}_X$-module.
        \item
        	For any $i=1,\ldots,d$ and $a \in \mathbb{R}$,
          ${}^i\mathrm{Gr}_{a}(\mathcal{E})$ is a locally free $\mathcal{O}_{D_i}$-module.
      \end{itemize}
      
      For example,
      we define the trivial filtered bundle $\mathcal{O}_{X}(\ast D)$
      as $P_{\mb{a}}\mathcal{O}_{X}(\ast D) = \mathcal{O}_{X}(\sum [a_i]D_i)$,
      where $[a_i] \in \mathbb{Z}$ is the greatest integer satisfying $[a_i] \leq a_i$.
      
      The filtered bundle $P_{\ast}{\mathcal{H}om}_{\mathcal{O}_X(\ast D)}(\mathcal{E}_1,\mathcal{E}_2)$
      over ${\mathcal{H}om}_{\mathcal{O}_X(\ast D)}(\mathcal{E}_1,\mathcal{E}_2)$
      is defined as follows:
      \[ 
      	P_{\mb{a}}{\mathcal{H}om}_{\mathcal{O}_X(\ast D)}(\mathcal{E}_1, \mathcal{E}_2) = 
        	\{
          	f \in {\mathcal{H}om}(\mathcal{E}_1,\mathcal{E}_2) 
        		\mid f(P_{\mb{b}}\mathcal{E}_1) \subset P_{\mb{a}+\mb{b}}\mathcal{E}_2\,(\forall\mb{b}\in\mathbb{R}^d)
          \}.
      \]
      We denote $P_{\ast}{\mathcal{H}om}_{\mathcal{O}_X(\ast D)}(\mathcal{E}_1,\mathcal{E}_2)$
      by $P_{\ast}{\mathcal{H}om}(\mathcal{E}_1,\mathcal{E}_2)$ if there is no risk of confusion.
      
      For any filtered bundle $P_{\ast}\mathcal{E}$, the dual filtered bundle $P_{\ast}(\mathcal{E}^{\lor})$ is defined as
      $P_{\ast}{\mathcal{H}om}(\mathcal{E}, \mathcal{O}_{X}(\ast D))$,
      Then we have a natural isomorphism 
      \[
      	P_{\mb{a}}(\mathcal{E}^{\lor}) \simeq 
      	(P_{< -\mb{a} + \mb{\delta}}\mathcal{E})^{\lor} = (\bigcup_{\mb{b} < -\mb{a} + \mb{\delta}}P_{\mb{b}}\mathcal{E})^{\lor},
      \]
      where $\mb{\delta}=(1,\ldots,1) \in \mathbb{R}^d$.
    \subsubsection{Stable filtered bundles on $(\mathbb{P}^1 \times T^2, \{0,\infty\}\times T^2)$}\label{SubSubSec:par bdle on P^1xT^2}
    	We consider the case $X = \mathbb{P}^1 \times T^2$ and $D = (\{0\} \times T^2) \sqcup  (\{\infty\} \times T^2)$,
      where $T^2$ is an elliptic curve.
      For a filtered bundle $P_{\ast}\mathcal{E}$,
      we write $P_{ab}\mathcal{E}$ instead of $P_{\mb{a}}\mathcal{E}$ for $\mb{a}=(a,b) \in \mathbb{R}^2$.
      
      For a filtered bundle $P_{\ast\ast}\mathcal{E}$,
      we define the parabolic degree of $P_{\ast\ast}\mathcal{E}$ by
      \[
      	\mathrm{par\mathchar`-deg}(P_{\ast\ast}\mathcal{E})
        = \frac{\sqrt{-1}}{2}\int_{\mathbb{P}^1 \times \{w_0\}} \mathrm{par\mathchar`-c_1}(P_{\ast\ast}\mathcal{E}),\;(\forall w_0 \in T^2).
      \]
      
      \begin{Def}\label{Def:stable filtered bundle}
      	Let $P_{\ast\ast}\mathcal{E}$ be a filtered bundle on $(X,D)$.
        \begin{itemize}
        	\item
		        $P_{\ast\ast}\mathcal{E}$ is stable if it satisfies the following:
        		\begin{enumerate}
        			\item\label{Enum:gr semistable}
          			For any $a,b \in \mathbb{R}$,
            		$P_{ab}\mathcal{E}|_{\{0\}\times T^2}$ and $P_{ab}\mathcal{E}|_{\{\infty\}\times T^2}$ are semistable and of degree $0$.
          		\item
          			For any filtered subsheaf $P_{\ast\ast}\mathcal{E}' \subset P_{\ast\ast}\mathcal{E}$
            		satisfying (\ref{Enum:gr semistable}) and $0<\mathrm{rank}(\mathcal{E}') < \mathrm{rank}(\mathcal{E})$,
                we have
            		\[
      						\mathrm{par\mathchar`-deg}(P_{\ast\ast}\mathcal{E}') / \mathrm{rank}(\mathcal{E}') \;<\; 
      						\mathrm{par\mathchar`-deg}(P_{\ast\ast}\mathcal{E}) / \mathrm{rank}(\mathcal{E}).
      					\]
        		\end{enumerate}
          \item
          	$P_{\ast\ast}\mathcal{E}$ is polystable if there exists a decomposition $P_{\ast\ast}\mathcal{E} = \bigoplus_{i \in I} P_{\ast\ast}\mathcal{E}_i$
            such that $P_{\ast\ast}\mathcal{E}_i$ is stable and that
            $
            	\mathrm{par\mathchar`-deg}(P_{\ast\ast}\mathcal{E}_i) / \mathrm{rank}(\mathcal{E}_i) =
              \mathrm{par\mathchar`-deg}(P_{\ast\ast}\mathcal{E}_j) / \mathrm{rank}(\mathcal{E}_j)
            $ holds for any $i,j \in I$.
      	\end{itemize}
      \end{Def}
      \vspace{3mm}
      \noindent We have the following cohomology vanishing for stable filtered bundles of degree 0.
      \begin{Prp}\label{Prp:vanish cohom of stable bundle}
      	Let $p:\mathbb{P}^1\times T^2 \to T^2$ be the projection map.
        Let $F$ be a holomorphic line bundle of degree $0$ on $T^2$.
        For a stable filtered bundle $P_{\ast\ast}E$ on $\mathbb{P}^1\times T^2$ satisfying
        $\mathrm{par\mathchar`-deg}(P_{\ast\ast}E)=0$ and $\mathrm{rank}(E)>1$,
        we have
        $H^i(\mathbb{P}^1 \times T^2, P_{-\hat{t}\hat{t}}E \otimes p^{\ast}F) = H^i(\mathbb{P}^1 \times T^2, P_{<-\hat{t}<\hat{t}}E \otimes p^{\ast}F) =0$ for any $\hat{t}\in\mathbb{R}$ and any $i\neq 1$.
      \end{Prp}
      \begin{proof}
      	By replacing $E$ with $E\otimes p^{\ast}F$, we may assume that $F$ is trivial.
      	By considering $P_{\ast-\hat{t},\ast+\hat{t}}E$ instead of $P_{\ast\ast}E$,
      	we may also assume $\hat{t}=0$.
      	If $i<0$ or $i>2$ holds,
      	then the cohomologies vanish obviously.
      	Thus we prove only the cases $i=0$ and $i=2$. 
      	If there exists a non-zero global section of $P_{00}E$,
        we have a filtered subsheaf $P_{\ast\ast}\mathcal{O} \subset P_{\ast\ast}E$ of rank $1$
        that satisfies (\ref{Enum:gr semistable}) in Definition \ref{Def:stable filtered bundle} and
        $\mathrm{par\mathchar`-deg}(P_{\ast\ast}\mathcal{O}) \geq 0$.
        However, it contradicts the stability of $P_{\ast\ast}E$.
        Thus, we have $H^0(\mathbb{P}^1 \times T^2, P_{00}E) =0$.
        By the natural inclusion $P_{<0<0}V\subset P_{00}V$,
        we have $H^0(\mathbb{P}^1 \times T^2, P_{<0<0}E) =0$.
        By the natural isomorphism
        $P_{00}(E^{\lor}) \simeq (P_{<1<1}E)^{\lor}$,
        we have $H^0(\mathbb{P}^1 \times T^2, (P_{<1<1}E)^{\lor})= 0$.
        Using the Serre duality theorem and isomorphisms $P_{<1<1}E \simeq P_{<0<0}E\otimes O_X(D) \simeq P_{<0<0}E \otimes (\Omega^2_{\mathbb{P}^1 \times T^2})^{-1}$,
        we obtain $H^2(\mathbb{P}^1 \times T^2, P_{<1<1}E\otimes \Omega^2_{\mathbb{P}^1 \times T^2}) = H^2(\mathbb{P}^1 \times T^2, P_{<0<0}E)= 0$.
        We also have $H^2(P_{00}E)=0$ by the short exact sequence $0\to P_{<0<0}V \to P_{00}V \to {}^{1}\mathrm{Gr}_{0}(P_{\ast\ast}V)\oplus {}^{2}\mathrm{Gr}_{0}(P_{\ast\ast}V)\to 0$ and trivial vanishing of cohomologies $H^2(T^2,{}^{1}\mathrm{Gr}_{0}(P_{\ast\ast}V))=H^2(T^2,{}^{2}\mathrm{Gr}_{0}(P_{\ast\ast}V))=0$.
      \end{proof}
      
  \subsection{The Fourier-Mukai transform of semistable bundles}
  	We recall the Fourier-Mukai transform of semistable bundles of degree $0$ on an elliptic curve $T^2 := \mathbb{C}/\Lambda_2$
    by following \cite[Subsubsection 2.1.2]{Ref:Moc1}.
    
    Let $w$ be the standard coordinate of $\mathbb{C}$.
    We will denote by $\hat{T}^2$ the dual torus of $T^2$.
    Let $V$ be a semistable bundle of degree $0$ and of rank $r$ on $T^2$.
    As a part of result in \cite[Proposition 2.2]{Ref:Moc1}, we have the following proposition.
    \begin{Prp}\label{Prp:decomp of ss deg0}
      There exist $k \in \mathbb{N}$, $F_i \in \mathrm{Pic}^0(T^2)$ and nilpotent matrices $N_i \in \mathrm{Mat}(\mathbb{C},r_i)$
      ($1\leq i \leq k$ and $\sum_i r_i = r$) such that we have $F_i \not\simeq F_j$ and
      an isomorphism $V \simeq \bigoplus^k_{i=1} F_i \otimes (\underline{\mathbb{C}^{r_i}}, \deebar + N_i d\bar{w})$.
      Moreover, if we take another isomorphism $V \simeq \bigoplus^{k'}_{i=1} F'_i \otimes (\underline{\mathbb{C}^{r'_i}}, \deebar + N'_i d\bar{w})$,
      then we have $k=k'$ and there exist a permutation $\sigma$ of $\{1,\ldots,k\}$ and linear transformations $g_i \in GL(r_i,\mathbb{C})$ such that we have $F_i \simeq F_{\sigma(i)}$, $r_i=r'_i$ and $N_i = \mathrm{Ad}(g_i)N'_i$ for any $1\leq r\leq k$.
    \end{Prp}
    Recall the spectrum of semistable bundle of degree $0$ is defined as follows.\;(See \cite{Ref:Moc1}.)
    \begin{Def}\label{Def:spectrum of ss deg0}
      We take the decomposition $V\simeq \bigoplus^k_{i=1} F_i \otimes (\underline{\mathbb{C}^{r_i}}, \deebar + N_i d\bar{w})$
      as in Proposition \ref{Prp:decomp of ss deg0}.
      We define the spectrum set $\mathrm{Spec}(V) \subset \hat{T}^2$
      to be $\mathrm{Spec}(V) := \{F_1,\ldots,F_k\}$ under the identification $\hat{T}^2 \simeq \mathrm{Pic}^0(T^2)$.
    \end{Def}
	The following corollary is also obtained by \cite[Proposition 2.2]{Ref:Moc1}.
    \begin{Cor}\label{Cor:FM-trans of ss deg0}
      We take the decomposition $V\simeq \bigoplus^k_{i=1} F_i \otimes (\underline{\mathbb{C}^{r_i}}, \deebar + N_i d\bar{w})$
      as in Proposition \ref{Prp:decomp of ss deg0}.
    	Let $\mathcal{L} \to T^2\times \hat{T}^2$ be the Poincar\'{e} bundle.
      Let $p_i$ be the projection of $T^2 \times \hat{T}^2$ to the $i$-th component.
      For any $\alpha \in \mathrm{Spec}(V)$,
      we set a multi-index $I_{\alpha} = (i_{\alpha,1},\ldots,i_{\alpha,k_{\alpha}}) \in \mathbb{N}^{k_{\alpha}}$ as a tuple of 
      the sizes of the Jordan blocks of $N_i$ corresponding to $\alpha$.
      The Fourier-Mukai transform $\mathrm{FM}(V) := Rp_{2 \ast}(p^{\ast}_{1}V \otimes L) \in D^b(\mathrm{Coh}(\mathcal{O}_{\hat{T}^2}))$ is given as follows:
    	\begin{align*}
    		H^i(\mathrm{FM}(V)) &= 0,\;(i\neq 1)\\
        H^1(\mathrm{FM}(V)) &\simeq 
        \bigoplus_{\alpha \in \mathrm{Spec}(V)} \bigoplus_{j=1}^{k_{\alpha}} \mathcal{O}_{\hat{T}^2,\alpha} / m_{\hat{T}^2,\alpha}^{i_{\alpha,j}},
    	\end{align*}
    	where $m_{\hat{T}^2,\alpha}$ is the maximal ideal of the stalk $\mathcal{O}_{\hat{T}^2,\alpha}$.
    \end{Cor}

	\section{$L^2$-finite instantons on $\mathbb{R}\times T^3$}\label{Sec:L^2 fin Inst on RxT^3}
	In this  section, we fix a positive number $0<\lambda<1$.
  Recall that a function $f$ is called of $C^{i,\lambda}$-class 
  if $f$ is of $C^i$-class and all derivatives of $f$ of order $i$ are locally $\lambda$-H\"{o}lder continuous.
  \subsection{Asymptotic behavior of solutions of the Nahm equation on $(0,\infty)$}
  	For a $T^3$-invariant unitary connection $A$ on $(\underline{\mathbb{C}^r},\underline{h})$ on $(0,\infty)\times T^3$,
    the tuple $(\underline{\mathbb{C}^r},\underline{h},A)$ is an instanton if and only if
    the connection form of $A=\alpha\,dt+\sum_iA_idx^i$ satisfies the Nahm equation:
    \begin{eqnarray}\label{Eqn:Nahm eq}
    	\left\{\begin{array}{l}
      	\displaystyle\frac{\partial A_1}{\partial t} + [\alpha, A_1] = -[A_2, A_3] \\
        \\
        \displaystyle\frac{\partial A_2}{\partial t} + [\alpha, A_2] = -[A_3, A_1] \\
        \\
        \displaystyle\frac{\partial A_3}{\partial t} + [\alpha, A_3] = -[A_1, A_2]. \\
			\end{array}\right.
    \end{eqnarray}
    
    For skew-Hermitian commuting matrices $\Gamma_i \in \mathfrak{u}(r)\;(i=1,2,3)$,
    let $\mathrm{Center}(\Gamma)$ denote the centralizer of $\Gamma:=(\Gamma_i)$ in $\mathfrak{u}(r)$,
    i.e.\ $\mathrm{Center}(\Gamma)=\{a \in \mathfrak{u}(r)\;|\;[\Gamma_i,a]=0\;(i=1,2,3)\}$.
    If we take $N_i \in \mathrm{Center}(\Gamma)\;(i=1,2,3)$ satisfying the relations $N_i = [N_j,N_k]$ for any even permutation $(ijk)$ of $(123)$,
    then the tuple of $\alpha=0$ and $A_i=\Gamma_i+N_i/t$ $(i=1,2,3)$ forms a solution of (\ref{Eqn:Nahm eq}) on $(0,\infty)$.
    \begin{Def}\label{Def:model sol of Nahm eq}
    	A pair of tuples $(\Gamma=(\Gamma_i), N=(N_i))$ as described above is called a model solution of the Nahm equation.
    \end{Def}
    
    In \cite[Corollary 2.2 and Proposition 3.1]{Ref:Biq}, Biquard proved the following theorem.
    \begin{Thm}\label{Thm:Asymp of Nahm}
    	Let $\left(\alpha(t), A_i(t)\right)$ be a solution of the Nahm equation of $C^{3,\lambda}$-class on $(0,\infty)$ of rank $r$.
      Then, there exist a model solution of the Nahm equation $(\Gamma,N)$ and a $C^{4,\lambda}$-gauge transformation $g:(0,\infty) \to U(r)$
      such that the following conditions are satisfied.
      \begin{enumerate}
      	\item
        	The gauge transformation $g$ satisfies $g^{-1}\alpha g + g^{-1}\partial_{t}g = 0$.
        \item
        	We take the decomposition $g^{-1}A_i g - (\Gamma_i + N_i /t) = \eps_{1,i}(t) + \eps_{2,i}(t)$ satisfying 
          $\eps_{1,i}(t) \in \mathrm{Center}(\Gamma)$ and $\eps_{2,i}(t) \in (\mathrm{Center}(\Gamma))^{\perp}$,
          where $(\mathrm{Center}(\Gamma))^{\perp}$ means the orthogonal complement of
          $\mathrm{Center}(\Gamma)$ in $\mathfrak{u}(r)$ with respect to the inner product $\i<A,B> := -\mathrm{tr}(AB)$.
          Then, there exists $\delta>0$ such that the following estimates hold for any $1\leq i\leq 3$ and $0\leq j\leq 3$.
          \begin{eqnarray*}
          	\left\{
            	\begin{array}{ll}
		          	|\partial^{j}_{t} \eps_{1,i}(t)| &= O(t^{-(1+j+\delta)})\\
  		          |\partial^{j}_{t} \eps_{2,i}(t)| &= O(\exp(- \delta t))
              \end{array}
            \right.
          \end{eqnarray*}
      \end{enumerate}
    \end{Thm}
    \begin{proof}
    	By Corollary 2.2 and Proposition 3.1 in \cite{Ref:Biq},
      there exist a positive number $\delta>0$ and a gauge transformation $g_0:(0,\infty) \to U(r)$ such that
      for the transformed solution $(\tilde{\alpha},\tilde{A}_i)=(g_0^{-1}\alpha g_0 + g_0^{-1}\partial_{t}g_0, g_0^{-1}A_ig_0)$
      we have the following estimates:
      \begin{eqnarray}\label{Eqn:C^0 est. of Nahm eq.}
          	\left\{
            	\begin{array}{ll}
              		\left|\tilde{\alpha}\right| &= O(\exp(- \delta t))\\
		          	\left|\tilde{\eps}_{1,i}\right| &= O(t^{-(1+\delta)})\\
  		         	\left|\tilde{\eps}_{2,i}\right| &= O(\exp(- \delta t)),
              \end{array}
            \right.
      \end{eqnarray}
      where we set $\tilde{A}_i - (\Gamma_i + N_i /t) =: \tilde{\eps}_{i,1} + \tilde{\eps}_{i,2} \in \mathrm{Center}(\Gamma)\oplus(\mathrm{Center}(\Gamma))^{\bot}$.
      We take another gauge transform $g_1:(0,\infty) \to U(r)$ satisfying the following conditions:
      \begin{eqnarray*}
          	\left\{
            	\begin{array}{ll}
              	g_1^{-1}\tilde{\alpha} g_1 + g_1^{-1}\partial_{t}g_1 = 0\\
                \lim_{t \to \infty}g_1(t) = \mathrm{Id}.
            	\end{array}
            \right.
    	\end{eqnarray*}
      Then we have an estimate $|g_1(t) - \mathrm{Id}| = O(\exp(- \delta t))$.
      Hence the same estimate as (\ref{Eqn:C^0 est. of Nahm eq.}) holds for the gauge transform $g:=g_0g_1$.
      Moreover, by definition of $g_1$ we have $g^{-1}\alpha g + g^{-1}\partial_{t}g = 0$,
      and this shows that $g$ is of $C^{4,\lambda}$-class.
      For a permutation $(ijk)$ of $(123)$,
      the equation (\ref{Eqn:Nahm eq}) can be written as follows:
      \begin{equation}\label{Eqn:deriv. of err. term}
      	\partial_t(\eps_{1,i}+\eps_{2,i})
      	 = [N_j/t,\eps_{1,k}]+[\eps_{1,j},N_k/t]+ [\Gamma_j+N_j/t,\eps_{2,k}]+[\eps_{2,j},\Gamma_k+N_k/t] + [\eps_{1,j}+\eps_{2,j}, \eps_{1,k}+\eps_{2,k}].
      \end{equation}
      Since we have $[\mathrm{Center}(\Gamma),\mathrm{Center}(\Gamma)]\subset \mathrm{Center}(\Gamma)$, by bootstrapping argument from (\ref{Eqn:deriv. of err. term}) we obtain the desired estimates for derivatives of $\eps_{1,i}(t)$ and $\eps_{2,i}(t)$.
    \end{proof}
  \subsection{Asymptotic behavior of $L^2$-finite instantons on $(0,\infty) \times T^3$}\label{SubSec:Asymp of Inst on (0,inf)xT^3}
  	The following theorem is proved in \cite[Lemma 3.3.2, Theorem 4.3.1, Corollary 4.3.3, Corollary 5.1.3 and Theorem 5.2.2]{Ref:Mor-Mro-Rub}.
    \begin{Thm}\label{Thm:Appro of Inst}
    	Let $(V,h,A)$ be an $L^2$-finite instanton on $(0,\infty) \times T^3$.
      If we take a sufficiently large $R>0$, then there exist a positive number $\delta>0$, a trivialization of $C^{4,\lambda}$-class
      $\sigma:(V,h)|_{(R,\infty)\times T^3} \simeq (\underline{\mathbb{C}^r} , \underline{h})$
      and a $T^3$-invariant $L^2$-finite instanton $(\underline{\mathbb{C}^r} , \underline{h}, \tilde{A})$ on $(R,\infty)\times T^3$,
      such that we have the following estimates.
      \begin{eqnarray*}
      	\left\{\begin{array}{ll}
      		||A_\sigma||_{C^{3,\lambda}([t,t+1] \times  T^3)} &= O(1)\\
        	||A_\sigma - \tilde{A}||_{C^{3,\lambda}([t,t+1] \times  T^3)} &= O(\exp(- \delta t)),
        \end{array}\right.
      \end{eqnarray*}
      where $A_\sigma$ is the connection form of $A$ with respect to $\sigma$, and we identify $\tilde{A}$ with its connection form.
    \end{Thm}
    \begin{proof}
    	Since we have $\pi_1(T^3)\simeq \mathbb{Z}^3$,
    	by considering parallel transport,
    	it is proved that
    	for any flat Hermitian vector bundle $F$ on $T^3$ of rank $r$
    	there exists a tuple of Hermitian commuting matrices $\Gamma = (\Gamma_i)\subset \mathfrak{u}(r)\;(i=1,2,3)$ such that we have $F \simeq (\underline{\mathbb{C}^r},\underline{h},d+\sum_i \Gamma_i dx^i)$.
    	By taking a suitable gauge transformation of $(\underline{\mathbb{C}^r},\underline{h})$, we may assume the condition $\alpha_i - \beta_i \not\in (2\pi\sqrt{-1})\mathbb{Z}$ for any two distinct simultaneous eigenvalues $\alpha=(\alpha_i),\beta=(\beta_i) \in (\sqrt{-1}\mathbb{R})^3$ and for any $i=1,2,3$.
    	By Lemma 3.3.2, Theorem 4.3.1 and Corollary 4.3.3 in \cite{Ref:Mor-Mro-Rub},
    	there exist $R>0$, a commuting tuple of skew-Hermitian matrices $\Gamma = (\Gamma_i)\subset\mathfrak{u}(r)\;(i=1,2,3)$ and a trivialization of $C^{4,\lambda}$-class
    	$\tilde{\sigma}:(V,h)|_{(R,\infty)\times T^3} \simeq (\underline{\mathbb{C}^r} , \underline{h})$ such that
    	$||A_{\tilde{\sigma}} - \sum_i \Gamma_i dx^i||_{C^{3,\lambda}([t,t+1] \times  T^3)} = o(1)$.
    	In particular, we obtain $||A_{\tilde{\sigma}}||_{C^{3,\lambda}([t,t+1] \times  T^3)} = O(1)$.
    	
    	Let $\mathcal{A}_{L^2_3}$ be the space of $L^2_3$-connections of $(\underline{\mathbb{C}^r} , \underline{h})$ and $\mathcal{H}\subset \mathcal{A}_{L^2_3}$ be the Center manifold of the flat connection $\nabla_{\Gamma} = (d+\sum_i \Gamma_i dx^i)$ on $(\underline{\mathbb{C}^r} , \underline{h})$ in \cite[Section 5.1]{Ref:Mor-Mro-Rub}.
    	By the definition of the Center manifold, we have $\nabla_{\Gamma} \in \mathcal{H}$.
    	The $T^3$-action on $T^3$ itself induces the $T^3$-action on $\mathcal{A}_{L^2_3}$.
    	Since $\nabla_{\Gamma}$ is $T^3$-invariant,
    	by Corollary 5.1.3 in \cite{Ref:Mor-Mro-Rub},
    	$\mathcal{H}$ is a connected Riemannian manifold equipped with a $T^3$-action.
    	By \cite{Ref:Mor-Mro-Rub},
    	we have the $T^3$-equivariant isometry $T_{\nabla_{\Gamma}}\mathcal{H} \simeq H^1\left(\Omega_{T^3}^{\ast}(\underline{\mathfrak{u}(r)}),\nabla_{ad(\Gamma)}\right)$,
   		where $\nabla_{ad(\Gamma)}$ is the flat connection on $\underline{\mathfrak{u}(r)}$ induced by $\nabla_{\Gamma}$ with the adjoint representation of $\mathfrak{u}(r)$.
    	However, any elements of $H^1(T^3,\Omega^{\ast}(\underline{\mathbb{C}^r},\nabla_{\Gamma}))$ are $T^3$-invariant because of the assumption on simultaneous eigenvalues of $\Gamma$.
    	Therefore, since $\mathcal{H}$ is connected, the $T^3$-action on $\mathcal{H}$ is trivial.
    	Hence, by Lemma 3.3.2 and Theorem 5.2.2 in \cite{Ref:Mor-Mro-Rub},
    	there exist a positive number $\delta>0$, a trivialization $\sigma:(V,h)|_{(R,\infty)\times T^3} \simeq (\underline{\mathbb{C}^r} , \underline{h})$ and a $T^3$-invariant $L^2$-finite instanton $(\underline{\mathbb{C}^r} , \underline{h}, \tilde{A})$ on $(R,\infty)\times T^3$,
    	such that we have $||A_\sigma - \tilde{A}||_{C^{3,\lambda}([t,t+1] \times  T^3)} = O(\exp(- \delta t))$,
    	which completes the proof.
    \end{proof}
    We obtain the following corollary as a consequence of Theorem \ref{Thm:Appro of Inst} and Theorem \ref{Thm:Asymp of Nahm}.
    \begin{Cor}\label{Cor:Asymp of Inst}
    	Let $(V,h,A)$ be an $L^2$-finite instanton on $(0,\infty) \times T^3$ of rank $r$.
      If we take a sufficiently large $R>0$, then there exist a positive number $\delta>0$, a trivialization of $C^{4,\lambda}$-class
      $\sigma:(V,h)|_{(R,\infty)\times T^3} \simeq (\underline{\mathbb{C}^r} , \underline{h})$
      and a model solution of the Nahm equation $(\Gamma,N)$ such that the following holds.
    	\begin{enumerate}
      	\item
        	Let $\alpha dt + \sum_i A_i dx^{i}$ denote the connection form of $A$ with respect to $\sigma$.
          Then we have $\alpha = 0$.
        \item
        	There exist  decompositions 
          $A_i - (\Gamma_i + N_i /t) = \eps_{1,i}(t) + \eps_{2,i}(t) + \eps_{3,i}(t,x)$
          such that we have $\eps_{1,i}(t) \in \mathrm{Center}(\Gamma)$, $\eps_{2,i}(t) \in (\mathrm{Center}(\Gamma))^{\perp}$
          and that the following estimates for $1\leq i\leq 3$ and $0\leq j\leq 3$,
          \begin{eqnarray*}
      			\left\{\begin{array}{ll}
            	|\partial^{j}_{t} \eps_{1,i}| &= O(t^{-(1+j+\delta)})\\
            	|\partial^{j}_{t} \eps_{2,i}| &= O(\exp(- \delta t))\\
              ||\eps_{3,i}||_{C^{3,\lambda}([t,t+1]\times T^3)} &= O(\exp(- \delta t)).
            \end{array}\right.
          \end{eqnarray*}
      \end{enumerate}
    \end{Cor}
	\begin{Remark}\label{Rem:Asymp of Inst on neg. half line}
		For any model solution $(\Gamma,N)$,
		the tuple of $\alpha=0$ and $A_i:=\Gamma_i+N_i/t$ also forms a solution of the Nahm equation on $(-\infty,0)$.
		Hence, a similar result holds for $L^2$-finite instantons on $(-\infty,0)\times T^3$.
    \end{Remark}
    Let $(V,h,A)$ be an $L^2$-finite instanton on $\mathbb{R}\times T^3$.
    Applying Corollary \ref{Cor:Asymp of Inst} to $(V,h,A)|_{(0,\infty)\times T^3}$ and $(V,h,A)|_{(-\infty,0)\times T^3}$,
    we obtain model solutions $(\Gamma_{\pm},N_{\pm})$ which approximate $(V,h,A)$ at $t \to \pm\infty$.
    
    Since the simultaneous eigenvalues of $\sum_i\Gamma_{\pm,i}dx^i \in \Omega^1_{T^3}(\mathfrak{u}(r))$ are
    $T^3$-invariant pure imaginary $1$-forms on $T^3$,
    they can be regarded as elements of $\mathrm{Hom}(\mathbb{R}^3,\sqrt{-1}\mathbb{R})$.
    Thus we take $\widetilde{\mathrm{Spec}(\Gamma_{\pm})} \subset \mathrm{Hom}(\mathbb{R}^3,\mathbb{R})$
    as the set of $(2\pi\sqrt{-1})^{-1}$ times simultaneous eigenvalues of $\sum_i\Gamma_{\pm,i}dx^i$.
    We take unitary representations $\rho_{\pm}:\mathfrak{su}(2) \to \mathfrak{u}(r)$
    induced by $N_{\pm}$ to be $\rho(\sum_i a_ie_i) := \sum_i a_iN_{\pm,i}$,
    where $(e_i)_{i=1,2,3}$ is a basis of $\mathfrak{su}(2)$ satisfying 
    $e_i = [e_j,e_k]$ for any even permutation $(ijk)$ of $(123)$.
    Because $N_{\pm,i} \in \mathrm{Center}(\Gamma_{\pm})$,
    we have the decomposition
    $\rho_{\pm} = \displaystyle\bigoplus_{\xi \in \widetilde{\mathrm{Spec}(\Gamma_{\pm})}} \rho_{\pm,\xi}$
    which is induced by the simultaneous eigen decomposition of $\Gamma_{\pm}$.
    \begin{Def}\label{Def:sing set}\;\,\ 
      \begin{itemize}
      	\item
          We define the spectrum set $\mathrm{Spec}(\Gamma_{\pm}) \subset \hat{T}^3$ to be the image of $\widetilde{\mathrm{Spec}(\Gamma_{\pm})}$	
          by the quotient map $\mathrm{Hom}(\mathbb{R}^3,\mathbb{R}) \to \hat{T}^3$.
        \item
        	We define the singularity set of $(V,h,A)$ as
          $\mathrm{Sing}(V,h,A) := \mathrm{Spec}(\Gamma_{+}) \cup \mathrm{Spec}(\Gamma_{-})$.
        \item
        	We assume that the quotient map $\widetilde{\mathrm{Spec}(\Gamma_{\pm})} \to \mathrm{Spec}(\Gamma_{\pm})$ is bijective,
          and identify $\widetilde{\mathrm{Spec}(\Gamma_{\pm})}$ with $\mathrm{Spec}(\Gamma_{\pm})$. (See below Remark \ref{Remark:some attentions} \ref{Item:fundamental domain in Remark:some attentions}.)
          For $\xi \in \mathrm{Sing}(V,h,A)$,
          we define the unitary representations $\rho_{\pm,\xi}$ of $\mathfrak{su}(2)$ by
          putting $\rho_{\pm,\xi}:= 0$ for $\xi \in \mathrm{Sing}(V,h,A) \setminus \mathrm{Spec}(\Gamma_{\pm})$.
      \end{itemize}
    \end{Def}

    \begin{Remark}\label{Remark:some attentions}
    \ 
    	\begin{enumerate}[label=(\roman*)]
      	\item\label{Item:fundamental domain in Remark:some attentions}
		  We assume that there exists a fundamental domain $H \subset \mathrm{Hom}(\mathbb{R}^3,\mathbb{R})$ of $\hat{T}^3$
      		such that $0 \in H$ and $\widetilde{\mathrm{Spec}(\Gamma_{\pm})} \subset H$.
          Indeed, we can always take a suitable $\mathbb{R}$-invariant gauge transformation to satisfy this assumption.
        \item
        	For $\xi \in \hat{T}^3$, we take a flat Hermitian line bundle $L_{\xi}$ on $\mathbb{R}\times T^3$ as in subsection \ref{SubSec:Tori}.
          For any $\xi \in \hat{T}^3$ and any $L^2$-finite instanton $(V,h,A)$,
          we define $(V,h,A_{\xi}):=(V,h,A)\otimes L_{-\xi}$.
          Then we have
          \[
          	\mathrm{Sing}(V,h,A_{\xi}) = \mathrm{Sing}(V,h,A) - \xi = \{\mu-\xi \mid \mu \in \mathrm{Sing}(V,h,A)\}.
          \]
      \end{enumerate}
    \end{Remark}
  \subsection{Fredholmness of Dirac operators}\label{SubSec:Fredholm of Dirac op}
  	Let $(V,h,A)$ be an $L^2$-finite instanton on $\mathbb{R}\times T^3$.
    Take $\xi \in \hat{T}^3 \setminus\mathrm{Sing}(V,h,A)$,
    and we set $(V,h,A_{\xi}) := (V,h,A) \otimes L_{-\xi}$.
    We shall study the Dirac operators associated to $(V,h,A_{\xi})$ by following \cite{Ref:Cha2}.
    
    Let $\sigma_{0}$ be a global trivialization $(V,h)$ (see Corollary \ref{Cor:Topological triviality of inst}).
    Let $R_{\pm}>0$ be constants as in Corollary \ref{Cor:Asymp of Inst} with
    the $L^2$-finite instantons $(V,h,A)|_{(0,\infty)\times T^3}$ and $(V,h,A)|_{(-\infty,0)\times T^3}$ respectively.
    We set $R:=\mathrm{max}(R_+,R_-)$.
    We also denote by $\sigma_{\pm}$ trivializations of $(V,h)$ on $(R,\infty)\times T^3$ and $(-\infty,-R)\times T^3$
    in Corollary \ref{Cor:Asymp of Inst} respectively.
    Let $\sigma$ denote the triple $(\sigma_{-},\sigma_{0},\sigma_{+})$.
    Let $S_{\mathbb{R}\times T^3}=S^{+}\oplus S^{-}$ denote the spinor bundle of $\mathbb{R}\times T^3$ with respect to the trivial spin structure.
    \begin{Def}
      For $0\leq k \leq 4$, we define a norm 
      $||\cdot||_{L^2_{k,\sigma}} : L^2_{k,\mathrm{loc}}(\mathbb{R}\times T^3, V\otimes S^{\pm}) \to \mathbb{R}_{\geq 0} \cup \{\infty\}$
      as follows:
      \[
      	||f||^2_{L^{2}_{k, \sigma}} = ||\sigma_{+}(f)||_{L^{2}_{k}([R+1, \infty))}
        	 +||\sigma_{0}(f)||^2_{L^{2}_{k}(\left[-(R+2), R+2\right])}
           +||\sigma_{-}(f)||^2_{L^{2}_{k}((-\infty, -(R+1)])}.
      \]
      We set 
      $L^2_{k,\sigma}(\mathbb{R}\times T^3, V\otimes S^{\pm}) := \{f \in L^2_{k,\mathrm{loc}}(\mathbb{R}\times T^3, V\otimes S^{\pm}) \mid ||f||_{L^2_{k,\sigma}} < \infty\}$.
    \end{Def}
    \begin{Remark}
    	Since $||f||_{L^2} \leq ||f||_{L^2_{0,\sigma}} \leq 3||f||_{L^2}$,
      the ordinary $L^2$-norm and $L^2_{0,\sigma}$-norm are equivalent.
    \end{Remark}
    
    Let $S_{T^3}$ be the spinor bundle on $T^3$ with respect to the trivial spin structure.
    Let $p:\mathbb{R}\times T^3\to T^3$ is the projection map.
    There exists the isomorphism $S^{\pm} \simeq p^{\ast}S_{T^3}$
    such that the Clifford product can be written as follows.
    \begin{align*}
      	\mathrm{clif}(dt) &=
        	\left(
						\begin{array}{cc}
							0 & -\mathrm{Id}_{S_{T^3}} \\
							\mathrm{Id}_{S_{T^3}} & 0
						\end{array}
					\right),\\
				\mathrm{clif}(dx^i) &=
        	\left(
						\begin{array}{cc}
							0 & \mathrm{clif}_{T^3}(dx^i) \\
							\mathrm{clif}_{T^3}(dx^i) & 0
						\end{array}
					\right).
    \end{align*}
    Hence we obtain the following lemma.
    \begin{Lem}\label{Lem:Exp of Dirac op}
      Under the identification between $S^{+},S^{-}$ and $p_2^{\ast}S_{T^3}$,
      the Dirac operators $\dirac^{\pm}_{A}:V\otimes S^{\pm} \to V\otimes S^{\mp}$ with respect to
      the trivialization $\sigma_{\pm}$ can be written as follows:
      \[
    		\dirac_{A}^{\pm} = \pm\frac{\partial}{\partial t} + \Dirac_{A|_{\{t\}\times T^3}},
      \]
      where $\Dirac_{A|_{\{t\}\times T^3}}$ is the Dirac operator of $(V, A)|_{\{t\}\times T^3}$ on $T^3$.
    \end{Lem}
    \begin{proof}
    	The connection forms of $A$ with respect to $\sigma_{\pm}$ are temporal \textit{i.e.} they have no $dt$ terms.
    \end{proof}
    \begin{Prp}\label{Prp:a priori est}
    	For any $1\leq k\leq4$, there exist $K_k,C_k>0$ such that the following estimates hold for any
      $f \in L^2_{k,\sigma}(\mathbb{R}\times T^3, V\otimes S^{\pm})$.
      \[
      	||f||_{L^2_{k,\sigma}} \leq C_k \left(||f||_{L^2(|t|<K_k)} + ||\dirac^{\pm}_{A_{\xi}}(f)||_{L^2_{k-1,\sigma}}\right)
			\]
    \end{Prp}
    \begin{proof}
    	By considering $(V,h,A_{\xi})$ instead of $(V,h,A)$ we may assume $\xi=0$.
      We consider the case $k=1$.
      By the assumption $\xi = 0 \not\in \mathrm{Sing}(V,h,A)$
      we have $\mathrm{Ker}(\Dirac_{\Gamma_{\pm}})=\{0\}$,
      where $\Gamma_{\pm} = d + \sum_{i} (\Gamma_{\pm})_idx^i$ are flat unitary connections on the product bundle
      $(\underline{\mathbb{C}^r}, \underline{h})$ on $T^3$.
      Thus we can easily prove by using the Fourier series expansion that
      there exists $B_1 > 0$ such that for any section $g \in L^2_1(T^3, \underline{\mathbb{C}^r}\otimes S_{T^3})$ we have
      \[ B_1||g||_{L^2_1(T^3)} \leq ||\Dirac_{\Gamma_{\pm}}(g)||_{L^2(T^3)}. \]
      Take $U_1 > R+1$ such that
      \begin{equation}\label{Eqn:error est. in Prp:a priori est}
      	||A_{\sigma_{+}} - \Gamma_{+}||_{C^{0}(t > U_1 - 1)} < B_1/4, \quad
        ||A_{\sigma_{-}} - \Gamma_{-}||_{C^{0}(t < -(U_1 - 1))} < B_1/4.
      \end{equation}
      We take a smooth function $\varphi^{+}_{U_1}: \mathbb{R} \to [0,1]$ which satisfies the following:
      \begin{align*}
      	\varphi^{+}_{U_1}(t) &=
        	\left\{\begin{array}{ll}
 						1 &\quad (t > U_1)\\
						0 &\quad (t < U_1-1).
					\end{array}\right.\\
      \end{align*}
      We also set $\varphi^{-}_{U_1}(t) := \varphi^{+}_{U_1}(-t)$ and $\varphi_{U_1}(t):=\varphi^{+}_{U_1}(t)+\varphi^{-}_{U_1}(t)$.
      For any $f \in L^2_{1,\sigma}(\mathbb{R}\times T^3, V\otimes S^{\pm})$, we have
      \begin{eqnarray*}
      	\begin{array}{lll}
        	||\dirac_{A}^{\pm}(\varphi_{U_1} f)||_{L^2} &=& ||\partial_t(\varphi_{U_1} f) \pm \Dirac_{A}(\varphi_{U_1} f)||_{L^2}\\
          &\geq &
            ||\partial_t(\varphi_{U_1} f) \pm \left( \Dirac_{\Gamma_{+}}(\varphi^{+}_{U_1} f) + \Dirac_{\Gamma_{-}}(\varphi^{-}_{U_1} f)\right)||_{L^2} \\
            & &- ||\mathrm{clif}(A_{\sigma_{+}}-\Gamma_{+})\cdot (\varphi^{+}_{U_1} f)||_{L^2}
            - ||\mathrm{clif}(A_{\sigma_{-}}-\Gamma_{-})\cdot (\varphi^{-}_{U_1} f)||_{L^2}.
        \end{array}
      \end{eqnarray*}
      Here we use an equality
      $
      	\i<\partial_t(\varphi^{\pm}_{U_1} f), \Dirac_{\Gamma_{\pm}}(\varphi^{\pm}_{U_1} f)>_{L^2}
        = -\i<\Dirac_{\Gamma_{\pm}}(\varphi^{\pm}_{U_1} f), \partial_t(\varphi^{\pm}_{U_1} f)>_{L^2}
      $, then we have
      \begin{eqnarray*}
      	\begin{array}{lll}
        	||\dirac_{A}^{\pm}(\varphi_{U_1} f)||_{L^2}
        	& \geq&
          	\displaystyle\frac{1}{3}\left(
            	||\partial_t(\varphi_{U_1} f)||_{L^2} 
          		+ ||\Dirac_{\Gamma_{+}}(\varphi^{+}_{U_1} f)||_{L^2}
          		+ ||\Dirac_{\Gamma_{-}}(\varphi^{-}_{U_1} f)||_{L^2}
            \right)\\
            & & - ||\mathrm{clif}(A_{\sigma_{+}}-\Gamma_{+})\cdot (\varphi^{+}_{U_1} f)||_{L^2}
            - ||\mathrm{clif}(A_{\sigma_{-}}-\Gamma_{-})\cdot (\varphi^{-}_{U_1} f)||_{L^2}.
         \end{array}
      \end{eqnarray*}
      By this inequality and the inequalities (\ref{Eqn:error est. in Prp:a priori est}),
      there exists $B_2 > 0$ such that we have
      \[ ||\dirac_{A}^{\pm}(\varphi_{U_1} f)||_{L^2} \geq B_2 ||\varphi_{U_1} f||_{L^2_{1, \sigma}}. \]
      Therefore, the following inequalities hold.
      \begin{eqnarray*}
      	\begin{array}{ll}
        	||f||_{L^2_{1, \sigma}}
    	    	&\leq ||\varphi_{U_1} f||_{L^2_{1, \sigma}} + ||f||_{L^2_{1, \sigma}(|t|<U_1)}\\
            &\leq {B_2}^{-1} ||\dirac_{A}^{\pm}(\varphi_{U_1} f)||_{L^2} + ||f||_{L^2_{1, \sigma}(|t|<U_1)}\\
            &\leq {B_2}^{-1}(||\varphi_{U_1} \dirac_{A}^{\pm}(f)||_{L^2} + ||\mathrm{clif}(d\varphi_{U_1})f||_{L^2} ) + ||f||_{L^2_{1, \sigma}(|t|<U_1)}\\
            &\leq {B_2}^{-1}(||\dirac_{A}^{\pm}(f)||_{L^2}) + (1 + {B_2}^{-1}||\partial_t\varphi_{U_1}||_{L^{\infty}})||f||_{L^2_{1, \sigma}(|t|<U_1)}
        \end{array}
      \end{eqnarray*}
      Applying the interior estimate for elliptic operators to the above inequality,
      we obtain an inequality
      \[ ||f||_{L^2_{1, \sigma}} \leq  C_1 \left(||f||_{L^2([-K_1,K_1]\times T^3)} + ||\dirac^{\pm}_{A}(f)||_{L^2}\right),\]
      where $C_1>0$ and $K_1>U_1$ are constants independent of $f$.
      This is the desired inequality for $k=1$.
      
      We use an induction on $k$.
      Suppose that we have already obtained the desired inequality in the case $k=k_0$.
      We take a constant $U_2 > K_{k_0}$ and take a function $\varphi_{U_2}$ as above.
      Then, for any $f \in L^{2}_{k_0+1, \sigma}$ we have
      \begin{align}
      	||f||_{L^{2}_{k_0+1, \sigma}}
        	&\leq ||f||_{L^{2}_{k_0+1, \sigma}(|t|<K_2)} + ||\varphi_{U_2}f||_{L^{2}_{k_0+1, \sigma}}\nonumber \\
          &\leq 
          	||f||_{L^{2}_{k_0+1, \sigma}(|t|<U_2)} + 
            ||\varphi_{U_2}f||_{L^{2}_{k_0, \sigma}} + 
            ||\partial_t(\varphi_{U_2}f)||_{L^{2}_{k_0, \sigma}} + 
            \sum_{i=1}^3 ||\partial_i(\varphi_{U_2}f)||_{L^{2}_{k_0, \sigma}}, \label{Eqn:Mid Est}
      \end{align}
      where $\partial_t(\varphi_{U_2}f)$ and $\partial_i(\varphi_{U_2}f)$
      are taken under the trivializations $\sigma_{\pm}$.
      Since we can apply the interior estimate of elliptic operators and the assumption of the induction
      to the first and second terms of (\ref{Eqn:Mid Est}),
      there exist $B_3 > 0$ and $U_3 > U_2$ such that we have
      \begin{equation}\label{Eqn:Last Est1}
      	||f||_{L^{2}_{k_0+1, \sigma}(|t|<U_2)} + ||\varphi_{U_2}f||_{L^{2}_{k_0, \sigma}}
        	\leq B_3 \left(||f||_{L^2(|t|<U_3)} + ||\dirac^{\pm}_{A}(f)||_{L^2_{k_0,\sigma}}\right).
      \end{equation}
      We also make an estimate of the third and fourth terms of (\ref{Eqn:Mid Est}) as follows:
      \begin{align}
	      	\ &||\partial_t(\varphi_{U_2}f)||_{L^{2}_{k_0, \sigma}} + \sum_{i=1}^3 ||\partial_i(\varphi_{U_2}f)||_{L^{2}_{k_0, \sigma}}\nonumber\\
          \leq& ||\dirac_{A}^{\pm}(\partial_t(\varphi_{U_2}f))||_{L^{2}_{k_0-1, \sigma}} + \sum_{i=1}^3 ||\dirac_{A}^{\pm}(\partial_i(\varphi_{U_2}f))||_{L^{2}_{k_0-1, \sigma}}\nonumber\\
          \leq& ||\partial_t(\dirac_{A}^{\pm}(\varphi_{U_2}f))||_{L^{2}_{k_0-1, \sigma}}+ ||[ \dirac_{A}^{\pm}, \partial_t](\varphi_{U_2}f)||_{L^{2}_{k_0-1, \sigma}}\nonumber\\
          \ & \;\;\; + \sum_{i=1}^3\left\{||\partial_i(\dirac_{A}^{\pm}(\varphi_{U_2}f))||_{L^{2}_{k_0-1, \sigma}} + ||[ \dirac_{A}^{\pm}, \partial_i](\varphi_{U_2}f)||_{L^{2}_{k_0-1, \sigma}}\right\}\nonumber\\
        	\leq& 4||\dirac_{A}^{\pm}(\varphi_{U_2}f)||_{L^{2}_{k_0, \sigma}} + ||[ \dirac_{A}^{\pm}, \partial_t](\varphi_{K_2}f)||_{L^{2}_{k_0-1, \sigma}} + \sum_{i=1}^3||[ \dirac_{A}^{\pm}, \partial_i](\varphi_{U_2}f)||_{L^{2}_{k_0-1, \sigma}}\nonumber\\
          \leq& B_4(||f||_{L^{2}([|t|<U_4)} + ||\dirac_{A}^{\pm}(f)||_{L^{2}_{k_0, \sigma}}).\label{Eqn:Last Est2}
      \end{align}
      Here $B_4 >0$ and $U_4 > U_2$ is a constant independent of $f$.
      As a consequence of (\ref{Eqn:Mid Est}), (\ref{Eqn:Last Est1}) and (\ref{Eqn:Last Est2}),
      we obtain the desired inequality for $k=k_0+1$,
      and the proof is complete.
    \end{proof}
    \begin{Cor}\label{Cor:Reg of Dirac op}
    	For $0\leq k \leq 3$,
      if $f \in L_k^2(\mathbb{R}\times T^3,V \times S^{\pm})$ satisfies
      $\dirac^{\pm}_{A_{\xi}}(f)=g \in L_k^2(\mathbb{R}\times T^3,V \times S^{\mp})$ as a distribution,
      then $f \in L^2_{k+1,\sigma}(\mathbb{R}\times T^3,V \times S^{\pm})$.
    \end{Cor}
    \begin{proof}
    	By the regularity of elliptic operators,
      we have $f \in L^2_{k+1,\mathrm{loc}}(\mathbb{R}\times T^3,V \times S^{\pm})$.
      For $n \in \mathbb{N}$, we take bump functions $\varphi_n : \mathbb{R} \to [0,1]$ satisfying
      \begin{align*}
      	\varphi_{n}(t) &=
        	\left\{\begin{array}{ll}
 						1 &\quad (|t| < n)\\
						0 &\quad (|t| > n+1).
					\end{array}\right.
      \end{align*}
    	From Proposition \ref{Prp:a priori est}, there exist $C,K>0$ such that
      we have 
      \[
      	||\varphi_{n} f - \varphi_{m} f||_{L^2_{k+1, \sigma}} \leq 
        C \left(
        	||\mathrm{clif}(d\varphi_{n}) f||_{L^2_{k, \sigma}} + ||\mathrm{clif}(d\varphi_{m}) f||_{L^2_{k, \sigma}} +
          ||(\varphi_{n}-\varphi_{m})g||_{L^2_{k, \sigma}}
        \right)
      \]
      for any natural numbers $n,m >K$.
      Hence $\{\varphi_{n} f\}$ is a Cauchy sequence in $L^2_{k+1, \sigma}$.
      Moreover, this sequence converges pointwise to $f$.
      Therefore $f \in L^2_{k+1, \sigma}(\mathbb{R}\times T^3,V \times S^{\pm})$.
    \end{proof}
    \begin{Thm}\label{Thm:Dirac op is Fredholm}
    	For $1\leq k\leq 4$,
      the operators
      \[ \dirac^{\pm}_{A_{\xi}} : L^2_{k,\sigma}(\mathbb{R}\times T^3,V \times S^{\pm}) \to L^2_{k-1, \sigma}(\mathbb{R}\times T^3,V \times S^{\mp})\]
      are Fredholm,
       and $\mathrm{Ker}(\dirac^{\pm}_{A_{\xi}})$ are independent of $k$.
    \end{Thm}
    \begin{proof}
    	By Corollary \ref{Cor:Reg of Dirac op},
      it suffices to prove the case $k=1$.
      Thus we prove the following assertions.
      \begin{enumerate}[label=(\roman*)]
      	\item $\mathrm{dim}(\mathrm{Ker}(\dirac_{A_{\xi}}^{\pm})) < \infty$\label{Enum:ker finite}
        \item $\mathrm{dim}(\mathrm{Cok}(\dirac_{A_{\xi}}^{\pm})) < \infty$\label{Enum:cok finite}
        \item $R(\dirac_{A_{\xi}}^{\pm}) \subset L^2(V \times S^{\mp})$ is closed\label{Enum:range closed}
      \end{enumerate}
      If a normed space has a relatively compact neighborhood of the origin,
			then it is finite dimensional.
      Hence \ref{Enum:ker finite} is an easy consequence of Proposition \ref{Prp:a priori est} and 
      the compactness of the restriction map $L^2_{1, \sigma}(\mathbb{R}\times T^3, V\otimes S^{\pm}) \to L^2([-K, K]\times T^3, V\otimes S^{\pm})$.
      By Corollary \ref{Cor:Reg of Dirac op}, we have $\mathrm{Cok}(\dirac^{\pm}_{A_{\xi}}) = \mathrm{Ker}(\dirac^{\mp}_{A_{\xi}})$.
      Therefore, \ref{Enum:cok finite} is deduced from \ref{Enum:ker finite}.
      
      To prove \ref{Enum:range closed},
      it is enough to prove that there exists $C>0$ such that for any $f \in (\mathrm{Ker}(\dirac^{\pm}_{A_{\xi}}))^{\perp_{L^2}}$ we have
      \begin{equation}\label{Eqn: upper estimate in Thm:Dirac op is Fredholm}
      	||f||_{L^{2}_{1, \sigma}} \leq C||\dirac^{\pm}_{A_{\xi}}(f)||_{L^{2}},
      \end{equation}
      where $(\mathrm{Ker}(\dirac^{\pm}_{A_{\xi}}))^{\perp_{L^2}}$ means the orthogonal complement of $\mathrm{Ker}(\dirac^{\pm}_{A_{\xi}})$
      in $L^2_{1,\sigma}$ with respect to the ordinary $L^2$ inner product.
      Suppose that there is no constant $C>0$ satisfying the inequality (\ref{Eqn: upper estimate in Thm:Dirac op is Fredholm})
      for any $f \in (\mathrm{Ker}(\dirac^{\pm}_{A_{\xi}}))^{\perp_{L^2}}$.
      Take $f_n \in (\mathrm{Ker}(\dirac^{\pm}_{A_{\xi}}))^{\perp_{L^2}}$
      satisfying $||f_n||_{L^{2}_{1, \sigma}} = 1 > n||\dirac^{\pm}_{A_{\xi}}(f_n)||_{L^{2}}$ for any $n \in \mathbb{N}$.
      Since the restriction map $L^2_{1,\sigma}(\mathbb{R}\times T^3, V\otimes S^{\pm}) \to L^2([-K, K]\times T^3, V\otimes S^{\pm})$
      is compact, we may assume that $\{f_n|_{[-K, K]\times T^3}\}$ converges in $L^2([-K, K]\times T^3, V\otimes S^{\pm})$.
      We have $||\dirac^{\pm}_{A_{\xi}}(f_n)||_{L^{2}} < 1/n \to 0\;(n\to\infty)$,
      and hence $\{f_n\}$ also converges to some $f_{\infty} \in L^2_{1,\sigma}(\mathbb{R}\times T^3, V\otimes S^{\pm})$ by Proposition \ref{Prp:a priori est}.
      Then we have $f_{\infty} \in \mathrm{Ker}(\dirac^{\pm}_{A_{\xi}})$ and $f_{\infty} \neq 0$.
      This contradicts $f_n \in (\mathrm{Ker}(\dirac^{\pm}_{A_{\xi}}))^{\perp_{L^2}}$.
      Therefore the inequality (\ref{Eqn: upper estimate in Thm:Dirac op is Fredholm}) holds for some $C>0$.
    \end{proof}
  \subsection{Index of Dirac operators}\label{SubSec:Index thm of Dirac op}
  	We calculate the index of Dirac operators by following Charbonneau \cite{Ref:Cha2}.
    \begin{Thm}\label{Thm:Index of dirac op}
    	Let $(V,h,A)$ be an $L^2$-finite instanton on $\mathbb{R}\times T^3$ of rank $r$ and take $\xi \in \hat{T}^3 \setminus \mathrm{Sing}(V,h,A)$.
      The index of $\dirac^{+}_{A_{\xi}}$ is given by
      \[ \mathrm{index}(\dirac^{+}_{A_{\xi}}) = -\frac{1}{8\pi^2}||F(A_{\xi})||^{2}_{L^2} =-\frac{1}{8\pi^2}||F(A)||^{2}_{L^2}.\]
    \end{Thm}
    \begin{proof}
    	Replacing $(V,h,A)$ with $(V,h,A_{\xi})$, and we may assume $\xi=0$.
      We may also assume that any $\Gamma_{\pm,i}$ are diagonal matrices because $\Gamma_{\pm} =(\Gamma_{\pm,i})$ are commuting Hermitian matrices.
    	Take a positive constant $K>R$,
      and a partition of unity $\{\phi_{-},\phi_{0},\phi_{+}\}$ on
      $\mathbb{R}$ which is subordinate to the open cover $\left\{\left( -\infty,-K \right),\left( -(K+1),K+1 \right),\left( K+1,\infty \right)\right\}$.
      We set a connection $a_K := \phi_{-}\Gamma_{-} + \phi_{0}A + \phi_{+}\Gamma_{+}$,
      where $\Gamma_{\pm}$ are the connections given by $d + \sum_{i}\Gamma_{\pm,i} dx^i$ with respect to the trivialization $\sigma_{\pm}$.
      Then $\dirac^{+}_{A} - \dirac^{+}_{a_K} = \mathrm{clif}\left( \phi_{-}(A - \Gamma_{-}) + \phi_{+}(A - \Gamma_{+})\right)$
      is a compact operator, hence we have
      \[ \mathrm{index}(\dirac^{+}_{A}) = \mathrm{index}(\dirac^{+}_{a_K}). \]
      We take a continuous family of flat unitary connections $\{\Gamma_s\}_{s \in [0,1]}$
      on $(V,h)|_{(R,\infty)\times T^3}$ satisfying the following conditions:
      \begin{itemize}
      	\item
        	$\Gamma_0 = \Gamma_{+}$.
        \item
        	The connection form of $\Gamma_1$ with respect to $\sigma_{+}$ is given by
          $\sum_i\Gamma_{-,i}dx^i$.
        \item
        	For any $s \in [0,1]$, $0 \not\in \mathrm{Spec}(\Gamma_s)$.
      \end{itemize}
      We set connections $\{a^{s}_K\}$ on $(V,h)$ as 
      $a^{s}_K = \phi_{-}\Gamma_{-} + \phi_{0}{A_{xi}} + \phi_{+}\Gamma_{s}$.
      Then, $\{\dirac^{+}_{a^{s}_K}\}$ forms a continuous family of Fredholm operators.
      Hence we have
      \[ \mathrm{index}(\dirac^{+}_{A}) = \mathrm{index}(\dirac^{+}_{a_K}) = \mathrm{index}(\dirac^{+}_{a^{1}_K}). \]
      We construct a Hermitian vector bundle $(\tilde{V},\tilde{h})$ on a four-dimensional torus $T^4$
      by gluing $(V,h)$ on $t<-(K+1)$ and $t>K+1$ with trivializations $\sigma_{\pm}$ respectively.
      Since the connection forms of $a^{1}_K$ on $|t|>K+1$ with respect to $\sigma_{+}$ and $\sigma_{-}$ are equal,
      we also construct a connection $\widetilde{a^{1}_K}$ on $(\tilde{V},\tilde{h})$ from $a^{1}_K$.
      Then the relative index theorem in \cite{Ref:Gro-Law} tells us
      \begin{equation}\label{Eqn:relative index}
      	\mathrm{index}(\dirac^{+}_{a^{1}_K}) - \mathrm{index}(\dirac^{+}_{\gamma_{-}}) = \mathrm{index}(\dirac^{+}_{\widetilde{a^{1}_K}}) - \mathrm{index}(\dirac^{+}_{\widetilde{\gamma_{-}}}),
      \end{equation}
      where $\gamma_{-}$ (resp. $\widetilde{\gamma_{-}}$) is a flat connection on the product bundle $(\underline{\mathbb{C}^r},\underline{h})$
      on $\mathbb{R}\times T^3$ (resp. $T^4$) whose connection form is given by $\sum_i\Gamma_{-,i}dx^i$.
      By the assumption $\xi= 0 \not\in \mathrm{Sing}(V,h,A)$,
      we have $\mathrm{index}(\dirac^{+}_{\gamma_{-}}) = \mathrm{index}(\dirac^{+}_{\widetilde{\gamma_{-}}}) = 0$.
      Hence we obtain $\mathrm{index}(\dirac^{+}_{a^{1}_K}) = \mathrm{index}(\dirac^{+}_{\widetilde{a^{1}_K}})$.
      By the Atiyah-Singer index theorem, we obtain
      \[
      	\mathrm{index}(\dirac^{+}_{\widetilde{a^{1}_K}})
        = \mathrm{ch}_2(\widetilde{a^{1}_K})/[T^4].
      \]
      Hence we have
      \begin{align*}
      	\mathrm{index}(\dirac^{+}_{A})
        &= \frac{1}{8\pi^2}\int_{T^4} \mathrm{Tr}\left( F(\widetilde{a^{1}_K})\wedge F(\widetilde{a^{1}_K}) \right)\\
        &= \frac{1}{8\pi^2}\int_{[-(K+1), K+1]\times T^3} \mathrm{Tr}\left( F(a^{1}_K)\wedge F(a^{1}_K) \right).
			\end{align*}
      Since any $\Gamma_{\pm,i}$ are assumed to be diagonal matrices,
      by Corollary \ref{Cor:Asymp of Inst} we have 
      \[ \left| \int_{[-(K+1), K+1]\times T^3}\mathrm{Tr}\left( F(a^{1}_K)\wedge F(a^{1}_K) \right) - \int_{\mathbb{R}\times T^3}\mathrm{Tr}\left( F(A)\wedge F(A)\right)\ \right| = O(K^{-2}). \]
      Taking the limit of $K \to \infty$, we obtain
      \[\mathrm{index}(\dirac^{+}_{A}) = \frac{1}{8\pi^2}\int_{\mathbb{R}\times T^3}\mathrm{Tr}\left( F(A)\wedge F(A) \right) = -\frac{||F_{A}||^{2}_{L^2}}{8\pi^2},\]
			which proves the theorem.
    \end{proof}
    \begin{Remark}\label{Rem:surj of neg Dirac op}
    	Let $\xi \in \hat{T}^3 \setminus \mathrm{Sing}(V,h,A)$.
    	Since $\mathbb{R}\times T^3$ has infinite volume,
      the Weitzenb\"{o}ck formula $\dirac^{-}_{A_{\xi}}\dirac^{+}_{A_{\xi}} = \nabla^{\ast}_{A_{\xi}}\nabla_{A_{\xi}}$
      tells us
      \[
      	\mathrm{dim}\left(\mathrm{Ker}(\dirac^{-}_{A_{\xi}})\right) = 
        \mathrm{dim}\left(\mathrm{Cok}(\dirac^{+}_{A_{\xi}})\right) = 
        \frac{||F_{A_{\xi}}||^{2}_{L^2}}{8\pi^2},
      \]\[
      	\mathrm{dim}\left(\mathrm{Ker}(\dirac^{+}_{A_{\xi}})\right) = 
        \mathrm{dim}\left(\mathrm{Cok}(\dirac^{-}_{A_{\xi}})\right) = 
        0.
      \]
    \end{Remark}
  \subsection{Asymptotic behavior of Harmonic spinors}\label{SubSec:Asymp of Harmonic Spinor}
  	Let $(V,h,A)$ be an $L^2$-finite instanton on $\mathbb{R}\times T^3$.
    
  	\begin{Prp}\label{Prp:ODIneq for Harmonic Spinor}
    	There exist $K,\kappa:\mathbb{R}_{>0} \to \mathbb{R}_{>0}$ such that conditions are satisfied.
      \begin{itemize}
      	\item
	        $K(d),\kappa(d)^{-1} =  O(d^{-1})$ as $d \to 0$.
        \item
        	Let $\xi \in \hat{T}^3 \setminus \mathrm{Sing}(V,h,A)$ and $f \in \mathrm{Ker}(\dirac^{-}_{A_{\xi}})\cap L^2$.
          Set $F(t) := \int_{\{t\}\times T^3} |f(t,x)|^2dx$,
          where $dx$ means the volume form of $T^3$.
          For any $t > K(d)$ (resp. $t<-K(d)$)
          we have $F'(t) \leq -\kappa(d) F(t)$ (resp. $F'(t) \geq \kappa(d) F(t)$).
          Here we abbreviate $\mathrm{dist}(\xi,\mathrm{Sing}(V,h,A))$ to $d$.
    	\end{itemize}
    \end{Prp}
    \begin{proof}
    	We may assume $f \neq 0$.
      By the interior regularity of elliptic operators, we have $F \in C^4(\mathbb{R}) \cap L^1(\mathbb{R})$.
      Hence we can calculate derivatives of $F$.
      \begin{align*}
      	F'(t)  &= 2\int_{T^3} \i<\partial_{t}f, f> dx.\\
        F''(t) &= 2\left\{ \int_{T^3} \i<\partial^2_{t}f, f> dx +  2\int_{T^3} |\partial_{t}f|^2 dx \right\}.
      \end{align*}
      By Lemma \ref{Lem:Exp of Dirac op}, Dirac operators with respect to $\sigma_{+}$ can be written as
      $\dirac_{A_{\xi}}^{-} = -\partial_t + \Dirac_{{A_{\xi}}|_{\{t\}\times T^3}}$.
      Thus, for $t>R$ we have
      \begin{align}
      	F'(t)  &= 2\int_{T^3} \i<\Dirac_{A_z}f, f> dx. \nonumber\\
        F''(t) 
        	&= 2\left\{ \int_{T^3}\i<\partial_{t}(\Dirac_{A_{\xi}}(f)), f> dx +  \int_{T^3} |\Dirac_{A_{\xi}}f|^2 dx \right\} \nonumber\\
        	&= 2\left\{ \int_{T^3}\i<\Dirac_{A_{\xi}}(\partial_{t}(f)), f> dx + \int_{T^3}\i<[\partial_{t}, \Dirac_{A_{\xi}}](f), f> dx +  \int_{T^3} |\Dirac_{A_{\xi}}f|^2 dx \right\} \nonumber\\
      		&= 2\left\{ \int_{T^3}\i<\Dirac_{A_{\xi}}(\Dirac_{A_{\xi}}(f)), f> dx + \int_{T^3}\i<[\partial_{t}, \Dirac_{A_{\xi}}](f), f> dx +  \int_{T^3} |\Dirac_{A_{\xi}}f|^2 dx \right\} \nonumber\\
      		&= 4\int_{T^3} |\Dirac_{A_{\xi}}f|^2 dx + 2\int_{T^3}\i<[\partial_{t}, \Dirac_{A_{\xi}}](f), f> dx \nonumber\\
          &= 4\int_{T^3} |\Dirac_{\Gamma_{+, {\xi}}}(f) + \mathrm{clif}(A - \Gamma_{+})(f)|^2 dx + 2\int_{T^3}\i<[\partial_{t}, \Dirac_{A_{\xi}}](f), f> dx. \label{Eqn:Deriv of F}
      \end{align}
      
      Now we use the Fourier series expansion on $T^3$ and get the following estimate:
      there exists $C_1 >0$ such that we have
      \begin{equation}
      	\int_{T^3} |\Dirac_{\Gamma_{+, z}}(f)|^2 dx \geq d^2 C_1 F(t). \label{Eqn:Est of Dirac}
      \end{equation}
      Moreover, Corollary \ref{Cor:Asymp of Inst} tells us
      \begin{align}
      	\left| \int_{T^3}|\mathrm{clif}(A - \Gamma_{+})(f)|^2 \right| = O(t^{-2})\cdot F(t), \label{Eqn:Est of comm1}\\
        \left| \int_{T^3}\i<[\partial_{t}, \Dirac_{A_{\xi}}](f), f> \right| = O(t^{-2})\cdot F(t). \label{Eqn:Est of comm2}
      \end{align}
      By applying (\ref{Eqn:Est of Dirac}), (\ref{Eqn:Est of comm1}) and (\ref{Eqn:Est of comm2}) to (\ref{Eqn:Deriv of F}),
      we can take positive functions $K(d) = O(d^{-1})$ and $\kappa(d)^{-1} = O(d^{-1})$ such that if $t>K(d)$, then $F''(t) > \kappa(d)^2 F(t)$.
      
      We set $\tilde{F}(t) := \exp(\kappa(d) t)F(t)$.
      Then, the inequality $F''(t) > \kappa(d)^2 F(t)$ is equivalent to $\tilde{F}''(t) > 2\kappa(d) \tilde{F}'(t)$.
      If we suppose there exists $t_0 > K(d)$ such that $\tilde{F}'(t_0) > 0$, then $F(t)\exp(-\kappa(d) t/2) \to \infty\;(t\to\infty)$
      and this contradicts $F \in L^1(\mathbb{R})$.
      Therefore, for any $t>K$ we have $\tilde{F}'(t_0) \leq 0$ i.e.\ $F'(t) \leq -\kappa(d) F(t)$.
      
      The same proof works for $t<0$ mutatis mutandis.
    \end{proof}
    \begin{Cor}\label{Cor:decay of harmonic spinor}
    	Let $\xi \in \hat{T}^3 \setminus \mathrm{Sing}(V,h,A)$.
      There exists $C>0$ such that the following estimate holds for any $f \in \mathrm{Ker}(\dirac^{-}_{A_{\xi}})\cap L^2$:
      \[ ||f||_{C^{4, \lambda}([t, t+1]\times T^3)} = O\left( \exp(-C|t|) \right). \]
    \end{Cor}
  \section{Construction of the Nahm transform}\label{Sec:Construction of Nahm trans}
	\subsection{Construction of monopoles}\label{SubSec:Construction of monopole}
  	We construct the Nahm transform by following Charbonneau \cite{Ref:Cha2}.
  	Let $(V,h,A)$ be an $L^2$-finite instanton on $\mathbb{R}\times T^3$.
    Let $(\mathcal{V},||\cdot||_{L^2},d)$ be a flat Hermitian vector bundle on $\hat{T}^3$ which is the quotient of the product bundle 
    $\underline{\left( L^2(\mathbb{R}\times T^3, V\otimes S^{-}) , ||\cdot||_{L^2} \right)}$ on $\mathbb{R}^3$
    by the $\Lambda^{\ast}_3$-action $v\cdot(\xi, f) := (\xi + v, \exp(2\pi\sqrt{-1}\i< x,v >)f)$.
    We set $(\hat{V},\hat{h})$ as the finite-dimensional subbundle of $\mathcal{V}|_{\hat{T}^3 \setminus \mathrm{Sing}(V,h,A)}$
    defined by $\hat{V}_{\xi} := \mathrm{Ker}\left( \dirac^{-}_{A_{\xi}} \right) \cap L^2$.
    Indeed, as mentioned in Remark \ref{Rem:surj of neg Dirac op},
    $\dirac^{-}_{A_{\xi}}:L^2_{1,\sigma} \to L^2$ is a continuous family of surjective Fredholm operators,
    hence $(\hat{V},\hat{h})$ is a finite-dimensional subbundle of $\mathcal{V}|_{\hat{T}^3 \setminus \mathrm{Sing}(V,h,A)}$
    by the implicit function theorem.
    Moreover, Theorem \ref{Thm:Index of dirac op} tells us $\mathrm{rank}(\hat{V})=(8\pi^{2})^{-1}||F(A)||^{2}_{L^2}$.
    
    Let $\hat{A}$ be the connection on $(\hat{V},\hat{h})$ induced by the flat connection $d_{\mathcal{V}}$ on $\mathcal{V}$,
    namely $\hat{A} = Pd_{\mathcal{V}}$, where $P:\mathcal{V}|_{\hat{T}^3 \setminus \mathrm{Sing}(V,h,A)} \to \hat{V}$ is the orthogonal projection.
    
    Let $\hat{\Phi}$ denote a skew-Hermitian section of $\mathrm{End}(\hat{V})$ given by $\Phi_{\xi}(f) := P_{\xi}(2\pi\sqrt{-1} tf)$.
    Since any $f \in \mathrm{Ker}\left( \dirac^{-}_{A_{\xi}} \right) \cap L^2$ decays exponentially in $t \to \pm\infty$,
    $2\pi\sqrt{-1} tf$ is an $L^2$ section.
    \begin{Def}
    	$(\hat{V}, \hat{h}, \hat{A}, \hat{\Phi})$ is called the Nahm transform of $(V,h,A)$.
    \end{Def}
    \begin{Prp}\label{Prp:transformed bundle is monopole}
    	$(\hat{V}, \hat{h}, \hat{A}, \hat{\Phi})$ is a monopole on $\hat{T}^3 \setminus \mathrm{Sing}(V,h,A)$.
    \end{Prp}
    \begin{proof}
    	According to Charbonneau \cite[Subsection 3.1 and 3.2]{Ref:Cha2},
      $(\hat{V}, \hat{h}, \hat{A}, \hat{\Phi})$ is a monopole
      if for any open subset $U \subset \hat{T}^3\setminus \mathrm{Sing}(V,h,A)$ and any local section $f\in\Gamma(U,\hat{V})$,
      $(d_{\mathcal{V}}f)_{\xi} \in L^2(\mathbb{R}\times T^3,V\otimes S^{-})\otimes \Omega^1_{\hat{T}^3,\xi}$ decays exponentially at $t\to\pm\infty$ for any $\xi\in U$.
      Since $d_{\mathcal{V}}f$ satisfies the partial differential equation
      $\partial_t(d_{\mathcal{V}}f) = \Dirac_{A_\xi}(d_{\mathcal{V}}f) + \mathrm{clif}_{\mathbb{R}\times T^3}(\i<d\xi,dx>)f$ and Corollary \ref{Cor:decay of harmonic spinor},
      the decay condition of $d_{\mathcal{V}}f$ can be proved by a similar way with the proof of Proposition \ref{Prp:ODIneq for Harmonic Spinor}.
    \end{proof}
  \subsection{Singularities of the Nahm transform}\label{SubSec:Sing pts of monopole}
  	Let $(V,h,A)$ be an $L^2$-finite instanton on $\mathbb{R}\times T^3$ and
    $(\hat{V}, \hat{h}, \hat{A}, \hat{\Phi})$ be the Nahm transform of $(V,h,A)$.
    In this subsection, we prove the following theorem.
    \begin{Thm}\label{Thm:Transformed monopole is Dirac type}
    	Each point of $\mathrm{Sing}(V,h,A)$ is a Dirac-type singularity of $(\hat{V}, \hat{h}, \hat{A}, \hat{\Phi})$.
    \end{Thm}
    \begin{proof}
    	By Proposition \ref{Prp:Moc's criterion for Dirac sing},
      it suffices to show $|\hat{\Phi}(\xi)| = O(d(\xi,p)^{-1})\;(\xi \to p)$ for any $p \in \mathrm{Sing}(V,h,A)$.
      Since $\hat{\Phi}$ is skew-Hermitian with respect to $\hat{h}$,
      we have $|\hat{\Phi}(\xi)| \leq \mathrm{rank}(\hat{V}) \cdot \mathrm{max}\{|\i<\hat{\Phi}(f),f>| \mid f\in \hat{V}_{\xi}\mbox{ and }|f|=1\}$.
      Hence we have only to show $|\i<\hat{\Phi}(f),f>| = O(d(\xi,p)^{-1})||f||^2_{L^2}$ for $f \in \mathrm{Ker}(\dirac^{-}_{A_{\xi}}) \cap L^2$.
      We set $F(t) := \int_{\{t\}\times T^3} |f(t,x)|^2 dx$.
      By Proposition \ref{Prp:ODIneq for Harmonic Spinor},
      we take functions $K,\kappa:\mathbb{R}_{>0} \to \mathbb{R}_{>0}$.
      We abbreviate $K(d(\xi,p))$ and $\kappa(d(\xi,p))$ to $K$ and $\kappa$ respectively.
      Then we have
      \begin{align*}
      	|\i<\hat{\Phi}f, f>|
        	\leq& \int_{-K}^{K} |t|F(t) dt + \int_{\{|t|>K\}} |t|F(t) dt \\
          \leq&  K ||f||^2_{L^2} + \int_{\{|t|>K\}} |t|F(t) dt.
      \end{align*}
      Here we use integration by parts, then
      \begin{align*}
      	|\i<\hat{\Phi}f, f>|
        	\leq&  K ||f||^2_{L^2} + K \int_{\{|t|>K\}} F(t) dt + \int_{\{t>K\}}\left\{\int_{|s|>t} F(s) ds \right\}dt \\
          \leq&  2K ||f||^2_{L^2}  + \int_{\{t>K\}}\left\{\int_{|s|>t} F(s) ds \right\}dt.
      \end{align*}
      Since the inequality $F'(t) \leq -\kappa(d) F(t)$ (resp. $F'(t) \geq \kappa(d) F(t)$) holds for any $t > K(d)$ (resp. $t<-K(d)$),
      we have $\int_{|s|>t} F(s) ds \leq \kappa^{-1}( F(t) + F(-t) )$.
      Therefore we obtain
      \[
      	|\i<\hat{\Phi}f, f>|
          \leq  2K ||f||^2_{L^2}  + \kappa^{-1}\int_{\{|t|>K\}}F(t)dt 
          \leq (2K + \kappa^{-1})||f||^2_{L^2}.
      \]
      Since $K,\kappa^{-1} = O(d(\xi,p)^{-1})$, we have $|\i<\hat{\Phi}(f),f>| = O(d(\xi,p)^{-1})||f||^2_{L^2}$.
    \end{proof}
	\section{Algebraic Nahm transform}\label{Algebraic Nahm transform}
	In this section,
  we assume that $T^3$ is isomorphic to the product of a circle $S^1 = \mathbb{R}/\mathbb{Z}$ and a 2-dimensional torus $T^2 = \mathbb{R}^2/\Lambda_2$ as a Riemannian manifold.
  Then, we have $\hat{T}^3 = S^1 \times \hat{T}^2$, where $\hat{T}^2 = \mathrm{Hom}(\mathbb{R}^2,\mathbb{R})/\Lambda_2^{\ast}$ is the dual torus of $T^2$.
  Under this assumption, we can regard $\mathbb{R}\times T^3$ as a K\"{a}hler manifold by setting holomorphic coordinates
  $\tau = t + \sqrt{-1}x^1 \in \mathbb{R}\times S^1$ and $w = x^2 + \sqrt{-1}x^3 \in T^2$.
  Since the map $\mathbb{R} \times S^1 \ni \tau = t + \sqrt{-1}x^1 \to \exp(2\pi\tau) = z \in \mathbb{C}^{\ast}$ is biholomorphic,
  we have a biholomorphic and isometric map $\mathbb{R}\times T^3 \simeq (\mathbb{C}^{\ast}, dzd\bar{z}/|2\pi z|^2)\times T^2$.
  
  We will construct a stable filtered bundle on $(\mathbb{P}^1\times T^2,\{0,\infty\}\times T^2)$
  from an $L^2$-finite instanton on $\mathbb{R}\times T^3 \simeq \mathbb{C}^{\ast}\times T^2$
  as a prolongation of holomorphic vector bundles.
  Next, from a stable filtered bundle on $(\mathbb{P}^1\times T^2,\{0,\infty\}\times T^2)$ of rank $r>1$
  we construct a mini-holomorphic bundle on $\hat{T}^3 = S^1 \times\hat{T}^2$ outside a finite subset,
  and we call this construction \textit{the algebraic Nahm transform}.
  Finally, for an irreducible $L^2$-finite instanton $(V,h,A)$ on $\mathbb{R}\times T^3$ of rank $r>1$
  and the associated stable filtered bundle $P_{\ast\ast}V$ on $(\mathbb{P}^1\times T^2,\{0,\infty\}\times T^2)$,
  we show that the algebraic Nahm transform of $P_{\ast\ast}V$
  is isomorphic to the underlying mini-holomorphic bundle of the Nahm transform of $(V,h,A)$.
	\subsection{Asymptotic behavior of $L^2$-finite instantons as holomorphic bundles}\label{SubSec:Asymp of Inst as hol bdle}
  	We refine Corollary \ref{Cor:Asymp of Inst}
    in order to make it compatible with the complex structure of $\mathbb{R}\times T^3$.
    \begin{Prp}\label{Prp:hol Asymp of Inst}
    	Let $(V,h,A)$ be an $L^2$-finite instanton on $(0,\infty)\times T^3$ of rank $r$.
      If we take a sufficiently large $R>0$, then there exist a $C^2$-frame $\mb{v}=(v_i)$ of $V$ on $(R,\infty)\times T^3$,
      a model solution $(\Gamma,N)$ of the Nahm equation and a positive number $\delta>0$
      such that the following holds.
      \begin{enumerate}[label=(\roman*)]
      	\item\label{Enum:1st in Prp:hol Asymp of Inst}
        	If we write the $(0,1)$-part of connection form of $A$ with respect to $\mb{v}$ as
        	$\nabla^{0,1}_A(\mb{v}) = \mb{v}\left( A_{\bar{\tau}} d\bar{\tau} + A_{\bar{w}} d\bar{w} \right)$,
          then $A_{\bar{\tau}}$ and $A_{\bar{w}}$ are $T^2$-invariant, and we have
          $[A_{\bar{\tau}}, \Gamma_{\bar{w}}] = [A_{\bar{w}}, \Gamma_{\bar{w}}]=0$,
          where $\Gamma_{\bar{\tau}}d\bar{\tau} + \Gamma_{\bar{w}}d\bar{w} := (\sum_i\Gamma_idx^i)^{(0,1)}$.
        \item\label{Enum:2nd in Prp:hol Asymp of Inst}
        	We also take $N_{\bar{w}}, N_{\bar{\tau}}$ to be $N_{\bar{\tau}}d\bar{\tau} + N_{\bar{w}}d\bar{w} := (\sum_iN_idx^i)^{(0,1)}$.
          We set $\tilde{\eps}_{\bar{w}} := A_{\bar{w}} - (\Gamma_{\bar{w}} + N_{\bar{w}} / t)$
          and $\tilde{\eps}_{\bar{\tau}} := A_{\bar{\tau}} - (\Gamma_{\bar{\tau}} + N_{\bar{\tau}} / t)$.
          Then, the following estimates hold:
          \begin{eqnarray}
          	\left\{\begin{array}{ll}
            	|\tilde{\eps}_{\bar{\tau}}|, |\tilde{\eps}_{\bar{w}}| &= O(t^{-(1+\delta)})\\
            	|\partial_t \tilde{\eps}_{\bar{\tau}}|, |\partial_t \tilde{\eps}_{\bar{w}}| &= O(t^{-(2+\delta)})\\
            	|\partial_1 \tilde{\eps}_{\bar{\tau}}|, |\partial_1 \tilde{\eps}_{\bar{w}}| &= O(\exp(- \delta t)),
            \end{array}\right.\nonumber
          \end{eqnarray}
          where $\partial_1$ means the partial derivative with respect to $x^1$.
        \item\label{Enum:3rd in Prp:hol Asymp of Inst}
        	For any $1\leq i,j \leq r$,
          we have $||\i< v_i,v_j > - \delta_{ij}||_{C^2([t,t+1]\times T^3)} = O(\exp(- \delta t))$.
      \end{enumerate}
    \end{Prp}
	\begin{Remark}
		As Corollary \ref{Cor:Asymp of Inst} and Remark \ref{Rem:Asymp of Inst on neg. half line},
		we obtain a similar result for $L^2$-finite instantons on $(-\infty,0)\times T^3$.
	\end{Remark}
    \begin{proof}
    	Applying Corollary \ref{Cor:Asymp of Inst} to $(V,h,A)$, we obtain
      a trivialization $(V,h)|_{(R,\infty)\times T^3} \simeq (\underline{\mathbb{C}^r},\underline{h})$
      and a model solution $(\Gamma,N)$ of the Nahm equation.
      Take an orthonormal frame $\mb{u}$ on $(R,\infty)\times T^3$
      such that $\Gamma = (\Gamma_i)$ are diagonal.
      We take a Hermitian vector space $E = (\mathbb{C}^r,h_{\mathbb{C}^r})$
      and the eigen decomposition of $\Gamma_{\bar{w}}$ \textit{i.e.}  $E = \bigoplus_{\alpha \in \mathbb{C}}E_{\alpha}$ and
      $\Gamma_{\bar{w}} = \sum_{\alpha} \alpha \mathrm{Id}_{E_{\alpha}}$.
      
      We take Banach spaces $X_1$ and $X_2$ as
      \[ X_i := \Bigl(\bigoplus_{\alpha} C^{i, \lambda}(T^2, \underline{\mathrm{End}(E_{\alpha})})^{\perp}\Bigr)\oplus\Bigl(\bigoplus_{\alpha \neq \beta} C^{i, \lambda}(T^2, \underline{\mathrm{Hom}(E_{\alpha}, E_{\beta})})\Bigr), \]
      where $C^{i, \lambda}(T^2, \underline{\mathrm{End}(E_{\alpha})})^{\perp}$ is the kernel of the linear map $f \mapsto \int_{T^2}f$.
      We define a smooth map
      \[F:X_{2}\times \left(\bigoplus_{\alpha}\mathrm{End}(E_{\alpha}) \right)\times X_{1} \to X_{1}\]
      as
      \[ F(f, a, \eps) := \pi\Bigl( \mathrm{Ad}(\exp(f))(\Gamma_{\bar{w}}+a+\eps) + \exp(-f) \deebar_{w} \exp(f) \Bigr), \]
      where $\pi:C^{1, \lambda}(T^2, \underline{\mathrm{End}(E)}) \to X_1$ is the projection.
      Then, ${dF}_{(0, 0, 0)}|_{X_2} : X_2 \to X_1$ is an isomorphism because we have $dF_{(0,0,0)}(f,0,0) = \pi(\deebar_{w}(f))$.
      Therefore, the implicit function theorem shows that there exist a small neighborhood $U$ of $(0,0) \in \bigl(\bigoplus_{\alpha}\mathrm{End}(E_{\alpha}) \bigr)\times X_1$
      and a smooth map $G:U \to X_{2}$ such that for any $(f,a,\eps) \in \mathrm{Dom}(F)$ sufficiently close to $(0,0,0)$,
      $F(f,a,\eps)=0$ and $f = G(a,\eps)$ are equivalent.
      \begin{Lem}\label{Lem:est of G}
      	There exists $M>0$ such that for any $(a,\eps) \in U$ we have
        \begin{align*}
        	&||G(a, \eps)||_{C^{2, \lambda}(T^2)} \leq M||\eps||_{C^{1, \lambda}(T^2)}\\
          &||{\partial_a}G(a, \eps)||_{\mathcal{B}(\bigoplus_{\alpha}\mathrm{End}(E_{\alpha}), C^{2, \lambda}(T^2))} \leq M||\eps||_{C^{1, \lambda}(T^2)}\\
          &||\partial^{2}_{a} G(a, \eps)||_{\mathcal{B}(\bigoplus_{\alpha}\mathrm{End}(E_{\alpha}), \mathcal{B}(\bigoplus_{\alpha}\mathrm{End}(E_{\alpha}), C^{2, \lambda}(T^2)))} \leq M||\eps||_{C^{1, \lambda}(T^2)},
      	\end{align*}
        where $\mathcal{B}(X,Y)$ is the Banach space consisting of bounded linear maps from a Banach spaces $X$ to a Banach space $Y$.
      \end{Lem}
      \begin{proof}
      	By $F(0,a,0)=0$ and the definition of $G$, we have $G(a, 0)=0$.
        Thus, we have ${\partial_a}G(a, 0)=0$ and $\partial^{2}_{a} G(a, 0)=0$, and this proves the lemma.
      \end{proof}
      For the error terms $\eps_{i,j}\;(1\leq i,j \leq 3)$ in Corollary \ref{Cor:Asymp of Inst},
      we set $\eps_{i,\bar{\tau}}d\bar{\tau}+\eps_{i,\bar{w}}d\bar{w} := (\sum_j \eps_{i,j}dx^{j})^{(0,1)}$.
      Set a gauge transformation $g := \exp\bigl( G(\eps_{1,\bar{w}}, \eps_{2,\bar{w}} +\pi(\eps_{3,\bar{w}})) \bigr)$
      and take a frame $\mb{v} = \mb{u}g$.
      By Lemma \ref{Lem:est of G} and the estimates of $\eps_{2,j},\eps_{3,j}$ in Corollary \ref{Cor:Asymp of Inst},
      we have $||g - \mathrm{Id}||_{C^{2}([t, t+1]\times T^3)} = O(\mathrm{exp}(-\delta t))$.
      This proves \ref{Enum:2nd in Prp:hol Asymp of Inst} and \ref{Enum:3rd in Prp:hol Asymp of Inst}.
      
      We prove \ref{Enum:1st in Prp:hol Asymp of Inst}.
      By the definition of $F$,
      $A_{\bar{w}}$ is $T^2$-invariant and $[A_{\bar{w}},\Gamma_{\bar{w}}]=0$.
      We will prove that $A_{\bar{\tau}}$ is also $T^2$-invariant and $[A_{\bar{\tau}},\Gamma_{\bar{w}}]=0$.
      For an eigenvalue $\alpha \in \mathbb{C}$ of $\Gamma_{\bar{w}}$,
      let us denote by $\mb{v}_{\alpha}$ the subset of $\mb{v}$ corresponding to $\alpha$.
      Then we write $\nabla^{0,1}_A$ as
      \[
      	\nabla^{0.1}_{A}(\mb{v}_\alpha) =
      	(\sum_{\beta}\mb{v}_{\beta} A^{\beta}_{\bar{\tau}, \alpha})d\bar{\tau}+ \mb{v}_{\alpha} A_{\bar{w}, \alpha}d\bar{w}.
      \]
      Then, rewriting $\nabla^{0, 1}_{A} \comp \nabla^{0, 1}_{A} = 0$
      gives
      \[\deebar_{w}(A^{\beta}_{\bar{\tau}, \alpha}) = A^{\beta}_{\bar{\tau}, \alpha}A_{\bar{w}, \alpha} - A_{\bar{w}, \beta}A^{\beta}_{\bar{\tau}, \alpha}\]
      for any eigenvalues $\alpha \neq \beta$ of $\Gamma_{\bar{w}}$.
      We have $\lim_{t\to\infty}A_{\bar{w}, \alpha} = \alpha\mathrm{Id}$,
      and by Remark \ref{Remark:some attentions} (i)
      we also have $\alpha - \beta \not\in 2\pi\mathbb{Z}$ for any eigenvalues $\alpha \neq \beta$ of $\Gamma_{\bar{w}}$.
      Therefore, if $R>0$ is sufficiently large, then we obtain $A_{\bar{\tau}, \alpha}^{\beta} = 0$ on $(R,\infty)\times T^3$ by the Fourier series expansion.
      This is equivalent to $[A_{\bar{\tau}}, \Gamma_{\bar{w}}] =0$.
      Here we use $\nabla^{0, 1}_{A} \comp \nabla^{0, 1}_{A} = 0$ again,
      and we have
      \begin{equation}\label{Eqn:integral condition}
      	\deebar_{\tau}(A_{\bar{w}, \alpha}) - \deebar_{w}(A^{\alpha}_{\bar{\tau}, \alpha}) + [A^{\alpha}_{\bar{\tau}, \alpha}, A_{\bar{w}, \alpha}]=0.
      \end{equation}
      Removing the $T^2$-invariant part of (\ref{Eqn:integral condition}), and we obtain
      \[ \deebar_{w}\left( (A^{\alpha}_{\bar{\tau}, \alpha})^{\perp} \right) + \left[ A_{\bar{w}, \alpha} , (A^{\alpha}_{\bar{\tau}, \alpha})^{\perp} \right] =0,\]
      where $(A^{\alpha}_{\bar{\tau}, \alpha})^{\perp}$ means the non-constant part of the Fourier series expansion of $A^{\alpha}_{\bar{\tau}, \alpha}$.
      This equation implies that $(A^{\alpha}_{\bar{\tau}, \alpha})^{\perp}$ is a constant function on $T^2$, and hence $0$.
      Thus, $A_{\bar{\tau}}$ is $T^2$-invariant.
    \end{proof}
	\begin{Cor}\label{Cor:asymp. ortho. and hol. decomp. of inst.}
    	Set $\widetilde{\Gamma_{\bar{w}}} \in \Gamma((R,\infty)\times T^3,\mathrm{End}(V))$ as
    	$\widetilde{\Gamma_{\bar{w}}}(\mb{v}) := \mb{v}\Gamma_{\bar{w}}$.
    	For an eigenvalue $\alpha \in \mathbb{C}$ of $\Gamma_{\bar{w}}$,
		we take a subbundle $V_{\alpha}$ of $V|_{(R,\infty)\times T^3}$ to be $V_{\alpha} := \mathrm{Ker}(\widetilde{\Gamma_{\bar{w}}} - \alpha\mathrm{Id}_V)$.
		Then the decomposition $V|_{(R,\infty)\times T^3} = \bigoplus_{\alpha} V_{\alpha}$ is a holomorphic decomposition.
		Moreover, there exist $C,\delta>0$ such that for any $(t,x)\in\mathbb{R}\times T^3$, for any eigenvalues $\alpha\neq\beta$,
		for any $v_\alpha \in (V_{\alpha})_{(t,x)}$ and for any $v_\beta \in (V_{\beta})_{(t,x)}$,
		we have $|\i<v_\alpha,v_\beta>| < C\exp(-\delta t)|v_\alpha|\cdot|v_\beta|$.
    \end{Cor}
    \begin{Cor}\label{Cor:hol push down of asymp inst}
      Let $\pi:(0,\infty)\times T^3 \to (0,\infty)\times S^1$ be the projection.
      There exist a holomorphic Hermitian vector bundle $(E, \deebar_E, h_E)$
      on $(R,\infty)\times S^1$ and a holomorphic endomorphism $f \in \Gamma((R, \infty)\times S^1, \mathrm{End}(E))$
      such that the following holds:
      \begin{enumerate}[label=(\roman*)]
      	\item
        	$(V, \deebar_A)$ and $(\pi^{\ast}E, \pi^{\ast}(\deebar_E) + fd\bar{w})$ are isomorphic.
        \item
        	Under the isomorphism between $(V, \deebar_A)$ and $(\pi^{\ast}E, \pi^{\ast}(\deebar_E) + fd\bar{w})$,
        	we have the estimates
        	$
          		||h-\pi^{\ast}h_E||_{C^{2}([t,t+1]\times T^3)}
          		= O(\mathrm{exp}(-\delta t)),
			$
			where $\delta$ is a positive number, and the norm is induced by $h$.
        \item
        	For the Chern connection $\nabla_E$ of $(E,h_E,\deebar_{E})$,
        	we have $|F(\nabla_E)|_{h_E} = O(t^{-2})$.
        \item
        	$(E, \deebar_{E}, f)$ has a holomorphic and orthogonal decomposition
			\[(E, \deebar_{E}, h_E, f) = \bigoplus_{\alpha} (E_{\alpha}, \deebar_{E_{\alpha}}, h_{E_{\alpha}}, f_{\alpha}),\]
          	which is compatible with the decomposition $V = \bigoplus_{\alpha} V_{\alpha}$ in Corollary \ref{Cor:asymp. ortho. and hol. decomp. of inst.}.
      \end{enumerate}
    \end{Cor}
	\begin{proof}
		Let $(E,h_E)$ be a trivial Hermitian vector bundle on $(R,\infty)\times S^1$ and $\mb{e}$ be a orthonormal frame of $(E,h_E)$.
		We set $\deebar_E(\mb{e})=\mb{e}A_{\tau}$ and $f(\mb{e})=\mb{e}A_{\bar{w}}$.
		Then all conditions are satisfied by Proposition \ref{Prp:hol Asymp of Inst}.
	\end{proof}
  \subsection{Prolongation of $L^2$-finite instantons}\label{SubSec:Prolong of inst}
  	By following \cite{Ref:Moc1} and \cite{Ref:Moc3},
    we construct a polystable filtered bundles on $(\mathbb{P}^1\times T^2,\{0,\infty\}\times T^2)$
  	from $L^2$-finite instantons on $\mathbb{R}\times T^3 \simeq \mathbb{C}^{\ast}\times T^2$.
    \begin{Def}\label{Def:Prolongation}
    	Let $(X,g)$ be a K\"{a}hler manifold.
      Let $(V,h,A)$ be a holomorphic Hermitian bundle on $\Delta^{\ast} \times X$,
      where $\Delta^{\ast} = \{z \in \mathbb{C}\;|\;0<|z|<1\}$.
      Let $i:\Delta^{\ast}\times X\to\Delta\times X$ be the inclusion.
    	\begin{enumerate}
      	\item
        	$(V,h,A)$ is \textit{acceptable} if $|F(A)|$ is bounded on a neighborhood of $\{0\} \times X$,
          where the norm is induced by $h$ and the Poincar\'{e}-like metric $g + |z|^{-2}(\log|z|)^{-2}dzd\bar{z}$.
        \item
        	For any $a \in \mathbb{R}$,
			we define a (possibly non-coherent) $\mathcal{O}_{\Delta \times X}$-submodule $P_{a}V$ of $i_\ast V$ as follows:
			For any open subset $U\subset \Delta\times X$,
			a section $s \in \Gamma(U,i_{\ast}V)$ belongs to $\Gamma(U,P_{a}V)$ if and only if
			for any $p \in (\{0\}\times X) \cap U$ and for any $\eps>0$,
			an estimate $|s| = O(|z|^{-(a+\eps)})$ holds around $p$.
			We call $P_{\ast}V := \{P_{a}V\}_{a\in\mathbb{R}}$ the prolongation of $V$.
        \item
        	For a section $s \in \Gamma(X \times \Delta, P_{a}V)$,
          we define $\mathrm{ord}(s) \in \mathbb{R}$ to be
          \[
          	\mathrm{ord}(s) := \mathrm{min}\{ a' \in \mathbb{R} \mid s \in \Gamma(X \times \Delta, P_{a'}V)\}.
          \]
      \end{enumerate}
    \end{Def}
    \vspace{2mm}
    Mochizuki \cite{Ref:Moc3} proved the following theorem.
    \begin{Thm}\label{Thm:acceptable imp prolongable}
    	Let $(X,g), (V,h,A)$ be as in Definition \ref{Def:Prolongation}.
      	If $(V,h,A)$ is acceptable,
      	then $P_{\ast}V = \{P_{a}V\}_{a \in \mathbb{R}}$ forms a filtered bundle on $(\Delta\times X, \{0\}\times X)$.
    \end{Thm}
	Let $(V,h,A)$ be an $L^2$-finite instanton on $(-\infty,0)\times T^3$.
	Applying Corollary \ref{Cor:hol push down of asymp inst} to $(V,h,A)$,
	we take a positive number $R>0$, a holomorphic vector bundle $(E,\deebar_E, h_E)$ on $(-\infty,-R)\times S^1$ and a holomorphic endomorphism $f \in \Gamma((-\infty,-R)\times S^1, \mathrm{End}(E))$.
	For $r>0$,
	we set $\Delta(r):=\{z\in\mathbb{C}\mid |z|<\exp(-2\pi r)\}$
	and $\Delta(r)^{\ast} := \Delta(r)\setminus\{0\}$.
	Under the isomorphism $\Delta(R)^{\ast}\simeq (-\infty,R)\times S^1$ and $\Delta(0)^{\ast}\times T^2 \simeq (\infty,0)\times T^3$,
	we obtain the prolongation $P_{\ast}E$ and $P_{\ast}V$ from
	$(E,h_E,\deebar_{E})$ and $(V,h,A)$ respectively.
    \begin{Cor}\label{Cor:Inst is prolongable}
    	Both $(V,h,A)$ and $(E,h_E,\deebar_{E})$ are acceptable on $\Delta(0)^{\ast}\times T^2$ and $\Delta(R)^{\ast}$ respectively.
    	In particular, the prolongations $P_{\ast}V$ and $P_{\ast}E$ are filtered bundles on $(\Delta(0) \times T^2,\{0\}\times T^2)$ and $(\Delta(R),\{0\})$ respectively.
    \end{Cor}
    \begin{proof}
    	Under the coordinate change $z = \exp(2\pi(t + \sqrt{-1}x^{1}))$,
      	the Poincar\'{e} metric $|z|^{-2}(\log|z|)^{-2}dzd\bar{z}$ is written as $(dt^2 + (dx^1)^2) / t^2$.
      	By Corollary \ref{Cor:Asymp of Inst},
      	we have $|F(A)| = O(t^{-2})$,
      	where the norm is induced by $h$ and $g_{\mathbb{R}\times T^3}$.
      	Thus, $(V,h,A)$ is acceptable on $\Delta(0)^{\ast}\times T^2$.
      	By Corollary \ref{Cor:hol push down of asymp inst} (iii),
      	we also have $|F(\nabla_E)| = O(t^{-2})$,
      	and $(E,h_E,\deebar_{E})$ is acceptable on $\Delta(R)^{\ast}$.
      	Therefore, by Theorem \ref{Thm:acceptable imp prolongable},
      	$P_{\ast}V$ and $P_{\ast}E$ are filtered bundles on $(\Delta(0) \times T^2,\{0\}\times T^2)$ and $(\Delta(R),\{0\})$ respectively.
    \end{proof}
	Let us prove that the isomorphism $(V,\deebar_{A})\simeq (\pi^{\ast}E,\pi^{\ast}\deebar_{E} + \pi^{\ast}fd\bar{w})$ in Corollary \ref{Cor:hol push down of asymp inst}
	is extended over $\Delta(R) \times T^2$.
	\begin{Prp}\label{Prp:Rough property of prolong}
		For any $a \in \mathbb{R}$,
		we have the holomorphic isomorphism
		$P_{a}(V|_{\Delta(R)^{\ast} \times T^2}) \cong (\pi^{\ast}P_{a}E, \pi^{\ast}(\deebar_{P_{a}E}) + \pi^{\ast}(P_{a}f) d\bar{w})$
		constructed from the isomorphism in Corollary \ref{Cor:hol push down of asymp inst}.
		In particular, the vector bundle $P_{a}V|_{\{0\}\times T^2}$ is semistable of degree $0$.
    \end{Prp}
	\begin{proof}
		Let $W \subset \Delta(R)\times T^2$ be an open subset,
		and take any $s \in \Gamma\bigl(W,\left(\pi^{\ast}P_{a}E, \pi^{\ast}(\deebar_{P_{a}E}) + \pi^{\ast}(P_{a}f) d\bar{w}\right)\bigr)$.
		Then $s|_{W \setminus (\{0\} \times T^2)}$ is a holomorphic section of $(V,\deebar_A)$.
		Since $h$ and $\pi^{\ast}h_E$ are mutually bounded, $s$ is a holomorphic section of $P_{a}V$.
		
		Let $s'$ be a holomorphic section of $P_{a}V$ on $W$.
		We obtain a holomorphic section ${s'}|_{W \setminus (\{0\} \times T^2)}$ of $(\pi^{\ast}E, \pi^{\ast}\deebar_{E} + \pi^{\ast}f d\bar{w})$.
		Because of the definition of $P_{a}E$, $s'|_{W \cap (\mathbb{P}^1 \times \{w\})}$ is a holomorphic section of $P_{a}E$.
		Hence $s'$ is a holomorphic section of $(\pi^{\ast}P_{a}E, \pi^{\ast}\deebar_{P_{a}E} + \pi^{\ast}P_{a}f d\bar{w})$.
		
		Under this isomorphism,
		$P_{a}V|_{\{0\}\times T^2}$ is naturally isomorphic to 
		$\left( \pi^{\ast}(P_{a}E|_{\{0\}}), \deebar_{T^2} + \pi^{\ast}(P_{a}f|_{\{0\}}) d\bar{w} \right)$.
		Thus this vector bundle is semistable and of degree $0$.
	\end{proof}
	\begin{Cor}\label{Cor:Rough property of Gr}
		For any $a \in \mathbb{R}$, we have an isomorphism of vector bundles
		$\mathrm{Gr}_{a}(V) \simeq ( \mathrm{Gr}_{a}(E) \times T^2, \deebar_{T^2} + \mathrm{Gr}_{a}(f) d\bar{w})$ on $T^2$,
		where we regard the skyscraper sheaf $\mathrm{Gr}_{a}(E)$ and the endomorphism $\mathrm{Gr}_{a}(f)$
		as a vector space with an endomorphism.
	\end{Cor}
	Let $(V,h,A)$ be an $L^2$-finite instanton on $\mathbb{R}\times T^3$.
	We will denote by $P_{\ast\ast}V$ the associated filtered bundle on $(\mathbb{P}^1\times T^2,\{0,\infty\}\times T^2).$
    We prove that $P_{\ast\ast}V$ is polystable.
  	\begin{Thm}\label{Thm:inst is stable}
      The prolongation $P_{\ast\ast}V$ is polystable and $\mathrm{par\mathchar`-deg}(P_{\ast\ast}V)=0$.
      In particular, if $(V,h,A)$ is irreducible, then $P_{\ast\ast}V$ is stable.
    \end{Thm}
    Once we admit this theorem, we obtain the next corollary from Proposition \ref{Prp:vanish cohom of stable bundle}.
    \begin{Cor}
    	Let $p:\mathbb{P}^1\times T^2\to T^2$ be the projection.
    	Let $(V,h,A)$ be an irreducible $L^2$-finite instanton of rank $r>1$ on $\mathbb{R}\times T^3$.
      Then, we have
  	  $H^0(\mathbb{P}^1 \times T^2, P_{00}V\otimes p^{\ast}F) = H^2(\mathbb{P}^1 \times T^2, P_{<0<0}V\otimes p^{\ast}F) = 0$
      for any $F \in \mathrm{Pic}^{0}(T^2)$.
  	\end{Cor}
  
	\subsubsection{Norm estimate}
    As a preparation of the proof of Theorem \ref{Thm:inst is stable}, we show the following norm estimates.
    
    Let $(V,h,A)$ be an $L^2$-finite instanton on $(-\infty,0)\times T^3$ of rank $r$.
    Applying Corollary \ref{Cor:hol push down of asymp inst} to $(V,h,A)$,
    we take a positive number $R>0$,
    a holomorphic Hermitian vector bundle $(E, h_E, \deebar_E)$ on $(-\infty,-R)\times S^1\simeq \Delta(R)^{\ast}$ and a holomorphic endomorphism $f \in \Gamma(\Delta(R)^{\ast},\mathrm{End}(E))$.
    Let $P_{\ast}E$ denote the prolongation of $E$ on $\Delta(R)$.
    On the fiber $(P_{a}E)|_{0}$,
    the parabolic filtration $\{F_{c}(P_{a}E|_{0})\}_{a-1<c\leq a}$ is induced by the natural inclusion $P_{c}E \hookrightarrow P_{a}E$.
    Then an endomorphism $\mathrm{Gr}^{F}_{c}(f|_0)$ on $\mathrm{Gr}^{F}_{c}(P_aE|_{0})$ is induced by $f$,
    and the nilpotent part of $\mathrm{Gr}^{F}_{c}(f|_0)$ induces the weight filtration $\{W_k\mathrm{Gr}^{F}_{c}(P_aE|_{0})\}_{k\in\mathbb{Z}}$.
    For $b \in \mathrm{Gr}^{F}_{c}(P_{a}E|_{0})$,
    we denote by $\mathrm{deg}^{W}(b_i) \in \mathbb{Z}$ the degree of $b$ with respect to the weight filtration.
    For a holomorphic frame $\mb{b}=(b_1,\ldots,b_r)$ of $P_{a}E$ on $\Delta(R)$,
    $\mb{b}=(b_i)$ is compatible with the parabolic filtration and the weight filtration if the following conditions are satisfied:
    \begin{itemize}
    	\item 
    		For any $c\in (a-1,a]$,
    		$\{b_i|_{0} \mid \mathrm{ord}(b_i)\leq c\}$ forms a basis of $F_{c}(P_{a}E)$.
    	\item
    		For any $c\in (a-1,a]$ and any $k\in\mathbb{Z}$,
    		$\{[b_i] \mid \mathrm{ord}(b_i)=c,\;\mathrm{deg}^{W}([b_i])\leq k\}$
    		forms a basis of $W_k\mathrm{Gr}^{F}_{c}(P_aE)$,
    		where $[b_i]$ is the image of $b_i$ in $\mathrm{Gr}^{F}_{c}(P_aE|_{0})$.
    \end{itemize}
    \begin{Prp}\label{Prp:norm estimate}
    	Let $a\in \mathbb{R}$ and $\mb{b}=(b_1,\ldots,b_r)$ be a holomorphic frame of $P_{a}E$ on $\Delta(R)$ that is compatible with the parabolic filtration and the weight filtration.
      	We take a Hermitian metric $h_1$ on $E|_{\Delta(R)^{\ast}}$ given by $h_1(b_i, b_j) := \delta_{ij}|z|^{-2\mathrm{ord}(b_i)} (-\log|z|^2)^{\mathrm{deg}^{W}(b_i)}$.
      	Then $h_1$ and $h_E$ are mutually bounded.
    \end{Prp}
    \begin{proof}
    	According to \cite[Corollary 4.3]{Ref:Sim2},
      	we only need to prove that there exists $p>1$ such that we have $\bigl|F(E,h_E)+[fdz/z,f^{\dag}d\bar{z}/\bar{z}]\bigr|\in L^p(\Delta(R)^{\ast})$,
      	where $f^{\dag}$ is the adjoint of $f$ with respect to the metric $h_E$, and the norm is induced by $h_E$ and $dzd\bar{z}$.
      	We set a holomorphic Hermitian vector bundle $(\tilde{E},\deebar_{\tilde{E}},h_{\tilde{E}})$ on $\Delta(R)^{\ast}\times T^2$ by
      	$(\tilde{E},\deebar_{\tilde{E}},h_{\tilde{E}}) := (\pi^{\ast}E, \pi^{\ast}(\deebar_E) + fd\bar{w}, \pi^{\ast}h_E)$.
      	We write the curvature of the Chern connection of $(\tilde{E},\deebar_{\tilde{E}},h_{\tilde{E}})$ as 
      	$F(\tilde{E},h_{\tilde{E}}) = \tilde{F}_{z\bar{z}}dz\wedge d\bar{z} + \tilde{F}_{w\bar{w}}dw\wedge d\bar{w} +
      	\tilde{F}_{w\bar{z}}dw\wedge d\bar{z} + \tilde{F}_{z\bar{w}}dz\wedge d\bar{w}$.
      	Then we have
      	\begin{align*}
      		\pi^{\ast}\left( F(E,h_E) + [fdz/z , f^{\dag}d\bar{z}/\bar{z} ] \right)
        	&= \left(\tilde{F}_{z\bar{z}} + |z|^{-2}\tilde{F}_{w\bar{w}}\right) dz \wedge d\bar{z}.
      	\end{align*}
      	Since $(V,h,A)$ is an instanton, the ASD equation $F(A)_{z\bar{z}} = -|z|^{-2}F(A)_{w\bar{w}}$ holds,
      	where $F(A)_{z\bar{z}}$ and $F(A)_{w\bar{w}}$ are the components of the curvature $F(A)$.
      	Hence we have
      	\begin{align*}
      		\pi^{\ast}\left( F(E,h_E) + [fdz/z , f^{\dag}d\bar{z}/\bar{z} ] \right)
        	& = 
        	\left(
          		(F_{z\bar{z}} - F(A)_{z\bar{z}})
            	+ |z|^{-2}(F_{w\bar{w}} - F(A)_{w\bar{w}})
          	\right) dz \wedge d\bar{z}.
      	\end{align*}
      	Applying Corollary \ref{Cor:hol push down of asymp inst}, we obtain
      	\[ \left| \pi^{\ast}\left( F(E,h_E) + [fdz/z , f^{\dag}d\bar{z}/\bar{z} ] \right) \right| = O(|z|^{-2+\delta}), \]
      	where $\delta$ is a positive number.
      	Hence we have $\left|F(E,h_E) + [fdz/z , f^{\dag}d\bar{z}/\bar{z} ]\right| \in L^p(\Delta(R)^{\ast})$ for some $p>1$.
    \end{proof}
  	\subsubsection{Analytic degree for acceptable bundles on $(\mathbb{P}^1,\{0,\infty\})$}
  		As a preparation to prove Theorem \ref{Thm:inst is stable},
  		by following \cite{Ref:Moc1},
  		we introduce the notion of analytic degree of acceptable bundles on $(\mathbb{P}^1,\{0,\infty\})$,
  		and consider the relation between the parabolic degree and the analytic degree of acceptable bundles on $(\mathbb{P}^1,\{0,\infty\})$.

  		Let $(E, \deebar_E, h_E)$ be an acceptable holomorphic Hermitian vector bundle on $\mathbb{C}^{\ast}$.
  		We assume that for any non-zero holomorphic section $f$ of $E$ on a neighborhood of $0 \in \mathbb{P}^1$
  		there exist $C_1>0$ and $k_{0}(f) \in \mathbb{R}$ such that we have an estimate
  		\[
  			C_{1}^{-1}|z|^{-\mathrm{ord}_{0}(f)}(-\log|z|^2)^{k_{0}(f)} \leq |f|_{h_E} \leq C_{1}|z|^{-\mathrm{ord}_{0}(f)}(-\log|z|^2)^{k_{0}(f)}.
  		\]
  		We also assume a similar condition at $\infty \in \mathbb{P}^1$.
  		
  		For any holomorphic subbundle $P_{00}\mathcal{L}$ of $P_{00}E$,
  		we set $\mathcal{L} := P_{00}\mathcal{L}|_{\mathbb{C}^{\ast}}$ and $h_{\mathcal{L}}$ be the Hermitian metric induced by $h$.
  		We define the analytic degree and the parabolic degree of $\mathcal{L}$ by
  		\[
  			\mathrm{deg}(\mathcal{L}, h_E) := 
  				\sqrt{-1}\int_{\mathbb{C}^{\ast}} \mathrm{Tr}(\Lambda F(h_{\mathcal{L}})) d\mathrm{vol}_{\mathbb{C}^{\ast}}
  		\]
  		and
  		\[
  			\mathrm{par\mathchar`-deg}(P_{\ast\ast}\mathcal{L}) := \int_{P^1} \mathrm{par\mathchar`-c_1}(P_{\ast\ast}\mathcal{L}),
  		\]
  		where $P_{\ast\ast}\mathcal{L}$ is the strict filtered subbundle given by $P_{ab}\mathcal{L} := \mathcal{L} \cap P_{ab}E$.
  		\begin{Prp}\label{Prp:ana-deg is par-deg in P^1}
  			We have the equality $\mathrm{deg}(\mathcal{L}, h_E) = 2\pi\ \mathrm{par\mathchar`-deg}(P_{\ast\ast}\mathcal{L})$.
  		\end{Prp}
  		\begin{proof}
  			By considering $\bigwedge^{\mathrm{rank}(\mathcal{L})}E$ and $\mathrm{det}(\mathcal{L})$ instead of $E$ and $\mathcal{L}$,
  			we may assume $\mathrm{rank}(\mathcal{L}) = 1$.
  			Let $e_0 , e_{\infty}$ be holomorphic frames of $P_{00}\mathcal{L}$ on a neighborhood of $0 , \infty \in \mathbb{P}^1$.
  			We set a smooth function $\psi : \mathbb{C}^{\ast} \to \mathbb{R}$ satisfying
  			\begin{eqnarray*}
  				\psi(z)=
  				\left\{
  				\begin{array}{ll}
  					\mathrm{ord}(e_0)\log|z|^2 & \mbox{(on a small neighborhood of $0 \in \mathbb{P}^1$)}\\
  					-\mathrm{ord}(e_{\infty})\log|z|^2 & \mbox{(on a small neighborhood of $\infty \in \mathbb{P}^1$)}
  				\end{array}
  				\right.
  			\end{eqnarray*}
  			and set a metric $h'_E$ on $(E,\deebar_{E})$ by $h'_E := h_{E}e^{\psi}$.
  			We will denote by $P'_{\ast\ast}E, P'_{\ast\ast}\mathcal{L}$ the prolongation of $E$ and $\mathcal{L}$ with respect to the metric $h'_E$ respectively.
  			Then we have
  			\begin{eqnarray*}
  				\mathrm{par\mathchar`-deg}(P_{\ast\ast}\mathcal{L}) = \mathrm{par\mathchar`-deg}(P'_{\ast\ast}\mathcal{L}) - \mathrm{ord}(e_0) -\mathrm{ord}(e_{\infty})\\
  				\mathrm{deg}(\mathcal{L},h_E) = \mathrm{deg}(\mathcal{L},h'_E) - 2\pi\ \mathrm{ord}(e_0) + 2\pi\ \mathrm{ord}(e_{\infty}).
  			\end{eqnarray*}
  			Therefore, by replacing $h_E$ with $h'_E$,
  			we may assume $\mathrm{par\mathchar`-deg}(P_{\ast\ast}\mathcal{L}) = \mathrm{deg}(P_{00}\mathcal{L})$ and $\mathrm{ord}(e_0)=\mathrm{ord}(e_{\infty})=0$.
  			We take another metric $h_{1,\mathcal{L}}$ of $P_{00}\mathcal{L}$ satisfying $h_{1,\mathcal{L}}(e_0, e_0) = h_{1,\mathcal{L}}(e_{\infty}, e_{\infty}) = 1$.
  			We take a smooth function $\varphi:\mathbb{C}^{\ast} \to \mathbb{R}$ as $h_{1,\mathcal{L}} = e^{\varphi}h_{\mathcal{L}}$.
  			By the definition of $h_{1,\mathcal{L}}$, we have $\mathrm{deg}(P_{00}L) = (2\pi)^{-1}\sqrt{-1}\int_{\mathbb{P}^1} F(h_{1,\mathcal{L}})$.
  			We consider the following lemma.
  			\begin{Lem}\label{Lem:integrability on C ast}
  				The $2$-form $\deebar\partial\varphi$ is integrable on $\mathbb{C}^{\ast}$,
  				and we have $\int_{\mathbb{C}^{\ast}} \deebar\partial\varphi = 0$.
  			\end{Lem}
  			Once we admit this lemma,
  			then we obtain the desired equality $\mathrm{deg}(L,h_E) = 2\pi \mathrm{par\mathchar`-deg}(P_{\ast\ast}\mathcal{L})$
  			because we have $F(h_{\mathcal{L}}) = F(h_{1,\mathcal{L}}) + \deebar\partial\varphi$.
  		\end{proof}
  		\begin{proof}[{\upshape \textbf{proof of Lemma \ref{Lem:integrability on C ast}}}]
  			We set $\eta := z^{-1}$ and take a metric $g_1$ on $\mathbb{P}^1$ satisfying
  			$g_1 = dzd\bar{z}$ on a small neighborhood of $z=0$ and $g_1 = d\eta d\bar{\eta}$ on a small neighborhood of $\eta=0$.
  			We take a smooth function $\rho : \mathbb{R} \to [0,1]$ satisfying
  			$\rho(t)=1 \;( |t| < 1/2)$ and $\rho(t)=0\;(|t|>1)$ 
  			and set $\chi_{N}: \mathbb{C}^{\ast} \to [0,1]$ as $\chi_{N}(z) := \rho(N^{-1} \log |z|^2)$ for $N \in \mathbb{N}$.
  			Since $\chi_{N}$ is a compact supported function,
  			we have
  			\begin{equation}
  			0 
  			= \int_{\mathbb{C}^{\ast}}\deebar\partial(\chi_{N}\varphi)
  			= 
  			\int_{\mathbb{C}^{\ast}}\deebar\partial(\chi_{N})\varphi
  			+ \int_{\mathbb{C}^{\ast}}\deebar\chi_{N} \cdot \partial\varphi
  			- \int_{\mathbb{C}^{\ast}}\partial\chi_{N} \cdot \deebar\varphi
  			+ \int_{\mathbb{C}^{\ast}}\chi_{N} \deebar\partial(\varphi). \label{Eqn:Integral of phi}
  			\end{equation}
  			Here we consider the following lemma.
  			\begin{Lem}\label{Lem:dominating int.}
  				The integrands of the first to third term of rhs of (\ref{Eqn:Integral of phi}) are dominated by an integrable functions independent of $N$.
  				In particular,
  				the dominated convergence theorem shows
  				\[
  				\lim_{N \to \infty} \left(
  				\int_{\mathbb{C}^{\ast}}\deebar\partial(\chi_{N})\varphi
  				+ \int_{\mathbb{C}^{\ast}}\deebar\chi_{N} \cdot \partial\varphi
  				- \int_{\mathbb{C}^{\ast}}\partial\chi_{N} \cdot \deebar\varphi
  				\right) = 0.
  				\]
  			\end{Lem}
  			If we admit this lemma,
  			then we have $\lim_{N \to \infty}  \int_{\mathbb{C}^{\ast}}\chi_{N} \deebar\partial(\varphi) = 0$.
  			Therefore, $\deebar\partial(\varphi)$ is integrable, and we have $\int_{\mathbb{C}^{\ast}} \deebar\partial\varphi = 0$.
  		\end{proof}
  		\begin{proof}[{\upshape \textbf{proof of Lemma \ref{Lem:dominating int.}}}]
  			We calculate the derivatives of $\chi_{N}$, then
  			\begin{align*}
  			&\partial \chi_{N} = N^{-1} \rho'(N^{-1} \log |z|^2)z^{-1}dz \\
  			&\deebar \chi_{N} = N^{-1} \rho'(N^{-1} \log |z|^2)\bar{z}^{-1}d\bar{z} \\
  			&\deebar\partial \chi_{N} = - N^{-2} \rho''(N^{-1} \log |z|^2)|z|^{-2}dz \wedge d\bar{z}.
  			\end{align*}
  			Thus, there exists $C_2 > 0$ independent from $N$ such that on a small neighborhood of $z=0$ we have
  			\begin{align}
  			&|\partial \chi_{N}|_{g_1} = |\deebar \chi_{N}|_{g_1} \;\leq\; C_{2}|z|^{-1}(-\log|z|)^{-1} \label{Eqn:chi Est around 0} \\
  			&|\deebar\partial \chi_{N}|_{g_1} \;\leq\;  C_{2}|z|^{-2}(-\log|z|)^{-2} \label{Eqn:deldel chi Est around 0},
  			\end{align}
  			and on a small neighborhood of $\eta=0$ we also have
  			\begin{align}
  			&|\partial \chi_{N}|_{g_1} = |\deebar \chi_{N}|_{g_1} \;\leq\; C_{2}|\eta|^{-1}(-\log|\eta|)^{-1} \label{Eqn:chi Est around infty} \\
  			&|\deebar\partial \chi_{N}|_{g_1} \;\leq\;  C_{2}|\eta|^{-2}(-\log|\eta|)^{-2}. \label{Eqn:deldel chi Est around infty}
  			\end{align}
  			Hence by (\ref{Eqn:deldel chi Est around 0}), (\ref{Eqn:deldel chi Est around infty}) and the assumption of the norm of $e_0,e_{\infty}$,
  			there exist an integrable function dominating the first term of (\ref{Eqn:Integral of phi}) and independent of $N$.
  			To estimate the other terms, we prove the next lemma.
  			\begin{Lem}\label{Lem:Est of partial varphi}
  				$|\partial(\varphi)|_{g_1} = |\deebar(\varphi)|_{g_1}$ is an $L^2$-function on $(\mathbb{P}^1,g_1)$.
  			\end{Lem}
  			If we suppose that the lemma is true, by (\ref{Eqn:chi Est around 0}) and (\ref{Eqn:chi Est around infty})
  			we obtain an integrable function dominating the second and third terms of (\ref{Eqn:Integral of phi}) independent of $N$.
  			Hence the proof of Proposition \ref{Prp:ana-deg is par-deg} is complete.
  		\end{proof}
  		\begin{proof}[{\upshape \textbf{proof of Lemma \ref{Lem:Est of partial varphi}}}]
  			If we prove $|\partial(\varphi)|_{g_1}$ is an $L^2$-function on a neighborhood of $z=0$,
  			then the same proof works for a neighborhood of $\eta=0$.
  			Thus, we only needs to show $|\partial(\varphi)|_{g_1} \in L^2(\Delta^{\ast},g_1)$,
  			where $\Delta^{\ast} = \{z\in\mathbb{C}\mid 0<|z|<1\}$.
  			
  			Let $\nabla$ be the Chern connection of $(E,h_E,\deebar_{E})$.
  			By definition of $\varphi$ we have $\exp(-2\varphi)=|e_0|_{h_E}^2$.
  			Therefore we obtain $|\partial\exp(-2\varphi)| = |2\partial\varphi|\exp(-2\varphi)=|h_E(\nabla_{z}e_0,e_0)|\leq|e_0|_{h_E}\cdot|\nabla_{z}e_0|_{h_E}$,
  			where we set $\nabla =: \nabla_z dz + \nabla_{\bar{z}}d\bar{z}$.
  			Hence we have $|2\partial\varphi| \leq |e_0|_{h_E}^{-1}\cdot |\nabla_{z}e_0|_{h_E}$.
  			By the norm estimate of $e_0$, it suffices to prove that $|(-\log |z|^2)^{-k}|\nabla_z(e_0)|_{h_E} \in L^2(\Delta)$, where $k = k_{0}(e_0)$.
  			We take a smooth function $a: \mathbb{R} \to [0,1]$ satisfying $a(t)=0\;(t > 1)$, $a(t)=1\;(t < 1/2)$
  			and the condition that $a^{1/2}$ and $a' \cdot a^{-1/2}$ are also smooth.
  			We set a function $b_{N}:\Delta^{\ast} \to \mathbb{R}$ as \[b_{N}(z) := \left(1-a(-\log|z|^2) \right) \cdot a(-N^{-1}\log|z|^2)\]
  			for $N \in \mathbb{N}$.
  			Then $\partial(b_{N}) \cdot b_{N}^{-1/2}$ is a smooth function on $\Delta^{\ast}$ because of the definition of $a$.
  			Moreover, there exists $C_3 > 0$ such that we have
  			$|\partial(b_{N}) b_{N}^{-1/2}|_{g_1} \leq C_3 |z|^{-1}(-\log|z|^2)^{-1}$.
  			We consider the following integral.
  			\begin{align*}
  			&\int_{\Delta^{\ast}} b_{N}\cdot h_E(\nabla_{z} e_0, \nabla_{z} e_0)(-\log|z|^2)^{-2k} d\mathrm{vol} = \\
  			&\;\;\;\;-\int_{\Delta^{\ast}} \partial(b_{N})\cdot h_{E}(e_0, \nabla_{z} e_0)(-\log|z|^2)^{-2k} d\mathrm{vol} \\
  			&\;\;\;\; -\int_{\Delta^{\ast}} b_{N}\cdot h_{E}(e_0, F(h_E) e_0)(-\log|z|^2)^{-2k} d\mathrm{vol} \\
  			&\;\;\;\; +\int_{\Delta^{\ast}} b_{N}\cdot h_{E}(e_0, \nabla_{z} e_0)\cdot (-2k)(-\log|z|^2)^{-2k-1}z^{-1} d\mathrm{vol},
  			\end{align*}
  			where $d\mathrm{vol}=\sqrt{-1}dz\wedge d\bar{z}$. 
  			We have the following estimate on the first term of rhs.
  			\[
  			\left| \partial(b_{N})\cdot h_{E}(e_0, \nabla_{z} e_0)(-\log|z|^2)^{-2k} \right|
  			\leq \left( C_{3}C_{1}|z|^{-1}(\log|z|^2)^{-1} \right) \left( b^{1/2}(z)\cdot |\nabla_{z}e_0|_{h_E}(-\log|z|^2)^{-k} \right).
  			\]
  			For the second term, because of $(E,\deebar_{E},h_{E})$ is acceptable, there exists $C_4>0$ such that we have
  			\[
  			\left| b_{N}\cdot h_{E}(e_0, F(h_E) e_0)(-\log|z|^2)^{-2k} \right|
  			\leq C_{4} |z|^{-2}(-\log|z|^{2})^{-2}.
  			\]
  			For the third term, we also have
  			\begin{align*}
  			& \left| b_{N}\cdot h_{E}(e_0, \nabla_{z} e_0)\cdot (-2k)(-\log|z|^2)^{-2k-1}z^{-1} \right| \\
  			&\;\;\;\;\leq \left( C_1 b_N^{1/2} \cdot |z|^{-1}(-\log|z|^2)^{-1} \right)
  			\left(b_{N}^{1/2}|\nabla_{z}e_0|_{h_E}(-\log|z|^{2})^{-k} \right).
  			\end{align*}
  			Therefore, there exist $C_5, C_6 > 0$ such that
  			\begin{eqnarray*}
  				\int_{\Delta^{\ast}} b_{N}\cdot h_E(\nabla_{z} e_0, \nabla_{z} e_0)(-\log|z|^2)^{-2k} d\mathrm{vol}\\
  				\,\ \leq C_5 + C_6 \left(\int_{\Delta^{\ast}} b_{N}\cdot h_E(\nabla_{z} e_0, \nabla_{z} e_0)(-\log|z|^2)^{-2k} d\mathrm{vol} \right)^{1/2}
  				<\infty.
  			\end{eqnarray*}
  			Thus, we obtain
  			\[
  			\int_{\Delta^{\ast}} b_{N}\cdot h_E(\nabla_{z} e_0, \nabla_{z} e_0)(-\log|z|^2)^{-2k} d\mathrm{vol} < C_7,
  			\]
  			where $C_7$ is a constant independent of $N$.
  			Therefore, we conclude $|\partial(\varphi)|_{g_1} \in L^2(\mathbb{C}^{\ast},g_1)$.
  		\end{proof}
  	\subsubsection{Analytic degree and parabolic degree on $\mathbb{R}\times T^3$}
    	In order to prove Theorem \ref{Thm:inst is stable},
      	we also consider the analytic degree of holomorphic subsheaves of $L^2$-finite instantons on $\mathbb{R}\times T^3$ by following \cite{Ref:Moc1}.
  	  	\begin{Def}
      		Let $(V,h,A)$ be an $L^2$-finite instanton on $\mathbb{R}\times T^3 = \mathbb{C}^{\ast}\times T^2$.
        	Let $\mathcal{F}$ be a saturated $\mathcal{O}_{\mathbb{C}^{\ast} \times T^2}$-submodule of $(V,\deebar_A)$
        	and $h_{\mathcal{F}}$ the induced metric of smooth part of $\mathcal{F}$.
        	We define the analytic degree of $\mathcal{F}$ by
        	\[
        		\mathrm{deg}(\mathcal{F},h) := 
        		\sqrt{-1}\int_{\mathbb{C}^{\ast} \times T^2} \mathrm{Tr}(\Lambda F(h_{\mathcal{F}})) d\mathrm{vol}_{\mathbb{C}^{\ast} \times T^2},
        	\]
        	where $d\mathrm{vol}_{\mathbb{C}^{\ast} \times T^2}$ is the volume form
        	with respect to the Riemannian metric $|z|^{-2}dzd\bar{z} + dwd\bar{w}$.
        	By the Chern-Weil formula in \cite{Ref:Sim},
        	this can be written as
        	\[
        		\mathrm{deg}(\mathcal{F},h) = 
        			- \int_{\mathbb{C}^{\ast} \times T^2} |\deebar \pi|^2_{h} d\mathrm{vol}_{\mathbb{C}^{\ast} \times T^2},
        	\]
        	where $\pi:V \to \mathcal{F}$ is the orthogonal projection.
        \end{Def}

    	\begin{Lem}\label{Lem:finite degree equivalence}
      	Let $\mathcal{F}$ be a saturated subsheaf of an $L^2$-finite instanton $(V,h,A)$ on $\mathbb{R}\times T^3$.
      	Then, $\mathrm{deg}(\mathcal{F},h)$ is finite if and only if the following conditions are satisfied.
    	  \begin{enumerate}[label=(\roman*)]
    	  	\item\label{Enum:extension}
          	$\mathcal{F}$ can be extended to a saturated subsheaf $P_{00}\mathcal{F}$ of $P_{00}V$.
        	\item\label{Enum:deg_T^2 equals 0}
          	For any $z\in \mathbb{C}^{\ast}$, we have $\mathrm{deg}(\mathcal{F}|_{\{z\} \times T^2}) = 0$.
      	\end{enumerate}
    	\end{Lem}
    	\begin{proof}
      	Assume that $\mathrm{deg}(\mathcal{F},h)$ is finite.
        We write $\deebar = \deebar_{\mathbb{C}^{\ast}} + \deebar_{T^2}$.
        On one hand we have
    	  \[
    	  	\int_{T^2}d\mathrm{vol}_{T^2} \int_{\mathbb{C}^{\ast}} |\deebar_{\mathbb{C}^{\ast}}\pi|^2 d\mathrm{vol}_{\mathbb{C}^{\ast}}
    	    \leq \int_{T^2}d\mathrm{vol}_{T^2} \int_{\mathbb{C}^{\ast}} |\deebar\pi|^2 d\mathrm{vol}_{\mathbb{C}^{\ast}}
    	    = -\mathrm{deg}(\mathcal{F},h)<\infty.
    	  \]
        Therefore, there exists a measurable subset $A \subset T^2$ such that $|A| = |T^2|$ and
        $\int_{\mathbb{C}^{\ast} \times \{w\}} |\deebar_{\mathbb{C}^{\ast}}\pi|^2 d\mathrm{vol}_{\mathbb{C}^{\ast}} < \infty$
        holds for any $w \in A$.
        According to \cite[Lemma 10.5, Lemma 10.6]{Ref:Sim},
        this is equivalent to the condition that $\mathcal{F}|_{\mathbb{C}^{\ast} \times \{w\}}$ can be extended to a saturated subsheaf of
        $P_{00}V|_{\mathbb{P}^1 \times \{w\}}$.
        Moreover, since $A$ is a thick subset of $T^2$,
        by \cite[Theorem 4.5]{Ref:Siu} $\mathcal{F}$ can be extended to a saturated subsheaf of $P_{00}V$.
        On the other hand,
        we have
        \[ \int_{\mathbb{C}^{\ast}} d\mathrm{vol}_{\mathbb{C}^{\ast}} \int_{T^2} |\deebar_{T^2}\pi|^2 d\mathrm{vol}_{T^2}
    	  	\leq -\mathrm{deg}(\mathcal{F},h)<\infty.
    	  \]
        Hence there exists a sequence $\{z_i\}$ on $\mathbb{C}^{\ast}$ such that
        $z_i \to \infty$ and $ \int_{\{z_i\} \times T^2} |\deebar_{T^2}\pi|^2 d\mathrm{vol}_{T^2} \to 0$.
        We also have $|F(A)|_{\{z_i\}\times T^2} \to 0$ by Corollary \ref{Cor:Asymp of Inst}.
        Therefore, we have $\mathrm{deg}(\mathcal{F}|_{\{z_i\} \times T^2}) \to 0$ because of the Chern-Weil formula
        \begin{align}
    	  	\mathrm{deg}(\mathcal{F}|_{\{z\} \times T^2})
    	    	=& \sqrt{-1}\int_{\{z\} \times T^2} \mathrm{Tr}\left( \Lambda_{T^2}F(h_{\mathcal{F}}) \right) d\mathrm{vol}_{T^2}\nonumber \\
    	      =& \sqrt{-1}\int_{\{z\} \times T^2} \mathrm{Tr}\left( \pi\Lambda_{T^2}F(A) \right) d\mathrm{vol}_{T^2}
    	      	- \int_{\{z\} \times T^2} |\deebar_{T^2}\pi|^2 d\mathrm{vol}_{T^2}.\label{Eqn:Gauss-Codazzi in T^2}
    	  \end{align}
      	Since $\mathrm{deg}(\mathcal{F}|_{\{z\} \times T^2})$ is a continuous and $2\pi\mathbb{Z}$-valued function,
        we conclude $\mathrm{deg}(\mathcal{F}|_{\{z\} \times T^2}) = 0$ for any $z \in \mathbb{C}^{\ast}$.
      	
        Conversely, We assume \ref{Enum:extension} and \ref{Enum:deg_T^2 equals 0}.
        We have
        \begin{align*}
        	-\mathrm{deg}(\mathcal{F},h)
          &= \int_{\mathbb{C}^{\ast}\times T^2} (|\deebar_{\mathbb{C}^{\ast}}\pi|^2 + |\deebar_{T^2}\pi|^2)d\mathrm{vol}_{\mathbb{C}^{\ast}\times T^2}\\
          &= \int_{\mathbb{C}^{\ast}\times T^2}|\deebar_{\mathbb{C}^{\ast}}\pi|^2 d\mathrm{vol}_{\mathbb{C}^{\ast}\times T^2}
          	+\int_{\mathbb{C}^{\ast}\times T^2}|\deebar_{T^2}\pi|^2 d\mathrm{vol}_{\mathbb{C}^{\ast}\times T^2}.
        \end{align*}
        Here we use \ref{Enum:deg_T^2 equals 0} and (\ref{Eqn:Gauss-Codazzi in T^2}), then we obtain
        \begin{align*}
          -\mathrm{deg}(\mathcal{F},h)
          &= \int_{\mathbb{C}^{\ast}\times T^2}|\deebar_{\mathbb{C}^{\ast}}\pi|^2 d\mathrm{vol}_{\mathbb{C}^{\ast}\times T^2}
          	+\sqrt{-1}\int_{\mathbb{C}^{\ast}\times T^2}\mathrm{Tr}\left( \pi\Lambda_{T^2}F(A) \right) d\mathrm{vol}_{\mathbb{C}^{\ast}\times T^2}\\
          &\leq \int_{\mathbb{C}^{\ast}\times T^2}|\deebar_{\mathbb{C}^{\ast}}\pi|^2 d\mathrm{vol}_{\mathbb{C}^{\ast}\times T^2} + ||F(A)||_{L^1(\mathbb{R}\times T^3)}.
        \end{align*}
        By Proposition \ref{Prp:norm estimate},
        $(V, \deebar_A, h)|_{\mathbb{C}^\ast \times \{w\}}$ satisfies the assumption in Proposition \ref{Prp:ana-deg is par-deg in P^1} for any $w\in T^2$.
        Hence we have
        \begin{align*}
          -\mathrm{deg}(\mathcal{F},h)
          	&= -2\pi\int_{T^2}\mathrm{par\mathchar`-deg}((P_{\ast\ast}\mathcal{F})|_{\mathbb{C}^{\ast}\times\{w\}})d\mathrm{vol}_{T^2}
          		+\sqrt{-1}\int_{\mathbb{C}^{\ast}\times T^2}\mathrm{Tr}\left( \pi\Lambda_{\mathbb{C}^{\ast}}F(A) \right) d\mathrm{vol}_{\mathbb{C}^{\ast}\times T^2}\\
          	 	&\ \ \ +||F(A)||_{L^1(\mathbb{R}\times T^3)}\\
          &\leq -2\pi\int_{T^2}\mathrm{par\mathchar`-deg}((P_{\ast\ast}\mathcal{F})|_{\mathbb{C}^{\ast}\times\{w\}})d\mathrm{vol}_{T^2}
          			+2||F(A)||_{L^1(\mathbb{R}\times T^3)}.
        \end{align*}
        Then $\mathrm{par\mathchar`-deg}((P_{\ast\ast}\mathcal{F})|_{\mathbb{C}^{\ast}\times\{w\}})$ is a constant on $T^2$,
        and we have $||F(A)||_{L^1(\mathbb{R}\times T^3)}<\infty$ by Corollary \ref{Cor:Asymp of Inst}.
        Therefore we obtain $0\leq -\mathrm{deg}(\mathcal{F},h)<\infty$.
    	\end{proof}
    	\begin{Prp}\label{Prp:ana-deg is par-deg}
      	Let $\mathcal{F}$ be a saturated subsheaf of an $L^2$-finite instanton $(V,h,A)$ on $\mathbb{R}\times T^3$.
      	If $\mathrm{deg}(\mathcal{F},h)$ is finite,
        then we have $\mathrm{deg}(\mathcal{F},h) = 2\pi\mathrm{Vol}(T^2)\cdot\mathrm{par\mathchar`-deg}(P_{\ast\ast}\mathcal{F})$,
        where $P_{\ast\ast}\mathcal{F}$ is a filtered subsheaf defined by $P_{ab}\mathcal{F} := P_{ab}V \cap \mathcal{F}$.
    	\end{Prp}
    	\begin{proof}
      	Since $\mathrm{deg}(\mathcal{F},h)$ is finite,
        $\mathcal{F}$ can be extended to the saturated subsheaf $P_{00}\mathcal{F}$ of $P_{00}V$ by Lemma \ref{Lem:finite degree equivalence}.
        We denote by $B\subset T^2$ the set of all $w \in T^2$ that  $P_{00}\mathcal{F}|_{\mathbb{P}^1 \times \{w\}}$ is a subbundle of $P_{00}V|_{\mathbb{P}^1 \times \{w\}}$.
        Since $T^2$ is compact and $P_{00}\mathcal{F}$ is saturated, $T^2 \setminus B$ is a finite subset.
        By applying \ref{Enum:deg_T^2 equals 0} of Lemma \ref{Lem:finite degree equivalence},
        we have
    	  \begin{align*}
    	  	\mathrm{deg}(\mathcal{F},h)
    	  	&=\sqrt{-1}\int_{\mathbb{C}^{\ast} \times T^2} \Lambda_{\mathbb{C}^{\ast}} F(h_{\mathcal{F}}) d\mathrm{vol}_{\mathbb{C}^{\ast} \times T^2}\\
    	    &=\int_B 	\mathrm{deg}(\mathcal{F}|_{\mathbb{C}^{\ast} \times \{w\}}, h|_{\mathbb{C}^{\ast} \times \{w\}}) d\mathrm{vol}_{T^2}\\
          &=\int_B \mathrm{deg}(\mathcal{F}|_{\mathbb{C}^{\ast} \times \{w\}}, h|_{\mathbb{C}^{\ast} \times \{w\}}) d\mathrm{vol}_{T^2}.
    	  \end{align*}
        By Proposition \ref{Prp:norm estimate},
        $(V, \deebar_A, h)|_{\mathbb{P}^1 \times \{w\}}$ satisfy the assumption in Proposition \ref{Prp:ana-deg is par-deg in P^1}.
        Hence we have
        \[
    	  	\mathrm{deg}(\mathcal{F},h)
           = \int_B \mathrm{par\mathchar`-deg}(P_{\ast\ast}(F)|_{\mathbb{P}^1 \times \{w\}}) d\mathrm{vol}_{T^2}
           = 2\pi|T^2|\mathrm{par\mathchar`-deg}(P_{\ast\ast}\mathcal{F}).
        \]
    	  This is the desired equality.
      \end{proof}
  	\subsubsection{Proof of Theorem \ref{Thm:inst is stable}}
    	By Lemma \ref{Lem:rank1 inst is flat},
      $(\mathrm{det}(V), \mathrm{det}(h), \mathrm{Tr}(A))$ is a flat Hermitian line bundle.
      Hence we have $\mathrm{par\mathchar`-deg}(P_{\ast\ast}V)= \mathrm{par\mathchar`-deg}(P_{\ast\ast}(\mathrm{det}(V))) =0$.
      It is proved in Proposition \ref{Prp:Rough property of prolong} that
      $P_{ab}V|_{\{0\} \times T^2}$ and $P_{ab}V|_{\{\infty\} \times T^2}$ are semistable vector bundles of degree $0$ for any $a,b \in \mathbb{R}$.
      Let $P_{\ast\ast}\mathcal{F}$ be a filtered subsheaf of $P_{\ast\ast}V$ satisfying $0<\mathrm{rank}(\mathcal{F}) < \mathrm{rank}(V)$
      and (\ref{Enum:gr semistable}) in Definition \ref{Def:stable filtered bundle}.
      We may assume that $P_{00}\mathcal{F}$ is saturated.
      We set $\mathcal{F} := P_{\ast\ast}\mathcal{F}|_{\mathbb{C}^{\ast} \times T^2}$.
      Let $U\subset \mathbb{R}\times T^3$ be the maximal open subset such that $\mathcal{F}|_U$ is a subbundle of $V|_U$.
      Since $P_{00}\mathcal{F}$ is a saturated subsheaf of  $P_{00}V$,
      $(\mathbb{R}\times T^3)\setminus U$ is a finite subset.
      By Lemma \ref{Lem:finite degree equivalence} and Proposition \ref{Prp:ana-deg is par-deg},
      we have $\mathrm{deg}(\mathcal{F},h) = \mathrm{par\mathchar`-deg}(P_{\ast\ast}\mathcal{F})$.
      Therefore, $\mathrm{par\mathchar`-deg}(P_{\ast\ast}\mathcal{F}) \leq 0$ holds by the Chern-Weil formula.
      Moreover, if $\mathrm{par\mathchar`-deg}(P_{\ast\ast}\mathcal{F})=0$ holds, then we have $\deebar\pi =0$.
      Hence $\mathcal{F}|_U$ and $(\mathcal{F}|_U)^{\perp}$ become instantons by the induced metric from $h$,
      and  $(V,h,A)|_U = (\mathcal{F}|_U)\oplus (\mathcal{F}|_U)^{\perp}$ is a decomposition as instantons.
      Moreover, this decomposition is invariant under parallel transports.
      Thus it can be extended to the decomposition on whole $\mathbb{R}\times T^3$.
      By repeating these arguments, we can prove that $P_{\ast\ast}V$ is polystable.
      In particular, if $(V,h,A)$ is irreducible, then $P_{\ast\ast}V$ is stable.
			\qed
  \subsection{Some properties of $\mathrm{Gr}_{a}(P_{\ast}V)$}\label{SubSec:Prop of Gr}
  	Let $(V,h,A)$ be an $L^2$-finite instanton on $(-\infty,0)\times T^3$.
    Applying Proposition \ref{Prp:hol Asymp of Inst} to $(V,h,A)$,
    we take a positive number $R>0$,
    a $C^2$-frame $\mb{v}=(v_i)$ of $V$ on $(-\infty,-R)\times T^3$ and a model solution $(\Gamma,N)$ of the Nahm equation.
  	Assume that $\Gamma = (\Gamma_i)$ are diagonal, and
    we will denote by $\mb{v}_{\alpha}$ the subset of $\mb{v}$
    corresponding to an eigenvalue $\alpha \in \mathbb{C}$ of $\Gamma_{\bar{w}}$.
    Applying Corollary \ref{Cor:hol push down of asymp inst} to $(V,h,A)$,
    we take a holomorphic vector bundle $(E,\deebar_E, h_E)= \bigoplus_{\alpha}(E_{\alpha},\deebar_{E_{\alpha}},h_{E_{\alpha}})$ on $(-\infty,-R)\times S^1$
    and a holomorphic endomorphism $f=\bigoplus_{\alpha}f_{\alpha} \in \Gamma((-\infty,-R)\times S^1, \mathrm{End}(E))$.
    Let $\mb{e}$ (resp. $\mb{e}_{\alpha}$) be the $C^{\infty}$-frame of $E$ (resp. $E_{\alpha}$) which corresponds to $\mb{v}$ (resp. $\mb{v}_{\alpha}$).
    We will denote by $P_{\ast}V$ and $P_{\ast}E$ the prolongations of $V$ over $\Delta(0) \times T^2$ and $E$ over $\Delta(R)$ respectively,
    where $\Delta(s) := \{z \in \mathbb{C}\mid |z|< \exp(-2\pi s)\}$ and $\Delta(s)^{\ast} := \Delta(s) \setminus \{0\}$.
    For $a\in\Par(P_{\ast}E_{\alpha})$,
    we have the weight filtration $\{W_{i}\mathrm{Gr}_{a}(P_{\ast}E_{\alpha})\}_{i\in\mathbb{Z}}$ on $\mathrm{Gr}_{a}(P_{\ast}E_{\alpha})$ which is induced by the nilpotent part of 
    $\mathrm{Gr}_{a}(f_{\alpha})$ on $\mathrm{Gr}_{a}(E_{\alpha})$.
    
    We set a holomorphic Hermitian vector bundle $(E', \deebar_{E'}, h_{E'})$ on $\Delta(R)^{\ast}$
    and a holomorphic endomorphism $f' \in \Gamma(\Delta(R)^{\ast},E')$ as
    \begin{eqnarray*}
    	\left\{\begin{array}{ll}
    		\deebar_{E'}(\mb{e'}) &= \mb{e'}(\Gamma_{\bar{\tau}} + ((2\pi)^{-1}\log|z|)^{-1}N_{\bar{\tau}})d\bar{z}/2\pi\bar{z} \\
    		h_{E'}(e'_i, e'_j) &= \delta_{ij} \\
    		f'(\mb{e'}) &= \mb{e'}(\Gamma_{\bar{w}} + ((2\pi)^{-1}\log|z|)^{-1}N_{\bar{w}}),
    	\end{array}\right.
    \end{eqnarray*}
    where $\mb{e'} = (e'_i)$ is a $C^{\infty}$-frame of $E'$ on $\Delta(R)^{\ast}$.
    Let $E' = \bigoplus E'_{\alpha}$ be the holomorphic decomposition induced by the eigen decomposition of $\Gamma_{\bar{w}}$.
    We have $f'(E'_{\alpha})\subset E'_{\alpha}$,
    hence we write $f'= \bigoplus_{\alpha} f'_{\alpha}$.
    Then $(E'_{\alpha},h_{E'_{\alpha}},\deebar_{E'_{\alpha}})$ is also acceptable
    as $(E_{\alpha},h_{\alpha},\deebar_{E_{\alpha}})$.
    Moreover,
    for $a\in\Par(P_{\ast}E'_{\alpha})$
    we also have the weight filtration $\{W_{i}\mathrm{Gr}_{a}(P_{\ast}E'_{\alpha})\}_{i\in\mathbb{Z}}$ on $\mathrm{Gr}_{a}(P_{\ast}E'_{\alpha})$ which is induced by the nilpotent part of 
    $\mathrm{Gr}_{a}(f'_{\alpha})$ on $\mathrm{Gr}_{a}(E'_{\alpha})$.
    \begin{Prp}\label{Prp:Isom to model bundle}
      There exists a holomorphic isomorphism $\Psi : (E, h_E, \deebar_E) \to (E', h_{E'}, \deebar_{E'})$
      such that $\Psi$ and $\Psi^{-1}$ are bounded, and $\Psi$ preserves the decompositions $E = \bigoplus E_{\alpha}$ and $E' = \bigoplus E'_{\alpha}$.
      In particular, we have $\Par(P_{\ast}E'_{\alpha}) = \Par(P_{\ast}E_{\alpha})$ and
      the induced isomorphisms $\Psi :\mathrm{Gr}^{W}(\mathrm{Gr}_{a}(E_{\alpha})) \to \mathrm{Gr}^{W}(\mathrm{Gr}_{a}(E'_{\alpha}))$.
    \end{Prp}
    \begin{proof}
      For an eigenvalue $\alpha \in \mathbb{C}$ of $\Gamma_{\bar{w}}$,
      let $\mb{e'}_{\alpha}$ be the subset of $\mb{e'}$ which spans $E'_{\alpha}$.
      We define $\Psi_{1,\alpha} : E_{\alpha} \to E'_{\alpha}$ as  $\Psi_{1,\alpha}(\mb{e}_{\alpha}) := \mb{e'}_{\alpha}$.
      Then we have $\deebar\Psi_{1,\alpha}(\mb{e}_{\alpha}) = \mb{e'}_{\alpha}\tilde{\eps}_{\bar{\tau},\alpha} d\bar{z}/\bar{z}$,
      where $\tilde{\eps}_{\bar{\tau}} = \sum_{\alpha} \tilde{\eps}_{\bar{\tau},\alpha}$
      is the decomposition of $\tilde{\eps}_{\bar{\tau}}$ in Proposition \ref{Prp:hol Asymp of Inst} induced by
      the decomposition $V|_{(-\infty,-R)\times T^3}=\bigoplus_{\alpha} V_{\alpha}$ in Corollary \ref{Cor:asymp. ortho. and hol. decomp. of inst.}.
      Let $\mb{b}_{\alpha}=(b_{\alpha,i})$ and $\mb{b'}_{\alpha}= (b'_{\alpha,j})$ be holomorphic frames of $P_{a}E_{\alpha}$ and $P_{a}E'_{\alpha}$ respectively
      that they have the norm estimates in Proposition \ref{Prp:norm estimate}.
      We set a function $K_{\alpha}=(K_{\alpha,ij}):\Delta^{\ast}(R) \to \mathrm{Mat}(r_{\alpha},\mathbb{C})$ as the change of basis of
      $\bar{z}^{-1}\tilde{\eps}_{\bar{\tau},\alpha}$ from the frames $\mb{e}_{\alpha}$
      and $\mb{e'}_{\alpha}$ to the frames $\mb{b}_{\alpha}$ and $\mb{b'}_{\alpha}$.
      Because of the estimate of $\tilde{\eps}_{\bar{\tau}}$ in Proposition  \ref{Prp:hol Asymp of Inst} and the norm estimates of $\mb{b}_{\alpha}$ and $\mb{b'}_{\alpha}$ in Proposition \ref{Prp:norm estimate},
      we have the following estimate
      \[
      	|K_{\alpha,ij}| = O\biggl(|z|^{\mathrm{ord}(b_{\alpha,i}) - \mathrm{ord}(b'_{\alpha,j}) - 1} (-\log|z|)^{-\mathrm{deg}^{W}(b_{\alpha,i}) + \mathrm{deg}^{W}(b'_{\alpha,j}) - 1 - \delta}\biggr).
      \]
      Therefore, according to \cite[Lemma 7.1]{Ref:Sim2}
      there exist functions $S_{\alpha,ij}$ such that we have $\deebar(S_{\alpha,ij}) = K_{\alpha,ij}d\bar{z}$, and it satisfies the following estimate
      \[
      	|S_{\alpha,ij}| = O\biggl(|z|^{\mathrm{ord}(b_{\alpha,i}) - \mathrm{ord}(b'_{\alpha,j})} (-\log|z|)^{-\mathrm{deg}^{W}(b_{\alpha,i}) + \mathrm{deg}^{W}(b'_{\alpha,j}) - \delta}\biggr).
      \]
      We set $S_{\alpha}:E_{\alpha} \to E'_{\alpha}$ as $S_{\alpha}(b_i) := \sum_{j} S_{\alpha,ij} \cdot b'_j$.
      Then we have $\deebar S_{\alpha} = \tilde{\eps}_{\bar{\tau},\alpha} d\bar{z}/\bar{z}$ and $|S_{\alpha}| = O((-\log|z|)^{-\delta})$.
      Therefore, for some $R'> R$, $\Psi_{\alpha} := \Psi_{1,\alpha} - S_{\alpha}$ is a holomorphic isomorphism on $\Delta^{\ast}(R')$
      such that $\Psi_{\alpha}$ and $\Psi_{\alpha}^{-1}$ are bounded.
      Finally, we set $\Psi := \oplus_{\alpha} \Psi_{\alpha}$. 
      Then $\Psi$ is holomorphic isomorphism,
      and both $\Psi$ and $\Psi^{-1}$ are bounded.
      We write $\Psi_{\alpha}(b_{\alpha,i}) = \sum_{j} \Psi_{\alpha,ij} \cdot b'_{\alpha,j}$.
      Then if $\mathrm{ord}(b_{\alpha,i}) < \mathrm{ord}(b'_{\alpha,j})$, or $\mathrm{ord}(b_{\alpha,i}) = \mathrm{ord}(b'_{\alpha,j})$ and $\mathrm{deg}^{W}(b_{\alpha,i}) < \mathrm{deg}^{W}(b'_{\alpha,j})$,
      then we have $\Psi_{ij}(0)= 0$ because of the norm estimates of $\mb{b}_{\alpha}$ and $\mb{b}'_{\alpha}$.
      Hence remains follow from it.
    \end{proof}
	For $a \in \Par(P_{\ast}E) = \Par(P_{\ast}E')$,
	the gradations $\mathrm{Gr}_{a}(E)$ and $\mathrm{Gr}_{a}(E')$ are skyscraper sheaves with the supports $\{0\}\subset\Delta(R)$,
	and $f$ and $f'$ induces the endomorphisms $\mathrm{Gr}_{a}(f)$ and $\mathrm{Gr}_{a}(f')$ on $\mathrm{Gr}_{a}(E)$ and $\mathrm{Gr}_{a}(E')$ respectively.
	We regard $(\mathrm{Gr}_{a}(E^{(\prime)}),\mathrm{Gr}_{a}(f^{(\prime)}))$ as
	a vector space with an endomorphism.
    \begin{Cor}\label{Cor:Isom to model bundle of Gr}
    	From the isomorphism $\Psi$,
    	we can construct an isomorphism 
		$(\mathrm{Gr}_{a}(E) \times T^2, \deebar_{T^2} + \mathrm{Gr}_{a}(f) d\bar{w} ) \simeq ( \mathrm{Gr}_{a}(E') \times T^2, \deebar_{T^2} + \mathrm{Gr}_{a}(f') d\bar{w} )$
      	for $a \in \Par(P_{\ast}E)$.
    \end{Cor}
    \begin{proof}
    	In Proposition \ref{Prp:Isom to model bundle},
      we proved that $\Psi$ induces an isomorphism between $\mathrm{Gr}_{a}(E_{\alpha})$ and $\mathrm{Gr}_{a}(E'_{\alpha})$,
      and it also induces an isomorphism between their gradation of the weight filtrations. 
      Therefore, the Jordan normal forms of $\mathrm{Gr}_{a}(f)$ and $\mathrm{Gr}_{a}(f')$ are equivalent.
      Hence the proof is complete.
    \end{proof}
	\subsection{Algebraic Nahm transform}\label{SubSec:Alg. Nahm trans.}
	  	We set a  hypersurface $D = D_1 \sqcup D_2 := (\{0\}\times T^2) \sqcup (\{\infty\}\times T^2) \subset \mathbb{P}^1\times T^2$.
	  	Let $P_{\ast\ast}V$ be a stable filtered bundle of degree $0$ and rank $r>1$ on $(\mathbb{P}^1\times T^2,D)$ unless otherwise noted.
	    We set $\mathrm{Sing}_{\mathbb{R}}(P_{\ast\ast}V):=-\Par(P_{\ast\ast}V,1)\cup\Par(P_{\ast\ast}V,2) \subset \mathbb{R}$,
	    where $-\Par(P_{\ast\ast}V,1) := \{a\in\mathbb{R}\mid -a \in \Par(P_{\ast\ast}V,1)\}$.
	    We also set $\mathrm{Sing}_{S^1}(P_{\ast\ast}V):=\pi(\mathrm{Sing}_{\mathbb{R}}(P_{\ast\ast}V)) \subset S^1$,
	    where $\pi:\mathbb{R}\to S^1$ is the quotient map.
	    
	    Let $\mathcal{L} \to T^2\times \hat{T}^2$ be the Poincar\'{e} bundle of $T^2$.
	    For $I \subset \{1,2,3\}$, let $p_I$ be the projection of $\mathbb{P}^1\times T^2\times \hat{T}^2$
	    onto the product of the $i$-th components $(i \in I)$.
	    We define a functor $F^i:\mathrm{Coh}(\mathcal{O}_{\mathbb{P}^1\times T^2}) \to \mathrm{Coh}(\mathcal{O}_{\hat{T}^2})$ to be
	    $F^i(\mathcal{F}):=R^{i}{p_3}_{\ast}(p^{\ast}_{12}\mathcal{F} \otimes p^{\ast}_{23}\mathcal{L})$.
	    \begin{Prp}\label{Prp:Algebraic Nahm trans}
	    	The sheaves $F^1(P_{<-\hat{t}<\hat{t}}V)$ and $F^1(P_{-\hat{t}\hat{t}}V)$ are locally free for any $\hat{t} \in \mathbb{R}$.
	    \end{Prp}
	  	\begin{proof}
	  		Let $p:\mathbb{P}^1\times T^2\to T^2$ be the projection.
	  		By Proposition \ref{Prp:vanish cohom of stable bundle},
	  		we have $H^0(\mathbb{P}^1\times T^2, P_{-\hat{t}\hat{t}}V \otimes p^{\ast}F) = H^2(\mathbb{P}^1\times T^2, P_{-\hat{t}\hat{t}}V\otimes  p^{\ast}F) =0$ for any $\hat{t}\in\mathbb{R}$ and any $F\in\mathrm{Pic}^0(T^2)$.
	  		Therefore $h^1(\mathbb{P}^1\times T^2, P_{-\hat{t}\hat{t}}V\otimes p^{\ast}F)$ is a constant for any $F\in\mathrm{Pic}^0(T^2)$ by the Riemann-Roch-Hirzebruch theorem.
	  		Hence $F^1(P_{-\hat{t}\hat{t}}V)$ is a locally free sheaves on $\hat{T}^2$.
	  		By the same way we can prove that $F^1(P_{<-\hat{t}<\hat{t}}V)$ is also locally free.
	    \end{proof}
	    We will denote by $\mathrm{AN}(P_{\ast\ast}V)_{\hat{t}}$
	    the locally free sheaf $F^1(P_{-\hat{t}\hat{t}}V)$ for $\hat{t}\in\mathbb{R}$.
	    Since we have $P_{-(\hat{t}+1)\ (\hat{t}+1)}V \simeq P_{-\hat{t}\hat{t}}V$,
	    we can regard $\{\mathrm{AN}(P_{\ast\ast}V)_{\hat{t}}\}$ as a family on $\hat{t} \in S^1$.
	    \begin{Def}\label{Def:Algebraic singularity set}
	      We define the algebraic singularity set $\mathrm{Sing}(P_{\ast\ast}V) \subset \hat{T}^3 = S^1\times \hat{T}^2$ as 
	      \begin{align*}
	      	\mathrm{Sing}(P_{\ast\ast}V) := &
	        \ \left(\bigcup_{\hat{t} \in \Par(P_{\ast\ast}V,1)} \{-\pi(\hat{t})\}\times \mathrm{Spec}({}^{1}\mathrm{Gr}_{\hat{t}}(P_{\ast\ast}V)) \right)\\
	        &\cup\left(\bigcup_{\hat{t} \in  \Par(P_{\ast\ast}V,2)} \{\pi(\hat{t})\}\times \mathrm{Spec}({}^{2}\mathrm{Gr}_{\hat{t}}(P_{\ast\ast}V)) \right),
	      \end{align*}
	      where
	      $\mathrm{Spec}(\cdot)$ is the spectrum set of a semistable bundle of degree $0$ on $T^2$ (See Definition \ref{Def:spectrum of ss deg0}).
	    \end{Def}
	    By Corollary \ref{Cor:Isom to model bundle of Gr},
	    the notions of singularity set and algebraic singularity set are compatible.
	    \begin{Prp}
	    	Let $(V,h,A)$ be an irreducible $L^2$-finite instanton on $\mathbb{R} \times T^3$  of rank $r>1$, and $P_{\ast\ast}V$ be the associated stable filtered bundle.
	      Then we have $\mathrm{Sing}(V,h,A) = \mathrm{Sing}(P_{\ast\ast}V)$.\qed
	    \end{Prp}
	    \begin{Prp}\label{Prp: algebraic scattering map}
	    	Let $I \subset \mathbb{R}$ be a closed interval with $|I| < 1$ and $\pi:\mathbb{R} \to S^1$ be the quotient map.
	      Let $U \subset \hat{T}^2$ be the complement of the image of $\left(\pi(I)\times\hat{T}^2\right) \cap \mathrm{Sing}(P_{\ast\ast}V)$
	      under the projection $\hat{T}^3 \to \hat{T}^2$.
	      Then, for any $\hat{t}, \hat{t}' \in I$, we have a natural isomorphism
	      $\mathrm{AN}(P_{\ast\ast}V)_{\hat{t}}|_{U} \simeq \mathrm{AN}(P_{\ast\ast}V)_{\hat{t'}}|_{U}$.
	    \end{Prp}
	    \begin{proof}
	    	We only need to prove the claim under the assumption $I=[-\eps,\eps]$ for a positive number $\eps>0$.
	    	By extending $I$,
	    	we may assume that any points in $\mathrm{Sing}_{\mathbb{R}}(P_{\ast\ast}V)\cap I$ are interior points of $I$.
	    	Hence we may assume either
	      	$I \cap \mathrm{Sing}_{\mathbb{R}}(P_{\ast\ast}V) = \emptyset$ or $I \cap \mathrm{Sing}_{\mathbb{R}}(P_{\ast\ast}V) = \{0\}$.
	      	Under this assumption, we will show that there exists the natural isomorphism
	      \[
	      	\mathrm{AN}(P_{\ast\ast}V)_{\hat{t}}|_{U} \simeq F^1(P_{-\eps-\eps}V)|_{U}
	      \]
	      for any $\hat{t} \in I$.
	      This isomorphism is obvious for $\hat{t}=0$,
	      and a proof for $\hat{t}>0$ is also valid for $\hat{t}<0$.
	      Hence we assume $\hat{t}>0$.
	      Then, from the short exact sequence $0 \to P_{-\eps-\eps}V \to P_{-\hat{t}\hat{t}}V \to {}^{2}\mathrm{Gr}_{1}(P_{\ast\ast}V) \to 0$,
	      we obtain the exact sequence
	      $F^0({}^{2}\mathrm{Gr}_{0}(P_{\ast\ast}V)) \to  F^1(P_{-\eps-\eps}V) \to \mathrm{AN}(P_{\ast\ast}V)_{\hat{t}} \to F^1({}^{2}\mathrm{Gr}_{0}(P_{\ast\ast}V))$.
	      Here we have $F^0({}^{2}\mathrm{Gr}_{0}(P_{\ast\ast}V)) = 0$
	      and $F^1({}^{2}\mathrm{Gr}_{0}(P_{\ast\ast}V))|_{U} = 0$
	      from Corollary \ref{Cor:FM-trans of ss deg0}.
	      Hence we obtain $\mathrm{AN}(P_{\ast\ast}V)_{\hat{t}}|_{U} \simeq F^1(P_{-\eps-\eps}V)|_{U}$.
	    \end{proof}
	    \begin{Def}
	    	We call this isomorphism the \textit{algebraic scattering map}.
	    \end{Def}
	    \begin{Cor}
	    	The family $\{\mathrm{AN}(P_{\ast\ast}V)_{\hat{t}}\}_{\hat{t}}$ forms a mini-holomorphic bundle
	      $(\mathrm{AN}(P_{\ast\ast}V),\deebar_{\mathrm{AN}},\partial_{\mathrm{AN},\hat{t}})$ on $\hat{T}^3 \setminus \mathrm{Sing}(P_{\ast\ast}V)$.
	    \end{Cor}
	    \begin{Def}
	    	We call this construction the \textit{algebraic Nahm transform}.
	    \end{Def}
		We prove that the Nahm transform and the algebraic Nahm transform are compatible.
		\begin{Thm}\label{Thm:Algebraic Nahm trans}
			Let $(V,h,A)$ be an irreducible $L^2$-finite instanton of rank $r>1$ on $\mathbb{R}\times T^3$
			and $P_{\ast\ast}V$ the associated stable filtered bundle on $(\mathbb{P}^1\times T^2,\{0,\infty\}\times T^2)$.
			Then, the underlying mini-holomorphic bundle $(\hat{V},\deebar_{\hat{A}},\partial_{\hat{V},\hat{t}})$
			of the Nahm transform $(\hat{V},\hat{h},\hat{A},\hat{\Phi})$ of $(V,h,A)$ is isomorphic to
			$(\mathrm{AN}(P_{\ast\ast}V),\deebar_{\mathrm{AN}},\partial_{\mathrm{AN},\hat{t}})$.
		\end{Thm}
		\begin{proof}
			For a smooth manifold $M$ and a vector bundle $F$ on $M$,
			let $C^{\infty}(F)$ denote the sheaf of $C^{\infty}$-sections of $F$.
			If $M$ is a complex manifold and $F$ is a holomorphic bundle,
			then we also denote by $\mathcal{O}(F)$ the sheaf of holomorphic sections of $F$.
			
			Let $U_{\hat{t}}\subset \hat{T}^2$ be the complement of the image of
			$\mathrm{Sing}(V,h,A)\cap(\{\hat{t}\}\times\hat{T}^2)$ under the projection map
			$\hat{T}^3\to\hat{T}^2$.
			We first construct an isomorphism $(\hat{V},\deebar_{\hat{A}})|_{\{\hat{t}\} \times U_{\hat{t}}} \simeq \mathrm{AN}(P_{\ast\ast}V)_{\hat{t}}|_{U_{\hat{t}}}$
			for any $\hat{t} \in S^1$.
			By replacing $(V,h,A)$ with $(V, h, A + 2\pi\sqrt{-1}\hat{t}dt)$,
			we may assume $\hat{t} = 0$.
			From the short exact sequence $0\to P_{<0<0}V \to P_{00}V \to {}^1\mathrm{Gr}_{0}(P_{\ast\ast}V)\oplus {}^2\mathrm{Gr}_{0}(P_{\ast\ast}V)\to 0$,
			we obtain the exact sequence $0\to F^1(P_{<0<0}V) \to \mathrm{AN}(P_{\ast\ast}V)_{0} \to F^1({}^1\mathrm{Gr}_{0}(P_{\ast\ast}V)\oplus {}^2\mathrm{Gr}_{0}(P_{\ast\ast}V))$.
			Hence by the definition of $U_0$,
			the isomorphism $F^1(P_{<0<0}V)|_{U_0} \simeq \mathrm{AN}(P_{\ast\ast}V)_{0}|_{U_0}$ holds.
			Hence it suffices to prove $(\hat{V},\deebar_{\hat{A}})|_{\{0\} \times U_{0}} \simeq F^1(P_{<0<0}V)|_{U_{0}}$.
			On one hand,
			since the Dolbeault resolution $(C^{\infty}(\Omega^{0,\ast}(p^{\ast}_{12}P_{<0<0}V \otimes p^{\ast}_{23}\mathcal{L})),\deebar)$
			of the holomorphic vector bundle $p^{\ast}_{12}P_{<0<0}V \otimes p^{\ast}_{23}\mathcal{L}$
			is acyclic for the functor ${p_3}_{\ast}$,
			we have an isomorphism
			$F^1(P_{<0<0}V) \simeq H^1({p_3}_{\ast}C^{\infty}(\Omega^{0,\ast}(p^{\ast}_{12}P_{<0<0}V \otimes p^{\ast}_{23}\mathcal{L})), \deebar)$.
			Let $\mathcal{V}^{i}$ be the flat bundle on $\hat{T}^2$
			which is the quotient of the product bundle $\underline{C^{\infty}(\mathbb{P}^1\times T^2, \Omega^{0,i}(P_{<0<0}V))}$
			on $\mathrm{Hom}(\mathbb{R}^2,\mathbb{R})$ by $\Lambda^{\ast}_2$-action as in section 4.
			Then we have $F^1(P_{<0<0}V) \simeq H^1({p_3}_{\ast}C^{\infty}(\Omega^{0,\ast}(p^{\ast}_{12}P_{<0<0}V \otimes p^{\ast}_{23}\mathcal{L})), \deebar) \simeq H^1(\mathcal{O}(\mathcal{V}^{\ast}),\deebar_{P_{<0<0}V_{\xi}})$ on $\hat{T}^2$.
			On the other hand,
			for the normed vector space $X^i := \{f \in L^2(\mathbb{R}\times T^3, \Omega^{0,i}(V)) \mid \deebar_{A}(f) \in L^2\}$,
			we construct the flat vector bundle $\mathcal{V}^{i}_{L^2}$ on $\hat{T}^2$  from $X^i$ in a similar way.
			Then, we also have an isomorphism 
			$(\hat{V},\deebar_{\hat{A}})|_{\{0\} \times U_{0}} \simeq H^1(\mathcal{O}(\mathcal{V}^{\ast}_{L^2}|_{U_0}), \deebar_{A_{\xi}})$.
			By Proposition \ref{Prp:norm estimate},
			we obtain the natural inclusions $C^{\infty}(\mathbb{P}^1\times T^2, \Omega^{0,i}(P_{<0<0}V)) \hookrightarrow L^2(\mathbb{R}\times T^3,\Omega^{0,i}(V))$
			and the induced chain homomorphism
			$\varphi:(\mathcal{O}(\mathcal{V}^{\ast}|_{U_0}),\deebar_{P_{<0<0}V_{\xi}}) \to (\mathcal{O}(\mathcal{V}^{\ast}_{L^2}|_{U_0}), \deebar_{A_{\xi}})$.
			In order to prove that $\varphi$ is a quasi-isomorphism,
			we only need to prove that the specialization
			$\varphi_{\xi}:(C^{\infty}(\mathbb{P}^1\times T^2, \Omega^{0,\ast}(P_{<0<0}V)),\deebar_{P_{<0<0}V_{\xi}}) \to (X^{\ast},\deebar_{A_{\xi}})$
			is a quasi-isomorphism for any $\xi \in U_0$.
			We will show this below in Proposition \ref{Prp:L^2 Dolbeault}.
			
			It remains to prove that the scattering map of $(\hat{V},\hat{h},\hat{A},\hat{\Phi})$ and
			the algebraic scattering map of $(\mathrm{AN}(P_{\ast\ast}V),\deebar_{\mathrm{AN}},\partial_{\mathrm{AN},t})$
			are compatible under the isomorphism $H^1(\varphi)$.
			It suffices to show the following lemma.
			\begin{Lem}\label{Lem:compatibility of scattering map}
				Let $I=[-\eps,\eps] \subset \mathbb{R}$ be a closed interval and $\xi \in \hat{T}^2$
				with the condition $(\pi(I) \times \{\xi\}) \cap \mathrm{Sing}(V,h,A) = \emptyset$.
				For $\hat{t} \in I$, we take $f_{\hat{t}} \in \mathrm{AN}(P_{\ast\ast}V)_{(\hat{t},\xi)} \simeq H^1(\mathbb{P}^1\times T^2,P_{-\hat{t}\hat{t}}V_{\xi})$
				satisfying the condition that $f_{\hat{t}}$ is a constant under the algebraic scattering map for any $\hat{t}\in I$. 
				Then, $H^1(\varphi)_{(\hat{t},\xi)}(f_{\hat{t}})$ is a constant under the scattering map of $(\hat{V},\hat{h},\hat{A},\hat{\Phi})$ for any $\hat{t}\in I$.
			\end{Lem}
		\end{proof}
		\begin{proof}[{\upshape \textbf{(Proof of Lemma \ref{Lem:compatibility of scattering map})}}]
			We may assume $\xi=0$.
			As in section 4,
			we set $\mathcal{V}$ the flat vector bundle on $\hat{T}^3$ which is the quotient of the product bundle
			$\underline{L^2(\mathbb{R}\times T^3, S^{-}\otimes V)}$ on $\mathrm{Hom}(\mathbb{R}^3,\mathbb{R})$ by $\Lambda^{\ast}_3$-action.
			By using the orthogonal projection $P:\mathcal{V}|_{\hat{T}^3 \setminus \mathrm{Sing}(V,h,A)} \to \hat{V}$,
			the equation of the scattering map can be written as follows:
			\begin{equation}\label{Eqn:scattering eq}
			\nabla_{\hat{A}_{\hat{t}}}(\cdot) - \sqrt{-1}\hat{\Phi}(\cdot) = 
			P\left( \partial_{\hat{t}}(\cdot) + \log|z|(\cdot) \right) = 0.
			\end{equation}
			By the assumption, there exists $f \in H^1(\mathbb{P}^1\times T^2, P_{-\eps-\eps}V)$ such that $f_{\hat{t}}$ is the image of $f$ under the natural isomorphism $H^1(\mathbb{P}^1\times T^2, P_{-\eps-\eps}V) \simeq H^1(\mathbb{P}^1\times T^2, P_{-\hat{t}\hat{t}}V)$.
			Here the image of $f_{\hat{t}}$ in $\hat{V}_{(\hat{t},0)}$ is 
			$P(f\cdot |z|^{-\hat{t}})$.
			Hence the derivative is $\partial_{\hat{t}}P(f\cdot |z|^{-\hat{t}})= P(\partial_{\hat{t}}(f\cdot |z|^{-\hat{t}}))=
			-P(f\cdot|z|^{-\hat{t}}\cdot\log|z|)$.
			Therefore $H^1(\varphi)_{(0,\hat{t})}(f_{\hat{t}})$ satisfies the equation (\ref{Eqn:scattering eq}),
			and this completes the proof.
		\end{proof}
	\subsubsection{$L^2$-Dolbeault lemma}
		Let $(V,h,A)$ be an irreducible $L^2$-finite instanton and $P_{\ast\ast}V$ its prolongation.
		We assume a condition $(0,0) \not\in \mathrm{Sing}(V,h,A)$.
		Let $\mathcal{A}^{0,i}(P_{<0<0}V)$ denote the sheaf of smooth sections of $\Omega^{0,i}(P_{<0<0}V)$.
		Then we have the natural isomorphism
		$H^1(\mathbb{P}^1\times T^2,(\mathcal{A}^{0,\ast}(P_{<0<0}V),\deebar_{P_{<0<0}V}))\simeq \mathrm{AN}(P_{\ast\ast}V)_{(0,0)}$.
		Let ${}_{p}\mathcal{A}^{0, i}_{L^2}(V)$ be the presheaf on $\mathbb{P}^1\times T^2$ that associates an open subset $W \subset \mathbb{P}^1\times T^2$ to a $\mathbb{C}$-vector space
		\[
			\Bigl\{s \in L^2\bigl(W \cap (\mathbb{R}\times T^3), \Omega^{0,i}(V)\bigr)  \;\Bigm|\; \deebar_{A}(s) \in L^2\bigl(W \cap (\mathbb{R}\times T^3), \Omega^{0,i+1}(V)\bigr) \Bigr\},
		\]
		where $L^2(W \cap (\mathbb{R}\times T^3), \Omega^{0,i}(V))$ means the set of $L^2$-sections of $\Omega^{0,i}(V)$ on $W \cap (\mathbb{R}\times T^3)$ with respect to $h$ and $g_{\mathbb{R}\times T^3}$.
		Let $\mathcal{A}^{0, i}_{L^2}(V)$ denote the sheafification of ${}_{p}\mathcal{A}^{0, i}_{L^2}(V)$.
		Then we have the natural isomorphism $H^1(\mathbb{P}^1\times T^2,(\mathcal{A}^{0,\ast}_{L^2},\deebar_{A})) \simeq \hat{V}_{(0,0)}$.
		Let $K:\mathcal{A}^{0, i}(P_{<0<0}V)\to \mathcal{A}^{0,i}_{L^2}(V)$ be the sheaf homomorphism induced by the inclusion map $C^{\infty}(W,\Omega^{0,i}(P_{<0<0}V)) \subset L^2\bigl(W \cap (\mathbb{R}\times T^3), \Omega^{0,i}(V)\bigr)$ for an open subset $W\subset \mathbb{P}^1\times T^2$.
		For simplicity, we use the same symbol $K$ for the chain map $(\mathcal{A}^{0,\ast}(P_{<0<0}V),\deebar_{P_{<0<0}V}) \to (\mathcal{A}^{0,\ast}_{L^2},\deebar_{A})$.
		\begin{Prp}\label{Prp:L^2 Dolbeault}
			The chain map $K$ induces an isomorphism $\hat{V}_{(0,0)} \simeq \mathrm{AN}(P_{\ast\ast}V)_{(0,0)}$.
		\end{Prp}
	\begin{proof}
    Let $q:\mathbb{P}^1\times T^2 \to \mathbb{P}^1$ be the projection.
    To show that $K$ induces an isomorphism $\hat{V}_{(0,0)} \simeq \mathrm{AN}(P_{\ast\ast}V)_{(0,0)}$,
    we only need to prove that $q_{\ast}K$ is a quasi-isomorphism.
  	Thus we consider the following lemma.
    \begin{Lem}\label{Lem:quasi-isom}
    	For any $z \in \mathbb{P}^1$, 
    	$(q_{\ast}K)_z$ is a quasi-isomorphism,
    	where $(q_{\ast}K)_z$ means the induced chain map between the stalks at $z$.
    \end{Lem}
    Once we admit the lemma,
    then $q_{\ast}K$ is a quasi-isomorphism and the proof is complete.
  \end{proof}
    \begin{proof}[{\upshape \textbf{(Proof of Lemma \ref{Lem:quasi-isom})}}]
    	This lemma is trivial unless $z=0$ or $z=\infty$,
		and the same proof works for both the cases $z=0$ and $z=\infty$.
		Thus it suffices to consider only the case $z=0$.
		
		From Proposition \ref{Prp:hol Asymp of Inst},
		we take a sufficiently small neighborhood $U$ of $0 \in \mathbb{P}^1$,
		a frame $\mb{v}=(v_i)$ of $V$ on $U^{\ast} \times T^2$,
		and a model solution $(\Gamma,N)$ of the Nahm equation
		that they satisfy conditions in Proposition \ref{Prp:hol Asymp of Inst},
		where $U^{\ast} := U \setminus \{0\}$.
		Let $(E, h_{E}, \deebar_{E}, f)=\bigoplus_{\alpha}(E_{\alpha}, h_{E_{\alpha}}, \deebar_{E_{\alpha}}, f_{\alpha})$ be the holomorphic Hermitian vector bundle with the endomorphism on $U^{\ast}$ constructed from $(V,h,A)$ in Corollary \ref{Cor:hol push down of asymp inst}.
		Let $\mb{e}$ be the $C^{\infty}$-frame of $E$ which corresponds to $\mb{v}$.
		We take the sheaves $\mathcal{A}^{0,i}_{L^2}(E)$ and $\mathcal{A}^{0,i}(P_{<0}E)$ on $U$ constructed from $E$ and $P_{<0}E$
		in a similar way to $\mathcal{A}^{0,i}_{L^2}(V)$ and $\mathcal{A}^{0,i}(P_{<0<0}V)$.
		Then we have $\mathcal{A}^{0,i}_{L^2}(E) \simeq\bigoplus_{\alpha}\mathcal{A}^{0,i}_{L^2}(E_{\alpha})$ and $\mathcal{A}^{0,i}(P_{<0}E)\simeq \bigoplus_{\alpha}\mathcal{A}^{0,i}(P_{<0}E_{\alpha})$ because $E_{\alpha}$ and $E_{\beta}$ are orthogonal for any $\alpha\neq\beta$.
		
		We write $\nabla^{0,1}_{A} = \deebar_{A_{\bar{z}}}d\bar{z} +\deebar_{A_{\bar{w}}}d\bar{w}$.
		For $s \in L^2(U^{\ast} \times T^2, V)$,
		we denote the Fourier series expansion of $s$ with respect to $\mb{v}$ by 
		\[
			s=\mb{v}\cdot\sum_{n \in \Lambda_2^{\ast}} s_{n}(z) \exp(2\pi\sqrt{-1} \i< n,w >).
		\]
		By using the Fourier series expansion,
		we set the bounded operator $I:L^2(U^{\ast} \times T^2, V) \to L^2(U^{\ast},E)$ by $I(s):= \mb{e}\cdot s_{0}(z)$,
		and set $L^2(U^{\ast} \times T^2, V)^{\perp} := \mathrm{Ker}(I)$.
		Because of Remark \ref{Remark:some attentions} (i),
		we can construct the inverse operator $G_{L^2}:L^2(U^{\ast} \times T^2,V)^{\perp}\to L^2(U^{\ast} \times T^2,V)^{\perp}$ of $\deebar_{A_{\bar{w}}}$ as
		\[
			G_{L^2}(s) := \mb{v}\cdot\sum_{n=(n_2,n_3) \in \Lambda_2^{\ast} \setminus \{(0,0)\}} \bigl((\sqrt{-1} n_2 - n_3)\pi + A_{\bar{w}} \bigr)^{-1} s_{n}(z) \exp(2\pi\sqrt{-1} \i< n,w >),
		\]
		where $A_{\bar{w}}$ is the component of $\nabla^{0,1}_{A}(\mb{v}) = \mb{v}(A_{\bar{w}}d\bar{w}+A_{\bar{z}}d\bar{z})$.
		Therefore $(q_{\ast}\mathcal{A}^{0, \ast}_{L^2}(V)(U) , \deebar_A)$ is quasi-isomorphic to the following complex:
		\[
			\mathcal{A}^{0, 0}_{L^2}(E)(U)
			\xrightarrow{f \oplus \deebar_{E}}
			\mathcal{A}^{0, 0}_{L^2}(E)(U) \oplus \mathcal{A}^{0, 1}_{L^2}(E)(U)
			\xrightarrow{\deebar_{E} \;-\; f}
			\mathcal{A}^{0, 1}_{L^2}(E)(U).
		\]
		Since $f_{\alpha}:\mathcal{A}^{0, 0}_{L^2}(E_{\alpha})(U) \to \mathcal{A}^{0, 0}_{L^2}(E_{\alpha})(U)$ is an isomorphism for $\alpha\neq0$,
		$(q_{\ast}\mathcal{A}^{0, \ast}_{L^2}(V)(U) , \deebar_A)$ is quasi-isomorphic to
		\[
			\mathcal{A}^{0, 0}_{L^2}(E_0)(U)
			\xrightarrow{f_0 \oplus \deebar_{E_0}}
			\mathcal{A}^{0, 0}_{L^2}(E_0)(U) \oplus \mathcal{A}^{0, 1}_{L^2}(E_0)(U)
			\xrightarrow{\deebar_{E_0} \;-\; f_0}
			\mathcal{A}^{0, 1}_{L^2}(E_0)(U).
		\]
		By the assumption $(0,0) \not\in\mathrm{Sing}(V,h,A)$,
         the same argument in the proof of \cite[Proposition 11.5]{Ref:Zuc} shows 
        that $\deebar_{E_0}$ is surjective.
        Thus $(q_{\ast}\mathcal{A}^{0, \ast}_{L^2}(V)(U) , \deebar_A)$ 
        is quasi-isomorphic to
        \[
        	\left( \mathrm{Ker}(\deebar_{E_0}) \cap \mathcal{A}^{0,0}_{L^2}(E_0)(U) \right)
          	\xrightarrow{f_0}
          \left( \mathrm{Ker}(\deebar_{E_0}) \cap \mathcal{A}^{0, 0}_{L^2}(E_0)(U) \right).
        \]
        By a similar way we can prove that 
        $(q_{\ast}\mathcal{A}^{0, \ast}(P_{<0<0}V)(U), \deebar_{P_{<0}V})$
        is also quasi-isomorphic to
        \[
        	\Gamma(U,P_{<0}(E_0))
          	\xrightarrow{P_{<0}f_0}
          \Gamma(U,P_{<0}(E_0)).
        \]
        By the assumption $(0,0)\not\in\mathrm{Sing}(V,h,A)$,
        we can show any $s \in \mathrm{Ker}(\deebar_{E_0}) \cap \mathcal{A}^{0, 0}_{L^2}(E_0)(U^{\ast})$
        decays exponentially in $t \to -\infty$.
        Hence we have $\mathrm{Ker}(\deebar_{E_0}) \cap \mathcal{A}^{0, 0}_{L^2}(E_0)(U^{\ast}) = \Gamma(U,P_{<0}(E_0))$
        and the proof is complete.
    \end{proof}
	\section{Correspondence between weights}\label{Correspondence between weights}
	As in Section 5,
	we assume that $T^3$ is isomorphic to the product of a circle $S^1 = \mathbb{R}/\mathbb{Z}$ and a 2-dimensional torus $T^2 = \mathbb{R}^2/\Lambda_2$ as a Riemannian manifold.
	Let $P_{\ast\ast}V$ be a stable filtered bundle on $(\mathbb{P}^1\times T^2,\{0,\infty\}\times T^2)$
	of $\mathrm{deg}(P_{\ast\ast}V)=0$ and of rank $r>1$.
	Let $(\mathrm{AN}(P_{\ast\ast}V),\deebar_{\mathrm{AN}},\partial_{\mathrm{AN},t})$ be the algebraic Nahm transform of $P_{\ast\ast}V$.
	Let $\mathrm{Sing}_{0}(P_{\ast\ast}V) \subset \hat{T}^2$ be the image of $\mathrm{Sing}(P_{\ast\ast}V) \cap \{0\}\times \hat{T}^2$
	under the projection $\hat{T}^3 = S^1 \times \hat{T}^2 \to \hat{T}^2$.
	Let $\mathcal{L} \to T^2\times \hat{T}^2$ be the Poincar\'{e} bundle of $T^2$.
	For $I \subset \{1,2,3\}$, let $p_I$ be the projection of $\mathbb{P}^1\times T^2\times \hat{T}^2$
	onto the product of the $i$-th components $(i \in I)$.
	We define a functor $F^i:\mathrm{Coh}(\mathcal{O}_{\mathbb{P}^1\times T^2}) \to \mathrm{Coh}(\mathcal{O}_{\hat{T}^2})$ to be
	$F^i(\mathcal{F}):=R^{i}{p_3}_{\ast}(p^{\ast}_{12}\mathcal{F} \otimes p^{\ast}_{23}\mathcal{L})$ as in subsection \ref{SubSec:Alg. Nahm trans.}.
  	\begin{Prp}\label{Prp:merom comparison}
    	We take a positive number $\eps>0$ small enough to satisfy
    	$\mathrm{Sing}(P_{\ast\ast}V) \cap \bigl(q([-\eps,\eps])\times \hat{T}^2\bigr) = \{0\} \times \mathrm{Sing}_{0}(P_{\ast\ast}V)$,
    	where $q : \mathbb{R} \to S^1$ be the quotient map.
    	Then, we have the following.
    	\begin{itemize}
    		\item
    			We have a sequence of injections
    			\[
    				F^1(P_{-\eps-\eps}V) \hookrightarrow \mathrm{AN}(P_{\ast\ast}V)_{\pm\eps}
    				\hookrightarrow F^1(P_{\eps\eps}V)
    			\]
    			which is compatible with the algebraic scattering map.
    		\item 
    			Under the above sequence of injections,
    			we have the following isomorphisms:
    	\begin{align*}
    		&\mathrm{AN}(P_{\ast\ast}V)_{-\eps} \cap \mathrm{AN}(P_{\ast\ast}V)_{\eps} 
    		\simeq F^1(P_{-\eps-\eps}V).\\
    		&\mathrm{AN}(P_{\ast\ast}V)_{-\eps} + \mathrm{AN}(P_{\ast\ast}V)_{\eps}
    		\simeq F^1(P_{\eps\eps}V). \\
    		&F^1(P_{\eps\eps}V)/\mathrm{AN}(P_{\ast\ast}V)_{\eps}
    		\simeq F^1({}^{1}\mathrm{Gr}_{0}(P_{\ast\ast}V))
    		\simeq  H^1(\mathrm{FM}({}^{1}\mathrm{Gr}_{0}(P_{\ast\ast}V))).\\
    		&F^1(P_{\eps\eps}V)/\mathrm{AN}(P_{\ast\ast}V)_{-\eps} 
    		\simeq F^1({}^{2}\mathrm{Gr}_{0}(P_{\ast\ast}V))
    		\simeq  H^1(\mathrm{FM}({}^{2}\mathrm{Gr}_{0}(P_{\ast\ast}V))).
    	\end{align*}
    	Here ${}^{i}\mathrm{Gr}_{0}(P_{\ast\ast}V))$ is the gradation of $P_{\ast\ast}V$ (See subsubsection 2.4.1), and
    	$\mathrm{FM}(\mathcal{E}) \in D^{b}(\mathrm{Coh}(\mathcal{O}_{\hat{T}^2}))$ is the Fourier-Mukai transform of a coherent sheaf $\mathcal{E}$ on $T^2$.
    	\end{itemize}
  	\end{Prp}
  \begin{proof}
  	From the definition of stable filtered bundle on $(\mathbb{P}^1\times T^2,\{0,\infty\}\times T^2)$,
    ${}^{i}\mathrm{Gr}_{0}(P_{\ast\ast}V)\;(i=1,2)$ are  semistable locally free sheaves on $\{0\}\times T^2$ and
    $\{\infty\}\times T^2$ respectively.
    Therefore by Corollary \ref{Cor:FM-trans of ss deg0} we have
    $F^0({}^{i}\mathrm{Gr}_{0}(P_{\ast\ast}V)) = H^0(\mathrm{FM}({}^{i}\mathrm{Gr}_{0}(P_{\ast\ast}V))) = 0$.
  	Hence by the short exact sequences
    $0 \to P_{-\eps-\eps}V \to P_{\eps-\eps}V \to {}^{1}\mathrm{Gr}_{0}(P_{\ast\ast}V) \to 0$
    and
    $0 \to P_{\eps-\eps}V \to P_{\eps\eps}V \to {}^{2}\mathrm{Gr}_{0}(P_{\ast\ast}V) \to 0$,
    we have an inclusion of sheaves
    $F^1(P_{-\eps-\eps}V) \hookrightarrow \mathrm{AN}(P_{\ast\ast}V)_{-\eps} \hookrightarrow F^1(P_{\eps\eps}V)$.
    Since $F^2(P_{\eps-\eps}V) = 0$ is shown in Proposition \ref{Prp:Algebraic Nahm trans},
    we have 
    \[
    	F^1(P_{\eps\eps}V) / \mathrm{AN}(P_{\ast\ast}V)_{-\eps} 
      	= F^1({}^{2}\mathrm{Gr}_{0}(P_{\ast\ast}V))
        = H^1(\mathrm{FM}({}^{2}\mathrm{Gr}_{0}(P_{\ast\ast}V))).
    \]
    In a similar way,
    we obtain
    $F^1(P_{-\eps-\eps}V) \hookrightarrow \mathrm{AN}(P_{\ast\ast}V)_{\eps} \hookrightarrow F^1(P_{\eps\eps}V)$
    and
    \[
    	F^1(P_{\eps\eps}V)/\mathrm{AN}(P_{\ast\ast}V)_{\eps} 
      	= F^1({}^{1}\mathrm{Gr}_{0}(P_{\ast\ast}V))
        = H^1(\mathrm{FM}({}^{1}\mathrm{Gr}_{0}(P_{\ast\ast}V))).
    \]
    The compatibility with the algebraic scattering map is trivial by the definition of the algebraic scattering map.
    
    We consider the short exact sequence
    $0 \to P_{-\eps-\eps}V \to P_{\eps-\eps}V \oplus P_{-\eps\eps}V \to P_{\eps\eps}V \to 0$.
    Since Proposition \ref{Prp:Algebraic Nahm trans} shows $F^0(P_{\eps\eps}V)=0$
    and $F^2(P_{-\eps-\eps}V)=0$,
    we obtain the exact sequence
    \[
    	0 \to F^1(P_{-\eps-\eps}V) \to 
      \mathrm{AN}(P_{\ast\ast}V)_{-\eps} \oplus \mathrm{AN}(P_{\ast\ast}V)_{\eps} \to
      F^1(P_{\eps\eps}V) \to 0.
    \]
    Therefore we have
    \begin{align*}
    	\mathrm{AN}(P_{\ast\ast}V)_{-\eps} \cap \mathrm{AN}(P_{\ast\ast}V)_{\eps} &\simeq F^1(P_{-\eps-\eps}V), \\
    	\mathrm{AN}(P_{\ast\ast}V)_{-\eps} + \mathrm{AN}(P_{\ast\ast}V)_{\eps} &\simeq F^1(P_{\eps\eps}V).
    \end{align*}
  \end{proof}
  
  	Let $(V,h,A)$ be an irreducible $L^2$-finite instanton of rank $r >1$
    and $(\hat{V},\hat{h},\hat{A},\hat{\Phi})$ the Nahm transform of $(V,h,A)$.
    For $\xi \in \mathrm{Sing}(V,h,A)$,
    let $\rho_{\pm,\xi}$ be the representation of $\mathfrak{su}(2)$ defined in Definition \ref{Def:sing set}.
    For the irreducible decomposition $\rho_{\pm,\xi} = \bigoplus_{i=1}^{m_{\pm,\xi}}\rho_{\pm,\xi,i}$,
    we set $w_{\pm,\xi} := (\mathrm{rank}(\rho_{\pm,\xi,i}))$.
    We decompose the weight $\vec{k}=(k_i) \in \mathbb{Z}^{\mathrm{rank}(\hat{V})}$ of $(\hat{V},\hat{h},\hat{A},\hat{\Phi})$ at $\xi \in \mathrm{Sing}(V,h,A)$
    into the positive part $k_{+}$ and the negative part $k_{-}$.
  \begin{Thm}\label{Thm:Correspodence between weights}
  	$k_{\pm}$ agrees with $\pm w_{\pm,\xi}$ under a suitable permutation.
  \end{Thm}
  \begin{proof}
  	By considering $V_{\xi} = (V,h,A_{\xi})$ instead of $(V,h,A)$,
    we assume $\xi=0$.
    By Proposition \ref{Prp:merom comparison} and Proposition \ref{Prp:dirac weight is merom degree},
    $k_{+}$ (resp. $k_{-}$) is determined by the stalk at $0 \in \hat{T}^2$ of $H^1(\mathrm{FM}({}^{2}\mathrm{Gr}_{0}(P_{\ast\ast}V)))$
    (resp. $H^1(\mathrm{FM}({}^{1}\mathrm{Gr}_{0}(P_{\ast\ast}V)))$).
    Applying Proposition \ref{Prp:hol Asymp of Inst} to $(V,h,A)|_{(0,\infty)\times T^3}$ and $(V,h,A)|_{(-\infty,0)\times T^3}$,
    we obtain the model solutions $(\Gamma_{\pm},N_{\pm})$ of the Nahm equation.
    We set $\Gamma_{\pm,\bar{\tau}}d\bar{\tau} + \Gamma_{\pm,\bar{w}}d\bar{w} := (\sum_i\Gamma_{\pm,i}dx^i)^{(0,1)}$
    and $N_{\pm,\bar{\tau}}d\bar{\tau} + N_{\pm,\bar{w}}d\bar{w} := (\sum_iN_{\pm,i}dx^i)^{(0,1)}$.
    We consider the following lemma.
    \begin{Lem}\label{Lem:explicit formula for gr. of model bundle}
    	Let $X_{\pm,0}$ be the eigenspace of $\Gamma_{\pm,\bar{\tau}}$ of eigenvalue $0$.
    	The gradation ${}^{i}\mathrm{Gr}_{0}(P_{\ast\ast}V)\;(i=1,2)$ are isomorphic with
    	$\tilde{E}_{\pm} := \left(X_{\pm,0} \times T^2, \deebar_{T^2} + (\Gamma_{\pm,\bar{w}} + N_{\pm,\bar{w}})|_{X_{\pm,0}}d\bar{w} \right)$ respectively.
    \end{Lem}
    If we admit this lemma,
    then by Corollary \ref{Cor:FM-trans of ss deg0},
    the stalk $H^1(\mathrm{FM}(\tilde{E}_{\pm}))_{0}$ is determined by the size of Jordan blocks of $(\Gamma_{\pm,\bar{w}} + N_{\pm,\bar{w}})|_{X_{\pm,0}}$
    whose eigenvalues are $0$.
    Since the size of Jordan blocks of $(\Gamma_{\pm,\bar{w}} + N_{\pm,\bar{w}})|_{X_{\pm,0}}$ whose eigenvalues are $0$ is the rank of irreducible representations of $\mathfrak{su}(2)$ contained in $\rho_{\pm,0}$,
    hence $k_{\pm}$ agrees with $\pm w_{\pm,\xi}$ with a suitable permutation.
  \end{proof}
  \begin{proof}[{\upshape \textbf{(proof of Lemma \ref{Lem:explicit formula for gr. of model bundle})}}]
    We may assume that any $\Gamma_i$ are diagonal.
  	As in subsection \ref{SubSec:Prop of Gr},
    we take $R>0$ and set the holomorphic Hermitian vector bundle with the endomorphism $({E}'_{\pm},h_{E'_{\pm}},\deebar_{{E}'_{\pm}},f'_{\pm})$
    on $\Delta^{\ast}(R):=\{z\in\mathbb{C}\mid 0<|z|<\exp(-2\pi R)\}$ as follows:
    \begin{eqnarray*}
    	\left\{\begin{array}{ll}
    		\deebar_{E'_{\pm}}(\mb{e'}_{\pm}) &= \mb{e'}\left(\Gamma_{\pm,\bar{\tau}} + ((2\pi)^{-1}\log|z|)^{-1}N_{\pm,\bar{\tau}}\right)d\bar{z}/(2\pi\bar{z}) \\
    		h_{E'_{\pm}}(e'_{\pm,i}, e'_{\pm,j}) &= \delta_{ij} \\
    		f'_{\pm}(\mb{e'}_{\pm}) &= \mb{e'}_{\pm}\left(\Gamma_{\pm,\bar{w}} + ((2\pi)^{-1}\log|z|)^{-1}N_{\pm,\bar{w}}\right),
    	\end{array}\right.
    \end{eqnarray*}
    where $\mb{e'}_{\pm} = (e'_{\pm,i})$ is a $C^{\infty}$-frame of $E'_{\pm}$ on $\Delta(R)^{\ast}$.
    We take a holomorphic frame $\mb{b}'_{\pm}=(b'_{\pm,i})$ of ${E}'_{\pm}$
    by $\mb{b}'_{\pm} := \mb{e'}_{\pm}\exp(-\Gamma_{\pm,\bar{\tau}}\bar{\tau}- 2N_{\pm,\bar{\tau}}\log(t))
    = \mb{e'}_{\pm}\exp( -(2\pi)^{-1}\Gamma_{\pm,\bar{\tau}}\log(\bar{z})- 2N_{\pm,\bar{\tau}}\log((2\pi)^{-1}\log|z|) )$.
    Then, we also set the holomorphic frame $\mb{b''}_{\pm}=(b''_{\pm,i})$ of $P_{0}E'$ by
    $b''_{\pm,i} := z^{- \lceil \mathrm{ord}(b'_{\pm,i}) \rceil}b'_{\pm,i}$,
    where $\lceil \alpha \rceil$ is the least integer satisfying $\lceil \alpha \rceil \geq \alpha$.
    Hence we have
    $
    	P_{0}f'(\mb{b''}_{\pm})
      =	\mb{b''}_{\pm}
      	\exp(\Gamma_{\pm,\bar{\tau}}\bar{\tau}+ 2N_{\pm,\bar{\tau}}\log(t))
        (\Gamma_{\pm,\bar{w}} + (\log|z|)^{-1}N_{\pm,\bar{w}})
        \exp(-\Gamma_{\pm,\bar{\tau}}\bar{\tau}-2N_{\pm,\bar{\tau}}\log(t))
      =	\mb{b''}_{\pm}(\Gamma_{\pm,\bar{w}} + N_{\pm,\bar{w}})
    $.
    By Remark \ref{Remark:some attentions} \ref{Item:fundamental domain in Remark:some attentions},
    ${}^{i}\mathrm{Gr}_{0}(P_{\ast\ast}V)\;(i=1,2)$ is spanned by the subset of $\mb{b''}_{\pm}$
    that correspond to eigenvectors of $\Gamma_{\pm,1}$ of eigenvalue $0$.
  	By Corollary \ref{Cor:Isom to model bundle of Gr},
    ${}^{i}\mathrm{Gr}_{0}(P_{\ast\ast}V)\;(i=1,2)$ is isomorphic to
    $
    	\tilde{E}'_{\pm} 
      := \left(\mathrm{Gr}_{0}(P_{\ast}E'_{\pm}) \times T^2, \deebar_{T^2} + \mathrm{Gr}_{0}(P_{\ast}f'_{\pm})d\bar{w} \right)
      \simeq \left(X_{\pm,0} \times T^2, \deebar_{T^2} + (\Gamma_{\pm,\bar{w}} + N_{\pm,\bar{w}})|_{X_{\pm,0}}d\bar{w} \right)
    $ respectively.
  \end{proof}
	

\begin{thebibliography}{99}
  	\bibitem{Ref:Ati}
    	Michael Atiyah,
      ``Vector bundles over an elliptic curve'',
      Proceedings of the London Mathematical Society 7(1957), 414--452
  	\bibitem{Ref:ADHM}
    	Michael Atiyah, Vladimir Drinfeld, Nigel Hitchin and Yuri Manin,
      ``Construction of instantons'',
      Physics Letters A 65 (1978), no. 3, 185--187
		\bibitem{Ref:Biq}
	 		Oliver Biquard,
	    ``Sur les \'{e}quations de Nahm et la structure de Poisson des alg\'{e}bres de Lie semi-simples complexes.'',
	    Mathematische Annalen 304 (1996), no. 2, 253--276
    \bibitem{Ref:Bis-Hur}
    	Indranil Biswas and Jacques Hurtubise,
      ``Monopoles on Sasakian Three-folds'',
      Communications in Mathematical Physics 339 (2015),
      no. 3, 1083--1100
		\bibitem{Ref:Cha1}
			Benoit Charbonneau,
	    ``Analytic aspects of Periodic Instantons'',
	    Ph.D Theis, Massachusetts Institute of Technology (2004)
	  \bibitem{Ref:Cha2}
	  	Benoit Charbonneau,
	    ``From spatially periodic instantons to singular monopoles'',
	    Communications in Analysis and Geometry 14 (2006), no. 1, 183--214
	  \bibitem{Ref:Cha-Hur}
	  	Benoit Charbonneau and Jacques Hurtubise,
	    ``Singular Hermitian-Einstein Monopoles on the Product of a Circle and a Riemann Surface'',
	    International mathematics research notices (2011), no. 1, 175--216.
    \bibitem{Ref:Cha-Hur2}
    	Benoit Charbonneau and Jacques Hurtubise,
	    ``Spatially periodic instantons: Nahm transform and moduli'',
	    arxiv : 1712.05075
    \bibitem{Ref:Che-Kap}
    	Sergey A Cherkis and Anton Kapustin,
      ``Periodic monopoles with singularities and N= 2 super-QCD'',
      Communications in mathematical physics 234 (2003), no. 1, 1--35
	  \bibitem{Ref:Dna-Kro}
	  	Simon Donaldson and Peter Kronheimer,
	  	``The geometry of four-manifolds'',
	    Oxford University Press (1990)
    \bibitem{Ref:Gro-Law}
    	Mikhael Gromov and H. Blaine Lawson,
      ``Positive scalar curvature and the Dirac operator on complete Riemannian manifolds'',
      Publications Math{\'e}matiques de l'IH{\'E}S 58 (1983), no. 1, 83--196.
    \bibitem{Ref:Hit}
    	Nigel J Hitchin,
      ``On the construction of monopoles'',
      Communications in Mathematical Physics 89 (1983), no. 2, 145--190.
    \bibitem{Ref:Jar}
    	Marcos Jardim,
      ``A survey on Nahm transform'',
      Journal of Geometry and Physics  52 (2004), no. 3, 313--327.
	  \bibitem{Ref:Kro}
	  	Peter Kronheimer,
	    ``Monopoles and Taub-NUT metrics'',
	    Master Thesis, University of Oxford (1985)
    \bibitem{Ref:Law-Mic}
    	H. Blaine Lawson and Marie-Louise Michelsohn,
      ``Spin Geometry'',
      Princeton University Press (1989)
		\bibitem{Ref:Moc1}
	  	Takuro Mochizuki,
	    ``Asymptotic behaviour and the Nahm transform of doubly periodic instantons with square integrable curvature'',
	    Geometry \& Topology 18 (2014), no. 5, 2823--2949.
	  \bibitem{Ref:Moc2}
	  	Takuro Mochizuki,
	    ``Notes on periodic monopoles and Nahm transforms'',
	    Preprint
	  \bibitem{Ref:Moc3}
	  	Takuro Mochizuki,
	    ``Wild harmonic bundles and wild pure D-modules'',
	    Ast\`{e}risque 340, Soci\'{e}t\'{e} Math\'{e}matique de France, Paris (2011)
    \bibitem{Ref:Moc-Yos}
    	Takuro Mochizuki and Masaki Yoshino,
      ``Some characterizations of Dirac type singularity of monopoles'',
      Communications in Mathematical Physics 356 (2017), no. 2, 613--625.
	  \bibitem{Ref:Mor-Mro-Rub}
	  	John W.Morgan, Tomasz Mrowka and Daniel Ruberman,
	    ``The $L^2$-Moduli Space and a Vanishing Theorem for Donaldson Polynomial Invariants'',
	    Monographs in Geometry and Topology, 2, International Press (1994)
    \bibitem{Ref:Nah}
    	Werner Nahm,
      ``The construction of all self-dual multimonopoles by the ADHM method'',
       Monopoles in quantum field theory (1982), 87--94.
    \bibitem{Ref:Nak}
    	Hiraku Nakajima,
      ``Monopoles and Nahm's equations'',
      Lecture Notes in Pure and Appl. Math, 145, (1993)
	  \bibitem{Ref:Sim}
    	Carlos T Simpson,
      ``Constructing variations of Hodge structure using Yang-Mills theory and applications to uniformization'',
      Journal of the American Mathematical Society 1 (1988), no. 4, 867--918
    \bibitem{Ref:Sim2}
    	Carlos T Simpson,
      ``Harmonic bundles on noncompact curves'',
      Journal of the American Mathematical Society 3 (1990), no. 3, 713--770
    \bibitem{Ref:Siu}
    	Yum Tong Siu,
      ``Techniques of extension of analytic objects'',
      Lecture Notes in Pure and Applied Mathematics, 8, Marcel Dekker (1974)
    \bibitem{Ref:Uhl-Yau}
    	Karen Uhlenbeck and Shing-Tung Yau,
      ``On the existence of Hermitian-Yang-Mills connections in stable vector bundles.'',
      Communications on Pure and Applied Mathematics 39.S1 (1986).
    \bibitem{Ref:Zuc}
    	Steven Zucker,
      ``Hodge theory with degenerating coefficients: $L^2$ cohomology in the Poincar{\'e} metric'',
      Annals of Mathematics 109 (1979), no. 3, 415--476
	\end{thebibliography}
\end{document}